\documentclass[12pt]{amsart}
\usepackage{a4wide,mathrsfs,hyperref,stmaryrd,tikz}

\theoremstyle{plain}
\newtheorem{Thm}{Theorem}
\newtheorem{Lem}{Lemma}[section]
\newtheorem{Prop}[Lem]{Proposition}
\newtheorem{Cor}[Lem]{Corollary}
\theoremstyle{definition}

\theoremstyle{remark}
\newtheorem{Rmk}[Lem]{Remark}
\newtheorem{Ex}[Lem]{Example}
\numberwithin{equation}{section}

\DeclareMathOperator\sgn{sgn}

\title{Natural extensions and entropy of $\alpha$-continued fractions}
\author{Cor Kraaikamp}
\address{Technische Universiteit Delft and Thomas Stieltjes Institute of Mathematics\\ EWI\\ Mekelweg 4\\ 2628 CD Delft, the Netherlands}
\email{c.kraaikamp@tudelft.nl}
\author{Thomas A. Schmidt}
\address{Oregon State University\\Corvallis, OR 97331, USA}
\email{toms@math.orst.edu}
\author{Wolfgang Steiner}
\address{LIAFA, CNRS UMR 7089, Universit\'e Paris Diderot -- Paris 7,
Case 7014, 75205 Paris Cedex 13, FRANCE}
\email{steiner@liafa.univ-paris-diderot.fr}
\thanks{This work has been supported by the Bezoekersbeurs B~040.11.083 of the Nederlandse Organisatie voor Wetenschappelijk Onderzoek (NWO), and the Hausdorff Research Institute for Mathematics.}

\begin{document}
\begin{abstract}
We construct a natural extension for each of Nakada's $\alpha$-continued fraction transformations and show the continuity as a function of $\alpha$ of both the entropy and the measure of the natural extension domain with respect to the density function $(1+xy)^{-2}$.
For $0 < \alpha \le 1$, we show that the product of the entropy with the measure of the domain equals $\pi^2/6$.    We show that the interval $(3-\sqrt{5})/2 \le \alpha \le (1+\sqrt{5})/2$ is a maximal interval upon which the entropy is constant. 
As a key step for all this, we give the explicit relationship between the $\alpha$-expansion of $\alpha-1$ and of $\alpha$.
\end{abstract}

\maketitle

\section{Introduction}

Shortly after the introduction at the end of the 1950s of the idea of Kolmogorov--Sinai  entropy, hereafter simply \emph{entropy}, Rohlin \cite{Rohlin61} defined the notion of natural extension of a dynamical system and showed that a system and its natural extension have the same entropy.
In briefest terms, a natural extension is a minimal invertible dynamical system of which the original system is a factor under a surjective map; 
natural extensions are unique up to metric isomorphism.  

In 1977, Nakada, Ito and Tanaka \cite{NakadaItoTanaka77} gave an explicit planar map fibering over the regular continued fraction map of the unit interval. 
Their planar map is so straightforward that it has an obvious invariant measure, and from this they  gave a natural manner to derive the invariant measure for the continued fraction map.  
(See \cite{Keane:95} for discussion of the possible historical implications.)
In particular, they showed that their planar system is a natural extension of the regular continued fraction system with its Gauss measure. 

In 1981,  Nakada \cite{Nakada81} introduced his \emph{$\alpha$-continued fractions}, which form a one dimensional family of interval maps, $T_{\alpha}$  with  $\alpha \in [0,1]$.   (In fact,  $T_1$ is the Gauss continued fraction map,  and $T_{1/2}$ is the nearest-integer continued fraction map.)   Using planar natural extensions, he gave the entropy for those maps corresponding to $\alpha \in [1/2, 1]$.     In 1991,     Kraaikamp \cite{Kraaikamp:91} gave a more direct calculation of these entropy values by using his $S$-expansions,  based upon inducing past subsets of the planar natural extension of the regular continued fraction map given in \cite{NakadaItoTanaka77}.     

It was not until 1999 that further progress was made on the entropy of the $\alpha$-continued fractions.   Moussa, Cassa and Marmi \cite{MoussaCassaMarmi:99} gave the entropy for the maps with  $\alpha \in [\sqrt{2}-1, 1/2)$.   
Let $h(T_{\alpha})$ denote the entropy of~$T_{\alpha}$, and let $g = (\sqrt{5}-1)/2$ be the golden mean;  with their results, one knew  
\[ 
h(T_{\alpha}) =  \left\{\begin{array}{cl}\dfrac{\pi^2}{6 \ln(1+\alpha)} & \mbox{for}\ g \le \alpha \le 1\,;\\[2ex] \dfrac{\pi^2}{6 \ln(1+g)} & \mbox{for}\ \sqrt{2}-1 \le  \alpha \le g\,.\end{array}\right.      
\]

In 2008,   Luzzi and Marmi  \cite{LuzziMarmi08}  presented numeric data showing that the entropy function $\alpha \mapsto h(T_{\alpha})$ behaves in a rather complicated fashion as $\alpha$ varies.  They also claimed that  $\alpha \mapsto h(T_{\alpha})$ is a continuous function of $\alpha$ whose limit at $\alpha = 0$ is zero.    Unfortunately,  their proof of continuity was flawed;  however,  Tiozzo \cite{Tiozzo:09}  has since salvaged the result for $\alpha > 0.056\dots\,$ (and, in an updated version, after our work was completed, has shown H\"older continuity throughout the full interval). 
Luzzi and Marmi also conjectured that, for non-zero~$\alpha$, the product of the entropy and the area of the standard number theoretic planar extension for~$T_{\alpha}$ is constant.
 
Also in 2008, prompted by the numeric data of \cite{LuzziMarmi08},  Nakada and Natsui ~\cite{NakadaNatsui08} gave explicit intervals on which $\alpha \mapsto h(T_{\alpha})$ is respectively constant, increasing, decreasing.   
Indeed, they showed this by exhibiting intervals of $\alpha$ such that  $T^{k}_{\alpha}(\alpha) = T^{k'}_{\alpha}({\alpha-1})$ for pairs of positive integers $(k,k')$ and showed that the entropy is constant (resp.\ increasing, decreasing) on such an interval if $k=k'$ (resp.\ $k>k'$, $k<k'$).     
They   conjectured that there is an open dense set of $\alpha \in [0,1]$ for which the $T_{\alpha}$-orbits of ${\alpha-1}$ and $\alpha$ synchronize.    (Carminati and Tiozzo ~\cite{Carminati-Tiozzo:10}  confirm this conjecture  and also identify maximal intervals where $T_{\alpha}$-orbits synchronize.)  

We prove the continuity of the entropy function and 
confirm the conjectures of Luzzi--Marmi and of Nakada--Natsui (including reproving results of \cite{Carminati-Tiozzo:10}).  Our main results are stated more precisely in Section~\ref{s:justTheFacts}.

\subsection*{Our approach}      
Our results follow from giving an explicit description of a planar natural extension for each $\alpha \in (0,1]$, see Section~\ref{sec:struct-natur-extens}, and this by way of giving details of the relationship between the $\alpha$-expansions of ${\alpha-1}$ and~$\alpha$; see Theorem~\ref{t:endpoints}.   

 Experimental evidence, and experience with $S$-expansions \cite{Kraaikamp:91},  ``quilting'' \cite{Kraaikamp-Schmidt-Smeets:10} and with analogous natural extensions for $\beta$-expansions \cite{Kalle-Steiner:12},  leads one to expect that the planar natural extension for $T_{\alpha}$ has fibers over the interval that are constant between points in the union of the $T_{\alpha}$ orbits of $\alpha-1$ and $\alpha$;  see e.g.\ Figure~\ref{f:2342}.     Thus, one is quickly interested in finding ``synchronizing intervals'' for which all $\alpha$ have orbits that meet after the same number of respective steps, and share initial expansions of $\alpha$ and $\alpha-1$.     This is easily expressed in terms of matrix actions,  and one can gain some geometric intuition; see Figure~\ref{f:rich}  and Remark~\ref{r:RvLv}.   From this perspective,  the fundamental relationship is 
expressed by~\eqref{e:funMatrix}.   Furthermore,  it is easy to discover the ``folding operation'' on these synchronizing intervals,  see Remark~ \ref{r:fold}.    
 
However,   the matrix methods by themselves are awkward when it is necessary to characterize the values $\alpha$ for which there is synchronization.    We do this in Theorem~\ref{t:endpoints}, using our \emph{characteristic sequences}.   Furthermore,  and crucially,  a detailed description of the natural extensions in general is too fraught with details without the use of formal language notation and vocabulary.   Mainly,  this is because of the fractal nature of pieces of these planar regions; see Figures~\ref{f:g2},~\ref{f:4},~\ref{f:23} and~\ref{f:2342} for hints of this phenomenon.  (See also Theorem~\ref{t:shapeOmega} for a statement giving the shape of a natural extension with our vocabulary.)     Thus,  we express the $\alpha$-expansions as words over an appropriate alphabet, and build up notation to represent the basic operations relating the expansion of $\alpha$ and~$\alpha-1$.      Further details on our approach are given in the outline below.    

\subsection*{Outline} 
The sections of the paper are increasingly technical, with the exception of the final two sections. 
We establish notation that is needed for formulating the results in the following section, including some operations on words and the definition of our characteristic sequences. 
We state a collection of our main results in Section~\ref{s:justTheFacts}.  

Thereafter, we first relate the regular continued fraction and the general $\alpha$-expansion of a real number.  
This then allows a proof that a natural extension for $T_{\alpha}$ is given by our $\mathcal{T}_{\alpha}$ on the closure of the orbits of $(x,0)$.   It also allows us to show the constancy of $h(T_\alpha)\, \mu(\Omega_\alpha)$, thus proving the conjecture of Luzzi and Marmi.

In order to reach the deeper results, in Section~\ref{s:relation} we give the explicit relationship between the $\alpha$ expansions of $\alpha-1$ and of~$\alpha$, which is used to describe the (maximal in an appropriate sense) intervals for synchronizing orbits. 
This is then applied in Section ~\ref{sec:struct-natur-extens} to give a detailed description of the natural extension domain, as the union of fibers that are constant on intervals void of the $T_{\alpha}$-orbits of ${\alpha-1}$ and~$\alpha$.      
In Section ~\ref{sec:evol-omeg-along}, we describe how the natural extensions deform along a synchronizing interval,  and derive the behavior of the entropy function along such an interval.  
     
Relying on the previous two sections, in Section ~\ref{s:continuity} we prove the main result of continuity.    In the following section,  we show the more challenging result that the entropy (and hence the measure of the natural extensions) is constant on the interval $[g^2,g]$.  (We give results along the way that show that this is a maximal interval with this property.)

In Section~\ref{sec:limit-points},  we give further results on the set of synchronizing orbits,  in particular showing the transcendence of limits under a natural folding operation on the set of intervals of synchronizing orbits. 
We end this paper with a list of remaining open questions.

\section{Basic Notions and Notation} \label{sec:basic-noti-notat}
\subsection*{One dimensional maps, digit sequences}
For $\alpha \in [0,1]$, we let $\mathbb{I}_\alpha := [{\alpha-1}, \alpha]$ and define the map $T_{\alpha}:\, \mathbb{I}_\alpha \to [{\alpha-1}, \alpha)$ by
\[
T_{\alpha}(x) := \bigg| \frac{1}{x} \bigg| -  \bigg\lfloor\,  \bigg| \frac{1}{x} \bigg| + 1 -\alpha \bigg\rfloor \qquad (x \neq 0),
\]
$T_{\alpha}(0) :=0$.
For $x \in \mathbb{I}_\alpha$, put
\[
\varepsilon(x) := \left\{\begin{array}{cl}+1 & \mbox{if}\ x \ge 0\,, \\ -1 & \mbox{if}\ x<0\,,\end{array}\right. \quad \mbox{and} \quad d_{\alpha}(x) := \bigg\lfloor \bigg| \frac{1}{x} \bigg| + 1 - \alpha \bigg\rfloor,
\]
with $d_\alpha(0) = \infty$.   
Furthermore, let
\[
\varepsilon_n = \varepsilon_{\alpha,n}(x) := \varepsilon(T^{n-1}_\alpha (x)) \quad \mbox{and}\quad d_n = d_{\alpha,n}(x) := d_\alpha(T^{n-1}_\alpha (x)) \qquad (n \ge 1).
\]
This yields the \emph{$\alpha$-continued fraction} expansion of $x \in \mathbb{R}$\,:
\[
x = d_0+ \dfrac{\varepsilon_1}{d_1 + \dfrac{\varepsilon_2}{d_2+\cdots}}\,, 
\]
where $d_0 \in \mathbb{Z}$ is such that $x-d_0 \in [{\alpha-1}, \alpha)$.
(Standard convergence arguments justify equality of $x$ and its expansion.) 
These $\alpha$-continued fractions include the {\em regular continued fractions} {\em (RCF)}, given by $\alpha = 1$, and the nearest integer continued fractions, given by $\alpha = 1/2$.
We will often use the {\em by-excess continued fractions}, given by $\alpha = 0$.
The map $T_{0}$ gives infinite expansions for all $x \in [-1, 0)$; each expansion has all signs $\varepsilon_n = -1$, and digits $d_n \ge 2$.   A number in this range is rational if and only if it has an eventually periodic expansion of period $(\varepsilon:d) = (-1:2)$; in particular,  $-1$ has the purely periodic expansion with this period. 

The point $\alpha$ is included in the domain of $T_\alpha$ because its $T_\alpha$-orbit plays an important role, as does that of~${\alpha-1}$. 
We thus define
\[
\underline{b}{}^\alpha_n = (\varepsilon_{\alpha,n}({\alpha-1}):d_{\alpha,n}({\alpha-1})) \quad \text{and} \quad \overline{b}{}^\alpha_n  = (\varepsilon_{\alpha,n}(\alpha):d_{\alpha,n}(\alpha)) \qquad (n \ge 1),
\]
and informally refer to these sequences as the $\alpha$-expansions of ${\alpha-1}$ and~$\alpha$.
Setting
\[
\llbracket ({\varepsilon_1:d_1}) ({\varepsilon_2:d_2}) \cdots\,\rrbracket := \dfrac{\varepsilon_1}{d_1 + \dfrac{\varepsilon_2}{d_2+\cdots}}\,, 
\]
gives equalities such as $\llbracket\, \underline{b}{}^\alpha_1 \underline{b}{}^\alpha_2 \cdots \,\rrbracket = {\alpha-1}$ and $\llbracket\, \overline{b}{}^\alpha_1 \overline{b}{}^\alpha_2 \cdots \,\rrbracket = \alpha$. 
We also set
\[
\llbracket (\varepsilon_1:d_1) \cdots (\varepsilon_n:d_n),\, y\, \rrbracket := \dfrac{\varepsilon_1}{d_1 + \cdots + \dfrac{\varepsilon_n}{d_n + y}} \qquad (y \in \mathbb{R}).
\]

Since $d_{\alpha}(x)\ge 1$ for all $x \in \mathbb{I}_\alpha$, $\alpha \in [0,1]$, and $d_{\alpha}(x)\ge 2$ when $\varepsilon(x)=-1$, let
\[
\mathscr{A}_0 := \mathscr{A} \cup \{(+1:\infty)\} \ \mbox{where}\  \mathscr{A} := \mathscr{A}_- \cup \mathscr{A}_+\;,\]
with
\begin{equation}\label{e:defAminPlus}  \mathscr{A}_- := \{\,(-1:d) \mid d \in \mathbb{Z},\, d \ge 2\}\ \mbox{and}\ \mathscr{A}_+ := \{\,(+1:d) \mid d \in \mathbb{Z},\, d \ge 1\}\,.
\end{equation}

Every ``digit'' $(\varepsilon(x):d_\alpha(x))$ is thus in $\mathscr{A}_0$.
We define an order $\preceq$ on $\mathscr{A}_0$ by
\[
(\varepsilon:d) \preceq (\varepsilon':d') \quad \mbox{if and only if} \quad \varepsilon/d \le \varepsilon'/d'\,.
\]
For any $x, x' \in \mathbb{I}_\alpha$, $\alpha \in [0,1]$, $x \le x'$ implies $(\varepsilon(x):d_\alpha(x)) \preceq (\varepsilon(x'):d_\alpha(x'))$.

The interval $\mathbb{I}_\alpha\setminus \{0\}$ is partitioned by the rank-one \emph{cylinders} of~$T_{\alpha}$, which are defined by 
\[ 
\Delta_{\alpha}(a) := \{x \in \mathbb{I}_\alpha \mid  (\varepsilon(x):d_{\alpha}(x)) = a\} \qquad (a \in \mathscr{A}_0)\,.
\]
All cylinders $\Delta_{\alpha}(a)$ with $a \in \mathscr{A}$, $\underline{b}{}^\alpha_1 \prec a \prec \overline{b}{}^\alpha_1$, are \emph{full}, that is their image under $T_{\alpha}$ is the interval $[{\alpha-1}, \alpha)$, and
\[
T_\alpha\big(\Delta_{\alpha}(\underline{b}{}^\alpha_1)\big) = \big[T_{\alpha}({\alpha-1}), \alpha\big), \quad T_\alpha\big(\Delta_{\alpha}(\overline{b}{}^\alpha_1)\big) =  \big[T_{\alpha}(\alpha), \alpha\big), \quad T_\alpha\big(\Delta_\alpha(+1:\infty)\big) = \{0\}\,.
\]

\subsection*{Two-dimensional maps, matrix formulation, invariant measure}\label{sss:2dMaps}
The standard number theoretic planar map associated to continued fractions is defined by
\[
\mathcal{T}_\alpha(x,y) := \bigg(T_\alpha(x), \frac{1}{d_{\alpha}(x)+ \varepsilon(x)\,y}\bigg)\, \quad (x \in \mathbb{I}_\alpha,\ y \in [0,1])\,.
\]
For any $x \in \Delta_\alpha(\varepsilon:d)$, $(\varepsilon:d) \in \mathscr{A}$, we have
\begin{equation} \label{e:matsForT}
\mathcal{T}_\alpha(x, y) = \big(M_{(\varepsilon:d)} \cdot x, N_{(\varepsilon:d)} \cdot y\big),
\end{equation}
where
\[
M_{(\varepsilon:d)} := (-1) \begin{pmatrix}-d & \varepsilon \\ 1 & 0 \end{pmatrix} \quad \mbox{and} \quad N_{(\varepsilon:d)} := {}^t\hspace{-.1em}M_{(\varepsilon:d)}^{-1} = (-\varepsilon) \begin{pmatrix}0 & 1 \\ \varepsilon & d\end{pmatrix}\,.
\]
As usual, the $2 \times 2$ matrix {\small$\begin{pmatrix}a & b \\ c & d\end{pmatrix}$} acts on real numbers by {\small$\begin{pmatrix}a & b \\ c & d\end{pmatrix} \cdot x = \dfrac{a x + b}{c x + d}$}, and ${}^t\hspace{-.1em}M$ denotes the transpose of~$M$.
Note that $M \cdot x$ is a projective action, therefore the factors $(-1)$ and $(-\varepsilon)$ do not change the actions of $M_{(\varepsilon:d)}$ and~$N_{(\varepsilon:d)}$.
However, these factors will be useful in several matrix equations.

Let $\mu$ be the measure on $\mathbb{I}_\alpha \times [0,1]$ given by
\[
d\mu = \dfrac{dx\, dy}{(1 + xy)^2}\,.
\]
Then we have, for any rectangle $[x_1,x_2] \times [y_1, y_2] \subset \mathbb{I}_\alpha \times [0,1]$ and any invertible matrix~$M$, 
\begin{equation} \label{e:mu}
\mu\big([x_1,x_2] \times [y_1, y_2]\big) = \log \frac{(1+x_1y_1) (1+x_2y_2)}{(1+x_1y_2) (1+x_2y_1)} = \mu\big(M \cdot [x_1,x_2] \times {}^t\hspace{-.1em}M^{-1} \cdot [y_1, y_2]\big).
\end{equation}

\subsection*{Words, symbolic notation} \label{sec:symb-notat-digit}
For any set~$V$, the Kleene star $V^* = \bigcup_{n\ge0} V^n$ denotes the set of concatenations of a finite number of elements in~$V$, and $V^\omega$~denotes the set of (right) infinite concatenations of elements in~$V$. 
The length of a finite word $v$ is denoted by~$|v|$, that is $|v| = n$ if $v \in V^n$.
For the Kleene star of a single word (or letter)~$v$, we write $v^*$ instead of~$\{v\}^*$, and $v^\omega$ denotes the unique element of $\{v\}^\omega$. 
We will also use the abbreviations $v_{[m,n]} = v_m v_{m+1} \cdots v_n$, $v_{[m,n)} = v_m v_{m+1} \cdots v_{n-1}$, where $v_{[m,m-1]} = v_{[m,m)}$ is the empty word, and $v_{[m,\infty)} = v_m v_{m+1} \cdots$. 

In light of~\eqref{e:matsForT}, we set, for $v = v_1 \cdots v_n \in \mathscr{A}^*$, 
\[
M_v := M_{v_n} \cdots M_{v_1} \quad \mbox{and} \quad N_v := {}^t\hspace{-.1em}M_v^{-1} = N_{v_n} \cdots N_{v_1}\,.
\] 
Then we have, for example, $M_{\underline{b}{}^\alpha_{[1,n]}} \cdot ({\alpha-1}) = T_{\alpha}^{n}({\alpha-1})$ and $M_{\overline{b}{}^\alpha_{[1,n]}} \cdot \alpha = T_{\alpha}^{n}(\alpha)$.

\subsection*{Operations on words via matrices}\label{ss:EW}
The two matrices 
\[
W :=  \begin{pmatrix}1 & 0 \\ -1 & -1\end{pmatrix} \quad \mbox{and} \quad
E := \begin{pmatrix}1 & -1 \\ 0 & 1\end{pmatrix}
\]
arise naturally in our discussion.
Note that $W^2$ is the identity, and also that
\begin{equation} \label{e:W}
M_{(\varepsilon:d)} W = M_{(-\varepsilon:d+\varepsilon)}\,.
\end{equation}
The action of $E$ is $E \cdot x = x-1$, and
\begin{equation} \label{e:E}
E^{\pm1} M_{(\varepsilon:d)} = M_{(\varepsilon:d\pm1)}\,.
\end{equation}
Therefore, let the left superscript $(W)$ and right superscripts $(+1), (-1)$ denote operators, related to $W$ and $E^{\pm1}$ respectively, acting on letters in $\mathscr{A}_0$ by
\[
{}^{(W)} (\varepsilon:d) := \left\{\begin{array}{cl}(-\varepsilon:d+\varepsilon) & \mbox{if}\ d < \infty\,, \\ (+1:\infty) & \mbox{if}\ d = \infty\,,\end{array}\right. \quad
(\varepsilon:d)^{(\pm1)} := \left\{\begin{array}{cl}(\varepsilon:d\pm1) & \mbox{if}\ d < \infty\,, \\ (+1:\infty) & \mbox{if}\ d = \infty\,.\end{array}\right.\quad
\]
We extend this definition to words $v = v_{[1,n]} \in \mathscr{A}_0^*$, $n \ge 2$, by setting ${}^{(W)}v := {}^{(W)}v_1 v_{[2,n]}$ and $v{}^{(\pm 1)} := v_{[1,n)} v_n{}^{(\pm 1)}$.
Similarly, we set ${}^{(W)}v := {}^{(W)}v_1 v_{[2,\infty)}$ for $v = v_{[1,\infty)} \in \mathscr{A}_0^\omega$.

\subsection*{Characteristic sequences, alternating order, operation $v \mapsto \widehat{v}$}
To every finite or infinite word on the alphabet~$\mathscr{A}_-$, we associate a  (correspondingly finite or infinite) 
\emph{characteristic sequence} of positive integers (and~$\infty$) in the following way.
\begin{itemize}
\item
The characteristic sequence of $v \in \mathscr{A}_-^*$ is $a_1 a_2 \cdots a_{2\ell+1}$, where the integers $\ell \ge 0$ and $a_j \ge 1$, $1 \le j \le 2\ell+1$, are defined by
\[
v = (-1:2)^{a_1-1}\, (-1:2+a_2)\, (-1:2)^{a_3-1}\, \cdots\, (-1:2+a_{2\ell})\, (-1:2)^{a_{2\ell+1}-1}\,.
\]
\item
The characteristic sequence of $v \in \mathscr{A}_-^\omega$ that does not end with the infinite periodic word $(-1:2)^\omega$ is $a_1 a_2 \cdots$, where the $a_j \ge 1$, $j \ge 1$, are the unique positive integers such that
\[
v = (-1:2)^{a_1-1}\, (-1:2+a_2)\, (-1:2)^{a_3-2}\, (-1:2+a_4)\, \cdots\,.
\]
\item
The characteristic sequence of $v \in \mathscr{A}_-^* (-1:2)^\omega$ is $a_1 a_2 \cdots$ with $a_j = \infty$ for all $j > 2\ell$, where the integers $\ell \ge 0$ and $a_j \ge 1$, $1 \le j \le 2\ell$, are defined by
\[
v = (-1:2)^{a_1-1}\, (-1:2+a_2)\, (-1:2)^{c_2}\, \cdots\, (-1:2+a_{2\ell})\, (-1:2)^\omega\,.
\]
\end{itemize}

We compare characteristic sequences using the \emph{alternating (partial) order} on words of integers (and~$\infty$), i.e.,
\[
a_{[1,n)} <_{\mathrm{alt}} a'_{[1,n)} \quad \mbox{if and only if} \quad  a_{[1,j]} = a'_{[1,j]}, \, (-1)^j a_{j+1} < (-1)^j a'_{j+1} \ \mbox{for some}\ 0 \le j < n\,.
\]

Using the characteristic sequences, we introduce an operation on words in $\mathscr{A}_-^* \cup \mathscr{A}_-^\omega \setminus \mathscr{A}_-^* (-1:2)^\omega$ that will allow us to express the relationship between the $\alpha$-expansion of $\alpha-1$ and~$\alpha$.
\begin{itemize}
\item
For $v \in \mathscr{A}_-^*$ with characteristic sequence $a_1 a_2 \cdots a_{2\ell+1}$, we set
\[
\widehat{v} := (-1:2+a_1)\, (-1:2)^{a_2-1}\, (-1:2+a_3)\, \cdots\, (-1:2)^{a_{2 \ell}-1}\, (-1:2+a_{2\ell+1})\,.
\]
\item
For $v \in \mathscr{A}_-^\omega \setminus \mathscr{A}_-^* (-1:2)^\omega$ with characteristic sequence $a_1 a_2 \cdots$, we set
\[
\widehat{v} := (-1:2+a_1)\, (-1:2)^{a_2-1}\, (-1:2+a_3)\, (-1:2)^{a_4-1}\, \cdots\,.
\]
\end{itemize}

The characteristic sequence $a_{[1,\infty)}$ of a number $x \in [-1,0)$ is defined to be the characteristic sequence of its by-excess expansion $(\varepsilon_{0,1}(x):d_{0,1}(x))\, (\varepsilon_{0,2}(x):d_{0,2}(x)) \cdots \in \mathscr{A}_-^\omega$.

\section{Results}\label{s:justTheFacts}
For the ease of the reader, we gather the main results of the paper in this section.  

For any $\alpha \in (0,1]$, the standard \emph{natural extension domain} is
\[
\Omega_\alpha := \overline{\big\{\mathcal{T}_\alpha^n(x,0) \mid x \in [{\alpha-1},\alpha),\, n \ge 0\big\}}\,.
\]
We establish the positivity and finiteness of $\mu(\Omega_{\alpha})$ in Section~\ref{sec:natural-extensions}.
The map $\mathcal{T}_{\alpha}$ is invertible almost everywhere on~$\Omega_{\alpha}$, and it is straightforward to define appropriate dynamical systems such that the system of $T_{\alpha}$ is a factor of the system of~$\mathcal{T}_{\alpha}$, by way of the (obviously surjective) projection to the first coordinate.
These systems also verify the minimality criterion for natural extensions, which yields the following theorem. 
For details, see Section~\ref{sec:natural-extensions}.

\begin{Thm} \label{t:natext}
Let $\alpha \in (0,1]$, $\mu_\alpha$~be the probability measure given by normalizing $\mu$ on~$\Omega_\alpha$, $\nu_\alpha$~the marginal measure obtained by integrating $\mu_\alpha$ over the fibers~$\{x\} \times \{y \mid (x,y) \in \Omega_\alpha\}$, $\mathscr{B}_\alpha$~the Borel $\sigma$-algebra of~$\mathbb{I}_\alpha$, and $\mathscr{B}_\alpha'$ the Borel $\sigma$-algebra of~$\Omega_\alpha$.
Then $(\Omega_\alpha, \mathcal{T}_\alpha, \mathscr{B}_\alpha', \mu_\alpha)$ is a natural extension of $(\mathbb{I}_\alpha, T_\alpha, \mathscr{B}_\alpha, \nu_\alpha)$.
\end{Thm}

In the same section, relying on Abramov's formula for the entropy of an induced system, we prove the following conjecture of Luzzi and Marmi \cite{LuzziMarmi08}.

\begin{Thm} \label{t:hmu}
For any $\alpha \in (0,1]$, we have $h(T_\alpha)\, \mu(\Omega_\alpha) = \pi^2/6$.
\end{Thm}
 
By Theorem~\ref{t:hmu}, all properties of the entropy $h(T_\alpha)$ can be directly derived from the properties of~$\mu(\Omega_\alpha)$. 
Therefore, we consider only $\mu(\Omega_\alpha)$ in the following. 
In particular, the following theorem implies the continuity of $\alpha \mapsto h(T_\alpha)$ on $(0,1]$, which was claimed to be proved in \cite{LuzziMarmi08} (see the introduction of this paper). 
The proof is given in Section~\ref{s:continuity}.
 
\begin{Thm} \label{t:continuous}
The function $\alpha \mapsto \mu(\Omega_\alpha)$ is continuous on $(0,1]$.
\end{Thm}

The following theorem, which is proved in Section~\ref{sec:constancy-entropy-g2}, extends results of \cite{Nakada81,MoussaCassaMarmi:99,Carminati-Marmi-Profeti-Tiozzo:10}.

\begin{Thm}\label{t:g2gConst}
For any $\alpha \in [g^2, g]$, we have $\mu(\Omega_\alpha) = \log(1+g)$.
\end{Thm}

Moreover, we show that $[g^2,g]$ is the maximal interval with this property, and we conjecture that $\mu(\Omega_\alpha) > \log(1+g)$ for all $\alpha \in (0,1] \setminus [g^2,g]$.

The proofs of Theorems~\ref{t:continuous} and \ref{t:g2gConst} heavily rely on understanding how the $\alpha$-expansion of $\alpha$ is related to that of ${\alpha-1}$ and how the evolution of the natural extension depends on this relation.

\smallskip
Theorems~\ref{t:endpoints} and~\ref{t:muOmega} strengthen and clarify results of \cite{NakadaNatsui08}.    The first, proved in Section~\ref{s:relation}, states that \emph{synchronization} of the $T_\alpha$-orbits of $\alpha$ and $\alpha-1$ occurs for $\alpha$ in
\[
\Gamma := \big\{\alpha \in (0,1] \mid T_\alpha^n({\alpha-1}) \ge 0\ \mbox{or}\ T_\alpha^n(\alpha) \ge 0\ \mbox{for some}\ n \ge 1\big\}\,,
\]
and the set of \emph{labels of finitely synchronizing orbits} is
\[
\mathscr{F} := \big\{v \in \mathscr{A}_-^* \mid  a_{[2j,2\ell+1]} <_{\mathrm{alt}} a_{[1,2\ell-2j+2]},\  a_{[2j+1,2\ell+1]} \le_{\mathrm{alt}} a_{[1,2\ell-2j+2]}\ \mbox{for all}\ 1 \le j \le \ell\big\}\,,
\]
where $a_{[1,2\ell+1]} = a_1 a_2 \cdots a_{2\ell+1}$ denotes the characteristic sequence of~$v$.
For $v \in \mathscr{F}$, set
\begin{gather*}
\zeta_v := \llbracket  (v\, \widehat{v}\,)^\omega\rrbracket + 1\,, \quad \chi_v := \left\{\begin{array}{cl}\llbracket  v, 0 {\rrbracket} + 1 & \mbox{if}\ |v| \ge 1, \\[1ex] 1 & \mbox{if}\ |v| = 0,\end{array}\right. \quad \eta_v := \left\{\begin{array}{cl}\llbracket  (v^{(+1)})^\omega \rrbracket + 1 & \mbox{if}\ |v| \ge 1, \\[1ex] 1 & \mbox{if}\ |v| = 0,\end{array}\right. \\[1ex]
\Gamma_v := \left\{\begin{array}{cl}(\zeta_v, \eta_v) & \mbox{if}\ |v| \ge 1\,, \\[1ex] (g,1] & \mbox{if}\ |v| = 0\,.\end{array}\right.
\end{gather*}

\begin{Thm} \label{t:endpoints}
The set $\Gamma$ is the disjoint union of the intervals $\Gamma_v$, $v \in \mathscr{F}$.

For any $\alpha \in \Gamma_v$, $v\in\mathscr{F}$, we have 
\[
\underline{b}{}^\alpha_{[1,|v|\,]} = v, \quad \overline{b}{}^\alpha_{[1,|\widehat{v}|\,]} = {}^{(W)}\widehat{v}{}^{(-1)}, \quad \overline{b}{}^\alpha_{|\widehat{v}|+1}  = {}^{(W)}\underline{b}{}^\alpha_{|v|+1}\,, \quad T_\alpha^{|v|+1}({\alpha-1}) = T_\alpha^{|\widehat{v}|+1}(\alpha)\,.
\]

We have $\alpha \in (0,1] \setminus \Gamma$ if and only if $\alpha \in (0,g]$ and the characteristic sequence $a_{[1,\infty)}$ of $\alpha-1$~satisfies $a_{[n,\infty)} \le_{\mathrm{alt}} a_{[1,\infty)}$ for all $n \ge 2$.

The set $(0,1] \setminus \Gamma$  is a set of zero Lebesgue measure. 
For any $\alpha$ in this set,   $\overline{b}{}^\alpha_{[1,\infty)} = {}^{(W)}\widehat{\underline{b}{}^\alpha_{[1,\infty)}}$.
\end{Thm}

We remark that similar results can be found in \cite{Carminati-Tiozzo:10,Bonanno-Carminati-Isola-Tiozzo}.
There, the description of $\Gamma$ is based on the RCF expansion  of $\alpha$ instead of the characteristic sequence of $\alpha-1$, and the number $\chi_v$ is called pseudocenter of~$\Gamma_v$.
In Section~\ref{sec:relat-betw-alpha}, we show that the characteristic sequence of $\alpha-1$ is essentially the same as the RCF expansion of~$\alpha$. 
In particular, this implies for $\alpha \in (0,1]$ that $\alpha \not\in \Gamma$ is equivalent with $T_1^n(\alpha) \ge \alpha$ for all $n \ge 1$.

The evolution of $\Omega_\alpha$ on a synchronizing interval is described by the following theorem, which is proved in Section~\ref{sec:evol-omeg-along}.

\begin{Thm} \label{t:muOmega}
For any $\alpha \in [\zeta_v, \eta_v]$, $v \in \mathscr{F}$, we have
\[
\mu(\Omega_\alpha) =  \mu(\Omega_{\zeta_v})\, \Big(1 + \big(|\widehat{v}| - |v|\big)\, \nu_{\zeta_v}\big([\zeta_v-1, {\alpha-1}]\big)\Big)\,, 
\]
with $\nu_{\zeta_v}$ as in Theorem~\ref{t:natext}.

On $[\zeta_v, \eta_v]$, the function $\alpha \mapsto \mu(\Omega_\alpha)$ is: constant if $|v| = |\widehat{v}|$; increasing if $|\widehat{v}| > |v|$; decreasing if $|\widehat{v}| < |v|$.
Inverse relations hold for the function $\alpha \mapsto h(T_\alpha)$, {\em cf.} Theorem~\ref{t:hmu}.
\end{Thm}
\bigskip

In order to describe the shape of $\Omega_\alpha$, $\alpha \in (0,1]$, we define
\[
\begin{aligned}
U_{\alpha,1} & := \big\{\underline{b}{}^\alpha_{[1,j]} \mid 0 \le j < k\big\}\,, & 
U_{\alpha,3} & := \big\{\underline{b}{}^\alpha_{[1,j)}\, a \mid 1 \le j < k,\, a \in \mathscr{A}_-,\, \underline{b}{}^\alpha_j \prec a \preceq {}^{(W)}\overline{b}{}^\alpha_1\big\}\,, \\
U_{\alpha,2} & := \big\{\overline{b}{}^\alpha_{[1,j]} \mid 1 \le j < k'\big\}\,, &  
U_{\alpha,4} & := \big\{\overline{b}{}^\alpha_{[1,j)}\, a \mid 2 \le j < k',\, a \in \mathscr{A}_-,\, \overline{b}{}^\alpha_j \prec a \preceq {}^{(W)}\overline{b}{}^\alpha_1\big\}\,,
\end{aligned}
\]
where $k = |v| + 1$, $k' = |\widehat{v}| + 1$ if $\alpha \in \Gamma_v$, $v \in \mathscr{F}$, $k = k' = \infty$ if $\alpha \in (0,1] \setminus \Gamma$.
Let
\begin{gather*}
\begin{aligned} 
\mathscr{L}_\alpha & := (U_{\alpha,3} \cup U_{\alpha,1}\, U_{\alpha,2}^*\, U_{\alpha,4})^*\,, & 
\mathscr{L}'_\alpha & := \mathscr{L}_\alpha\, U_{\alpha,1}\, U_{\alpha,2}^*\,, \\[1ex]
\Psi_\alpha & := \overline{\bigcup_{w\in\mathscr{L}_\alpha} N_w \cdot \big[0, \tfrac{1}{d_\alpha(\alpha)+1}\big]}\,, &
\Psi'_\alpha & := \overline{\bigcup_{w\in\mathscr{L}'_\alpha} N_w \cdot \big[0, \tfrac{1}{d_\alpha(\alpha)+1}\big]}\,,
\end{aligned} \\
\mathscr{C}_\alpha := \big\{\,\Psi_\alpha\,\big\}\, \cup\, \big\{N_{\underline{b}{}^\alpha_{[1,j]}} \cdot \Psi_\alpha \mid 1 \le j< k\big\}\, \cup\, \big\{N_{\overline{b}{}^\alpha_{[1,j]}} \cdot \Psi'_\alpha \mid 1 \le j< k'\big\}\,.
\end{gather*}

\begin{Thm} \label{t:shapeOmega}
Let $\alpha \in (0,1]$ and $k, k'$ as in the preceding paragraph.
Then we have
\begin{equation} \label{e:shapeOmega}
\Omega_\alpha =  \mathbb{I}_\alpha \times \Psi_\alpha\ \cup \overline{\bigcup_{1 \le j < k} \big[T_\alpha^j(\alpha-1), \alpha\big] \times N_{\underline{b}{}^\alpha_{[1,j]}} \cdot \Psi_\alpha}\ \cup \overline{\bigcup_{1 \le j < k'} \big[T_\alpha^j(\alpha), \alpha\big] \times N_{\overline{b}{}^\alpha_{[1,j)}} \cdot \Psi'_\alpha}\,.
\end{equation}
If $T_{\alpha}^j({\alpha-1}) \not\in (x,x')$ for all $0 \le j < k$ and $T_{\alpha}^j(\alpha) \not\in (x,x')$ for all $0 \le j < k'$, then the density of the invariant measure $\nu_\alpha$ defined in Theorem~\ref{t:natext} is continuous on $(x,x')$.

For any $Y \in \mathscr{C}_\alpha$, the Lebesgue measure of $Y \cap\, \overline{\bigcup_{Y' \in \mathscr{C}_\alpha \setminus \{Y\}} Y'}$ is zero, and  
\[
\overline{\bigcup_{Y \in \mathscr{C}_\alpha} Y} = \Psi'_\alpha = {}^t\hspace{-.1em}E \cdot \Psi_\alpha\,.
\]

For any $w \in \mathscr{L}'_\alpha$, we have $N_w \cdot \big(0, \tfrac{1}{d_\alpha(\alpha)+1}\big) \cap\, \overline{\bigcup_{w' \in\mathscr{L}'_\alpha \setminus \{w\}} N_{w'} \cdot \big[0, \tfrac{1}{d_\alpha(\alpha)+1}\big]} = \emptyset$.
\end{Thm}

This theorem is proved in Section~\ref{sec:struct-natur-extens}.
We remark that omitting the closure in the definitions of $\Psi_\alpha$ and $\Psi'_\alpha$ and in~\eqref{e:shapeOmega} changes the sets under consideration only by sets of measure zero.
Moreover, Section~\ref{sec:struct-natur-extens} also provides the speed of convergence of approximations of $\Omega_\alpha$ by finitely many rectangles. 
Note that $\big[0, \tfrac{1}{d_\alpha(\alpha)+1}\big] \subseteq \Psi_\alpha$, thus $\mathbb{I}_\alpha \times \big[0, \tfrac{1}{d_\alpha(\alpha)+1}\big] \subset \Omega_\alpha$, and that $\big[0, \tfrac{1}{d_\alpha(\alpha)}\big] = {}^t\hspace{-.1em}E \cdot \big[0, \tfrac{1}{d_\alpha(\alpha)+1}\big] \subseteq \Psi'_\alpha$.
By Proposition~\ref{p:dalpha}, we have $\big[T_\alpha(\alpha), \alpha] \times \big[0, \tfrac{1}{d_\alpha(\alpha)}\big] \subseteq \Omega_\alpha$ for $\alpha \in (0,1] \setminus \Gamma$, and the same can also be shown for $\alpha \in \Gamma$.

Finally, we show in Section~\ref{s:continuity} that to the left of any interval~$\Gamma_v$, $v \in \mathscr{F}$, there exists an interval on which $\mu(\Omega_\alpha)$ is constant.
To this end, we define the ``folding'' operation
\[
\Theta(v) := v\,\widehat{v}^{(-1)} \qquad (v \in \mathscr{A}_-^*)\,.
\]
We will see that $\Theta$ maps $\mathscr{F}$ to itself, and that $(\zeta_{\Theta^n(v)})_{n\ge0}$ is a sequence of rapidly converging quadratic numbers; see also~\cite{Carminati-Marmi-Profeti-Tiozzo:10}.
Therefore, we define
\[
\tau_v := \lim_{n\to\infty} \zeta_{\Theta^n(v)}\,.
\]

\begin{Thm}\label{t:tauTransc}
For any $v \in \mathscr{F}$, we have $\Theta(v) \in \mathscr{F}$ and $\zeta_v = \eta_{\Theta(v)}$. 

For any $\alpha \in [\tau_v, \zeta_v]$, $v \in \mathscr{F}$, we have $\mu(\Omega_\alpha) = \mu(\Omega_{\zeta_v})$.

The limit point $\tau_v$ is a transcendental real number.
\end{Thm}

\section{Relation between $\alpha$-expansions and RCF expansions} \label{sec:relat-betw-alpha}

We start with proving a relation between the characteristic sequence of $\alpha-1$ and the RCF expansion of~$\alpha$.  

\begin{Prop}[{cf.\ \cite[Exercise~3 on p.~131]{Zagier:81}}] \label{p:0toRCF}
Let $\alpha \in (0,1)$, and $a_1 a_2 \cdots$ be the characteristic sequence of $\alpha-1$.
Then
\[
\alpha = \left\{\begin{array}{ll}[0; a_1, a_2, a_3, \ldots\,]  & \mbox{if}\ \alpha \not\in \mathbb{Q}\,, \\[1ex] [0; a_1, a_2, \ldots, a_{2\ell}] & \mbox{if}\ a_{2\ell+1} = \infty\,.\end{array}\right.
\]
\end{Prop}

\begin{proof}
Let $\alpha \in (0,1)$, and $a_{[1,\infty)}$ be the characteristic sequence of $\alpha-1$.
Assume first that $\alpha \in \mathbb{Q}$, i.e., there exists some $\ell \ge 1$ such that $a_{2\ell+1} = \infty$, $1 \le a_j < \infty$ for $1 \le j \le 2\ell$.
Then since $\alpha-1$ is obviously rational, its by-excess expansion is eventually periodic and this period is that of the purely periodic  $-1$, we have 
\begin{align*}
\alpha - 1 & = \llbracket (-1:2)^{a_1-1}\, (-1:2+a_2)\, (-1:2)^{a_3-1}\, \cdots\, (-1:2+a_{2\ell})\, (-1:2)^\omega\rrbracket \\
& = \llbracket (-1:2)^{a_1-1}\, (-1:2+a_2)\, (-1:2)^{a_3-1}\, \cdots\, (-1:2+a_{2\ell}), -1\rrbracket\,,
\end{align*}
thus
\[
M_{(-1:2+a_{2\ell})} \cdots M_{(-1:2)}^{a_3-1} M_{(-1:2+a_2)} M_{(-1:2)}^{a_1-1} \cdot (\alpha-1) = -1 = E \cdot 0\,.
\]
Since ~\eqref{e:W} and~\eqref{e:E} give
\begin{equation} \label{e:Mn2}
M_{(-1:2+n)} = E\, M_{(+1:n)}\, W\,,
\end{equation} 
induction gives 
\begin{equation} \label{e:Mn}
M_{(-1:2)}^{n-1} = \begin{pmatrix}n & n-1 \\ 1-n & 2-n\end{pmatrix} = W\, M_{(+1:n)}\, E^{-1}\,,
\end{equation}
and $\alpha - 1 = E \cdot \alpha$ clearly holds, we obtain that
\begin{align*}
0 & = E^{-1}\, M_{(-1:2+a_{2\ell})}\, M_{(-1:2)}^{a_{2\ell-1}-1} \cdots M_{(-1:2+a_2)}\, M_{(-1:2)}^{a_1-1} \cdot (\alpha-1) \\
& = M_{(+1:a_{2\ell})}\, M_{(+1:a_{2\ell-1})}\, \cdots M_{(+1:a_2)}\, M_{(+1:a_1)} \cdot \alpha\,,
\end{align*}
thus $\alpha = \llbracket (+1:a_1) \cdots (+1:a_{2\ell}), 0\rrbracket = [0; a_1, \ldots, a_{2\ell}]$.

For $\alpha \not\in \mathbb{Q}$, we have 
\[
\alpha - 1= \lim_{\ell\to\infty}\ \llbracket (-1:2)^{a_1-1}\, (-1:2+a_2)\, (-1:2)^{a_3-1}\, \cdots\, (-1:2+a_{2\ell})\, (-1:2)^\omega\rrbracket\,,
\]
thus   $\alpha = \lim_{\ell\to\infty} [0; a_1, \ldots, a_{2\ell}] = [0; a_1, a_2, \ldots\,]$. 
\end{proof}

Proposition~\ref{p:0toRCF} and the ordering of the RCF expansions gives the following corollary.

\begin{Cor} \label{c:altorder}
Let $x, x' \in [-1,0)$ with characteristic sequences $a_{[1,\infty)}, a'_{[1,\infty)}$. 
Then 
\[
x \le x' \quad \mbox{if and only if} \quad a_{[1,\infty)} \ge_{\mathrm{alt}} a'_{[1,\infty)}\,.
\]
\end{Cor}

Now we show how the RCF expansion of $x \in (0,\alpha]$ can be constructed from the $\alpha$-expansion of~$x$.
This is a key argument in the following section.

\begin{Lem} \label{l:rewrite}
Let $\alpha \in (0,1)$.
For any $x \in (0,\alpha]$, the $1$-expansion of~$x$  is obtained from the $\alpha$-expansion of~$x$ by successively replacing all digits in $\mathscr{A}_-$ using the following rules:
\[
\hspace{-.2em}
\begin{array}{lll}
(+1:d)\, (-1:2)^{n-1}\, (-1:d') & \hspace{-.5em} \mapsto (+1:d-1)\, (+1:n)\, (+1:d'-1)\,,  & \hspace{-.25em} d \ge 2,\, n \ge 1,\, d' \ge 3, \\[1ex]
(+1:d)\, (-1:2)^n\, (+1:d') & \hspace{-.5em} \mapsto (+1:d-1)\, (+1:n)\, (+1:1)\, (+1:d')\,, \hspace{-3em} & \\
& & \hspace{-2.6em} d \ge 2,\, n \ge 1,\, 1 \le d' < \infty, \\[1ex]
(+1:d)\, (-1:2)^{n-1}\, (+1:\infty) & \hspace{-.5em} \mapsto (+1:d-1)\, (+1:n)\, (+1:\infty)\,, & \hspace{-.25em} d \ge 2,\, n \ge 2.
\end{array}
\hspace{-.2em}
\]
\end{Lem}

\begin{proof}
Let $v_{[1,\infty)}$ be the $\alpha$-expansion of $x \in (0,\alpha]$, i.e., $v_j = (\varepsilon_{\alpha,j}(x):d_{\alpha,j}(x))$ for all $j \ge 1$.
By~\eqref{e:Mn}, we have
\[
M_{(-1:d')}\, M_{(-1:2)}^{n-1}\, M_{(+1:d)} = M_{(-1:d')}\, W\, M_{(+1:n)}\, E^{-1} M_{(+1:d)} = M_{(+1:d'-1)}\, M_{(+1:n)}\, M_{(+1:d-1)}\,.
\]
Therefore, any sequence $v'_{[1,\infty)}$ which is obtained from $v_{[1,\infty)}$ by replacements of the form $(+1:d)\, (-1:2)^{n-1}\, (-1:d') \mapsto (+1:d-1)\, (+1:n)\, (+1:d'-1)$ satisfies $\llbracket v'_{[1,\infty)} \rrbracket = \llbracket v_{[1,\infty)} \rrbracket = x$.
This includes $(+1:d)\, (-1:2)^n \mapsto (+1:d-1)\, (+1:n)\, (+1:1)$.
We have of course $\llbracket (+1:n-1)\, (+1:1), 0\rrbracket = \llbracket (+1:n), 0\rrbracket$, hence replacing $(+1:d)\, (-1:2)^{n-1}\, (+1:\infty)$ by $(+1:d-1)\, (+1:n)\, (+1:\infty)$ also does not change the value of the sequence.

Since $v_{[1,\infty)}$ does not end with $(+1:1)\, (+1:\infty)^\omega$, the same holds for any new sequence~$v'_{[1,\infty)}$.
Therefore, it only remains to show that all digits in $\mathscr{A}_-$ can be replaced by digits in $\mathscr{A}_+$ using the given rules.
Since $x \in (0, \alpha]$, we have $v_1 \in \mathscr{A}_+$.
If $v_1 = (+1:1)$, then $T_\alpha(x) = 1/x - 1 > 0$ implies that $v_2 \in \mathscr{A}_+$. 
More generally, the pattern $(+1:1)\, (-1:d)$ does not occur in $v_{[1,\infty)}$.
Thus any digit $v_j \in \mathscr{A}_-$ is preceded by a word in $(+1:d)\, \mathscr{A}_-^*$ with $d \ge 2$, and replacements do not change this fact.
This implies that we can successively eliminate all digits in~$\mathscr{A}_-\,$.
\end{proof}

\begin{Rmk}
The above can be compared with the conversions from $\alpha$-expansions to RCF given in \cite{NakadaNatsui02,NakadaNatsui08}. 
\end{Rmk} 

\begin{Lem} \label{l:alpha1}
Let $\alpha \in (0,1)$, $x \in (0,\alpha]$, and suppose that $T_\alpha^m(x) > 0$ for some $m \ge 1$.
Then there is some $n \ge 1$ such that $T_\alpha^m(x) = T_1^n(x)$ and $\mathcal{T}_\alpha^m(x,y) = \mathcal{T}_1^n(x,y)$ for all $y \in [0,1]$.
\end{Lem}

\begin{proof}   
Let $v_{[1,\infty)}$ be the $\alpha$-expansion of $x \in (0,\alpha]$, and $T_\alpha^m(x) > 0$ for some $m \ge 1$. 
The procedure described in Lemma~\ref{l:rewrite} provides a sequence $v'_{[1,n]} \in \mathscr{A}_+^*$ with $M_{v'_{[1,n]}} = M_{v_{[1,m]}}$.
Since $M_{v'_{[1,n]}} \cdot x = M_{v_{[1,m]}} \cdot x = T_\alpha^m(x) \in (0,\alpha)$, we have $v'_j = (+1:d_{1,j}(x))$ for all $1 \le j \le n$, i.e., $T_1^n(x) = M_{v'_{[1,n]}} \cdot x = T_\alpha^m(x)$ and $\mathcal{T}_1^n(x,y) = \mathcal{T}_\alpha^m(x,y)$ for all $y \in [0,1]$.
\end{proof}

\begin{Lem} \label{l:2bounded}
Let $\alpha \in (0,1]$ and $x \in \mathbb{I}_\alpha$.
The $\alpha$-expansion of $x$ contains no sequence of $d_\alpha(\alpha)$ consecutive digits $(-1:2)$.
\end{Lem}

\begin{proof}
The $\alpha$-expansion of $x$ contains a sequence of $d_\alpha(\alpha)$ consecutive digits $(-1:2)$ if and only if the $\alpha$-expansion of $T_\alpha^m(x)$ starts with $(-1:2)^{d_\alpha(\alpha)}$ for some $m \ge 0$.
Therefore, it suffices to show that $(-1:2)^{d_\alpha(\alpha)}$ cannot be a prefix of an $\alpha$-expansion.

Suppose on the contrary that the $\alpha$-expansion of $x$ begins with $(-1:2)^{d_\alpha(\alpha)}$.
In particular, this means that $T_\alpha^n(x) = M_{(-1:2)}^n \cdot x < 0$ for all $0 \le n < d_\alpha(\alpha)$.
By~\eqref{e:Mn}, we have $M_{(-1:2)}^n \cdot z \ge 0$ for all $z \in \big[\frac{1}{n+1} - 1, \frac{1}{n} - 1\big)$.
It follows that $x \in \big[{\alpha-1}, \frac{1}{d_\alpha(\alpha)} - 1\big)$.
Since
\[
\alpha > T_\alpha^{d_\alpha(\alpha)}(x) = M_{(-1:2)}^{d_\alpha(\alpha)} \cdot x \ge M_{(-1:2)}^{d_\alpha(\alpha)} \cdot ({\alpha-1}) = \frac{d_\alpha(\alpha)\alpha+{\alpha-1}}{1-d_\alpha(\alpha)\alpha}\,,
\]
where we have used that the action of  $M_{(-1:2)}$ is order preserving on the negative numbers, and $\alpha \le x+1 < \frac{1}{d_\alpha(\alpha)}$, we obtain that $d_\alpha(\alpha) \alpha^2 + d_\alpha(\alpha) \alpha - 1 < 0$.
We must also have $\alpha > T_\alpha(\alpha) = \frac{1}{\alpha} - d_\alpha(\alpha)$, thus $\alpha^2 + d_\alpha(\alpha) \alpha - 1 > 0$.
Since $d_\alpha(\alpha) \ge 1$, this is impossible.
\end{proof}

\begin{Lem} \label{l:negativeToRCF}
Let $\alpha \in (0,1)$, $x \in (0,\alpha]$ and suppose that $T_\alpha^m(x) < 0$ for all $m \ge 1$.
Then, for any $n \ge 1$, we cannot have both $d_{1,n}(x) > d_\alpha(\alpha)$ and $d_{1,n+1}(x) > d_\alpha(\alpha)$.
\end{Lem}

\begin{proof}
If $x \in (0,\alpha]$, $T_\alpha^m(x) < 0$ for all $m \ge 1$, then we can write the $\alpha$-expansion of~$x$ as
\[
(+1:d)\, (-1:2)^{a_1-1}\, (-1:2+a_2)\, (-1:2)^{a_3-1}\, (-1:2+a_4) \cdots
\]
with $d \ge 2$, $a_j \ge 1$ for all $j \ge 1$.
By Proposition~\ref{p:0toRCF}, the $1$-expansion of~$x$ is 
\[
(+1:d-1)\, (+1:a_1)\, (+1:a_2)\, (+1:a_3)\, (+1:a_4)\, \cdots\,.
\]
By Lemma~\ref{l:2bounded}, we have $a_{2j+1} \le d_\alpha(\alpha)$ for all $j \ge 0$, which proves the lemma.
\end{proof}

\begin{Lem} \label{l:1alpha}
Let $\alpha \in (0,1)$, $x \in (0,\alpha]$, and $T_1^{n-1}(x) \in \big(0, \frac{1}{d_\alpha(\alpha)+1}\big]$, $T_1^n(x) \in \big(0, \frac{1}{d_\alpha(\alpha)+1}\big]$ for some $n \ge 1$.
Then there is some $m \ge 1$ such that $T_1^n(x) = T_\alpha^m(x)$ and $\mathcal{T}_1^n(x,y) = \mathcal{T}_\alpha^m(x,y)$ for all $y \in [0,1]$.
\end{Lem}

\begin{proof}
Let $v_{[1,\infty)}$ be the $\alpha$-expansion of $x \in (0,\alpha]$, and $v'_{[1,\infty)}$ its $1$-expansion. 
If $T_1^{n-1}(x) \in \big(0, \frac{1}{d_\alpha(\alpha)+1}\big]$, $T_1^n(x) \in \big(0, \frac{1}{d_\alpha(\alpha)+1}\big]$, then $v'_n = (+1:d)$, $v'_{n+1} = (+1:d')$ with $d, d' > d_\alpha(\alpha)$.
Similarly to the proof of Lemma~\ref{l:negativeToRCF}, Lemmas~\ref{l:rewrite} and~\ref{l:2bounded} imply that $v'_{n+1} = v_{m+1}$, $M_{v'_{[1,n]}} = M_{v_{[1,m]}}$ for some $m \ge 1$.
Therefore, we have $T_1^n(x) = T_\alpha^m(x)$ and $\mathcal{T}_1^n(x,y) = \mathcal{T}_\alpha^m(x,y)$.
\end{proof}

\section{Natural extensions and entropy} \label{sec:natural-extensions}

The advantage for number theoretic usage of the natural extension map in the form~$\mathcal{T}_\alpha$ is that  the Diophantine approximation to an $x \in [{\alpha-1}, \alpha)$ by the finite steps in its $\alpha$-expansion is directly related to the $\mathcal{T}_\alpha$-orbit of $(x,0)$; see \cite{Kraaikamp:91}.  
We define our natural extension domain in terms of  these orbits.
We show moreover that the entropy of~$T_\alpha$ is directly related to the measure of the natural extension domain;  that is, this section ends with the proof of Theorem~\ref{t:hmu}.

We will see that the structure of $\Omega_\alpha$ can be quite complicated.
Even for ``nice'' numbers such as $\alpha = g^2$ and $\alpha = 1/4$ it has a fractal structure;  see \cite{LuzziMarmi08} and Figures~\ref{f:g2} and~\ref{f:4}.

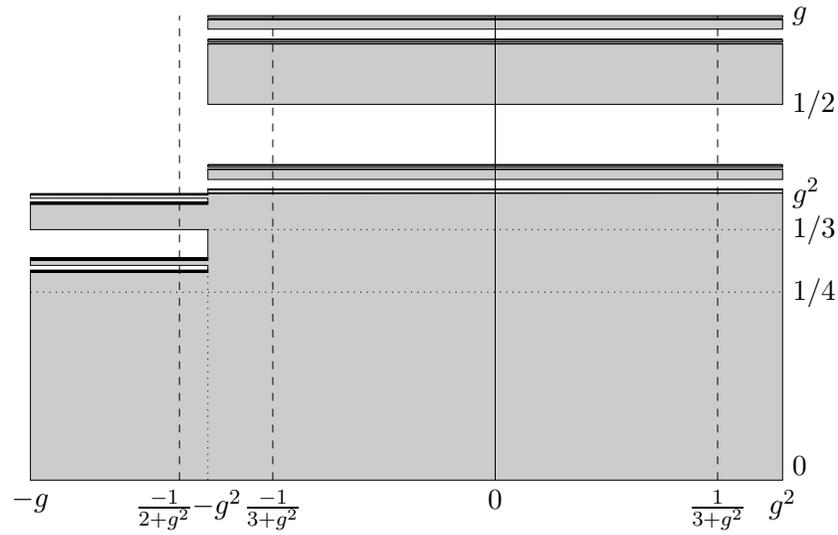
\begin{figure}[ht]
\begin{tikzpicture}[scale=10,fill=black!20,ultra thin]
\small
\filldraw(-.618,.0)--(-.618,.2763)--(-.382,.2763)--(-.382,.2766)--(-.618,.2766)--(-.618,.2767)--(-.382,.2767)--(-.382,.2778)--(-.618,.2778)--(-.618,.2787)--(-.382,.2787)--(-.382,.2791)--(-.618,.2791)--(-.618,.2792)--(-.382,.2792)--(-.382,.2857)--(-.618,.2857)--(-.618,.2923)--(-.382,.2923)--(-.382,.2927)--(-.618,.2927)--(-.618,.2929)--(-.382,.2929)--(-.382,.2941)--(-.618,.2941)--(-.618,.2951)--(-.382,.2951)--(-.382,.2955)--(-.618,.2955)--(-.618,.2956)--(-.382,.2956)--(-.382,.3333)--(-.618,.3333)--(-.618,.3671)--(-.382,.3671)--(-.382,.3673)--(-.618,.3673)--(-.618,.3675)--(-.382,.3675)--(-.382,.3684)--(-.618,.3684)--(-.618,.3692)--(-.382,.3692)--(-.382,.3696)--(-.618,.3696)--(-.618,.3697)--(-.382,.3697)--(-.382,.375)--(-.618,.375)--(-.618,.3797)--(-.382,.3797)--(-.382,.38)--(-.618,.38)--(-.618,.3801)--(-.382,.3801)--(-.382,.381)--(-.618,.381)--(-.618,.381)--(.382,.3819)--(.382,.0)--cycle;
\filldraw(-.382,.382)--(-.382,.3866)--(.382,.3866)--(.382,.382)--cycle;
\filldraw(-.382,.3871)--(-.382,.3874)--(.382,.3874)--(.382,.3871)--cycle;
\filldraw(-.382,.4)--(-.382,.4132)--(.382,.4132)--(.382,.4)--cycle;
\filldraw(-.382,.4138)--(-.382,.4141)--(.382,.4141)--(.382,.4138)--cycle;
\filldraw(-.382,.4167)--(-.382,.4188)--(.382,.4188)--(.382,.4167)--cycle;
\filldraw(-.382,.4194)--(-.382,.4196)--(.382,.4196)--(.382,.4194)--cycle;
\filldraw(-.382,.5)--(-.382,.5802)--(.382,.5802)--(.382,.5)--cycle;
\filldraw(-.382,.5806)--(-.382,.581)--(.382,.581)--(.382,.5806)--cycle;
\filldraw(-.382,.5833)--(-.382,.5856)--(.382,.5856)--(.382,.5833)--cycle;
\filldraw(-.382,.5862)--(-.382,.5865)--(.382,.5865)--(.382,.5862)--cycle;
\filldraw(-.382,.6)--(-.382,.6124)--(.382,.6124)--(.382,.6)--cycle;
\filldraw(-.382,.6129)--(-.382,.6132)--(.382,.6132)--(.382,.6129)--cycle;
\filldraw(-.382,.6154)--(-.382,.618)--(.382,.618)--(.382,.6154)--cycle;
\node[below] at (-.618,.0) {$-g$};
\draw[thin,dotted](-.382,.0)--(-.382,.2763);
\node[below] at (-.37,.0) {$-g^2$};
\draw(0,0)--(0,.618);
\node[below] at (0,0) {$0$};
\node[below] at (.382,.0) {$g^2$};
\draw[dashed](-.4198,.0)--(-.4198,.618);
\node[below] at (-.44,0) {$\frac{-1}{2+g^2}$};
\draw[dashed](-.2957,.0)--(-.2957,.618);
\node[below] at (-.2957,0) {$\frac{-1}{3+g^2}$};
\draw[dashed](.2957,.0)--(.2957,.618);
\node[below] at (.2957,0) {$\frac{1}{3+g^2}$};
\node[right] at (.382,.618) {$g$};
\node[right] at (.382,.5) {$1/2$};
\node[right] at (.382,.382) {$g^2$};
\draw[thin,dotted](-.382,.3333)--(.382,.3333)node[right]{$1/3$};
\draw[thin,dotted](-.618,.25)--(.382,.25)node[right]{$1/4$};
\node[right] at (.382,0.02) {$0$};
\end{tikzpicture}
\caption{The natural extension domain $\Omega_{g^2}$.}
\label{f:g2}
\end{figure}

\begin{figure}[ht]
\begin{tikzpicture}[scale=12,fill=black!20,ultra thin]
\small
\filldraw(-.75,.0)--(-.75,.2083)--(.0,.2083)--(.0,.2105)--(-.75,.2105)--(-.75,.2111)--(.0,.2111)--(.0,.2143)--(-.75,.2143)--(-.75,.2154)--(.0,.2154)--(.0,.2222)--(-.75,.2222)--(-.75,.225)--(.0,.225)--(.0,.2308)--(-.75,.2308)--(-.75,.2321)--(.0,.2321)--(.0,.25)--(-.75,.25)--(-.75,.2637)--(.25,.2637)--(.25,.0)--cycle;
\filldraw(-.75,.2667)--(-.75,.2676)--(.25,.2676)--(.25,.2667)--cycle;
\filldraw(-.75,.2727)--(-.75,.2745)--(.25,.2745)--(.25,.2727)--cycle;
\filldraw(-.75,.2778)--(-.75,.2785)--(.25,.2785)--(.25,.2778)--cycle;
\filldraw(-.75,.2857)--(-.75,.2903)--(.25,.2903)--(.25,.2857)--cycle;
\filldraw(-.75,.2941)--(-.75,.2949)--(.25,.2949)--(.25,.2941)--cycle;
\filldraw(-.75,.3)--(-.75,.3023)--(.25,.3023)--(.25,.3)--cycle;
\filldraw(-.75,.3333)--(-.75,.3582)--(.25,.3582)--(.25,.3333)--cycle;
\filldraw(-.75,.3636)--(-.75,.3654)--(.25,.3654)--(.25,.3636)--cycle;
\filldraw(-.75,.3684)--(-.75,.369)--(.25,.369)--(.25,.3684)--cycle;
\filldraw(-.75,.375)--(-.75,.3784)--(.25,.3784)--(.25,.375)--cycle;
\filldraw(-.75,.381)--(-.75,.3814)--(.25,.3814)--(.25,.381)--cycle;
\filldraw(-.75,.3846)--(-.75,.386)--(.25,.386)--(.25,.3846)--cycle;
\filldraw(-.75,.3889)--(-.75,.3896)--(.25,.3896)--(.25,.3889)--cycle;
\filldraw(-.75,.4)--(-.75,.4091)--(.25,.4091)--(.25,.4)--cycle;
\filldraw(-.75,.4118)--(-.75,.4125)--(.25,.4125)--(.25,.4118)--cycle;
\filldraw(-.75,.4167)--(-.75,.4182)--(.25,.4182)--(.25,.4167)--cycle;
\filldraw(-.75,.4211)--(-.75,.4217)--(.25,.4217)--(.25,.4211)--cycle;
\filldraw(-.75,.4286)--(-.75,.4333)--(.25,.4333)--(.25,.4286)--cycle;
\filldraw(-.75,.4375)--(-.75,.4384)--(.25,.4384)--(.25,.4375)--cycle;
\filldraw(-.6667,.5)--(-.6667,.5581)--(.25,.5581)--(.25,.5)--cycle;
\filldraw(-.6667,.5625)--(-.6667,.5634)--(.25,.5634)--(.25,.5625)--cycle;
\filldraw(-.6667,.5652)--(-.6667,.5657)--(.25,.5657)--(.25,.5652)--cycle;
\filldraw(-.6667,.5714)--(-.6667,.5758)--(.25,.5758)--(.25,.5714)--cycle;
\filldraw(-.6667,.5789)--(-.6667,.5795)--(.25,.5795)--(.25,.5789)--cycle;
\filldraw(-.6667,.5833)--(-.6667,.5849)--(.25,.5849)--(.25,.5833)--cycle;
\filldraw(-.6667,.5882)--(-.6667,.589)--(.25,.589)--(.25,.5882)--cycle;
\filldraw(-.6667,.6)--(-.6667,.6087)--(.25,.6087)--(.25,.6)--cycle;
\filldraw(-.6667,.6111)--(-.6667,.6118)--(.25,.6118)--(.25,.6111)--cycle;
\filldraw(-.6667,.6154)--(-.6667,.6167)--(.25,.6167)--(.25,.6154)--cycle;
\filldraw(-.6667,.619)--(-.6667,.6196)--(.25,.6196)--(.25,.619)--cycle;
\filldraw(-.6667,.625)--(-.6667,.6286)--(.25,.6286)--(.25,.625)--cycle;
\filldraw(-.6667,.6316)--(-.6667,.6322)--(.25,.6322)--(.25,.6316)--cycle;
\filldraw(-.6667,.6364)--(-.6667,.6383)--(.25,.6383)--(.25,.6364)--cycle;
\filldraw(-.5,.6667)--(-.5,.6935)--(.25,.6935)--(.25,.6667)--cycle;
\filldraw(-.5,.7)--(-.5,.7021)--(.25,.7021)--(.25,.7)--cycle;
\filldraw(-.5,.7059)--(-.5,.7067)--(.25,.7067)--(.25,.7059)--cycle;
\filldraw(-.5,.7143)--(-.5,.7188)--(.25,.7188)--(.25,.7143)--cycle;
\filldraw(-.5,.7222)--(-.5,.7229)--(.25,.7229)--(.25,.7222)--cycle;
\filldraw(-.5,.7273)--(-.5,.7292)--(.25,.7292)--(.25,.7273)--cycle;
\filldraw(-.5,.7333)--(-.5,.7344)--(.25,.7344)--(.25,.7333)--cycle;
\filldraw(.0,.75)--(.0,.7654)--(.25,.7654)--(.25,.75)--cycle;
\filldraw(.0,.7692)--(.0,.7705)--(.25,.7705)--(.25,.7692)--cycle;
\filldraw(.0,.7727)--(.0,.7732)--(.25,.7732)--(.25,.7727)--cycle;
\filldraw(.0,.7778)--(.0,.7805)--(.25,.7805)--(.25,.7778)--cycle;
\filldraw(.0,.7857)--(.0,.7869)--(.25,.7869)--(.25,.7857)--cycle;
\filldraw(.0,.7895)--(.0,.7901)--(.25,.7901)--(.25,.7895)--cycle;
\draw(0,0)--(.0,.7901);
\draw[thin,dotted](-.6667,.0)--(-.6667,.5);
\draw[thin,dotted](-.5,.0)--(-.5,.6667);
\node[below] at (-.75,.0) {$-3/4$};
\node[below] at (-.6667,.0) {$-2/3$};
\node[below] at (-.52,.0) {$-1/2$};
\node[below] at (0,0) {$\vphantom{/}0$};
\node[below] at (.27,.0) {$1/4$};
\draw[dashed](-.4444,.0)--(-.4444,.7901);
\node[below] at (-.4444,0) {$-4/9$};
\draw[dashed](.2353,.0)--(.2353,.7901);
\node[below] at (.21,0) {$4/17$};
\node[right] at (.25,.005) {$0$};
\node[right] at (.25,.75) {$3/4$};
\node[right] at (.25,.6667) {$2/3$};
\node[right] at (.25,.5) {$1/2$};
\node[right] at (.25,.3333) {$1/3$};
\draw[thin,dotted](.0,.25)--(.25,.25)node[right]{$1/4$};
\draw[thin,dotted](-.75,.2)--(.25,.2)node[right]{$1/5$};
\end{tikzpicture}
\caption{The natural extension domain $\Omega_{1/4}$.}
\label{f:4}
\end{figure}
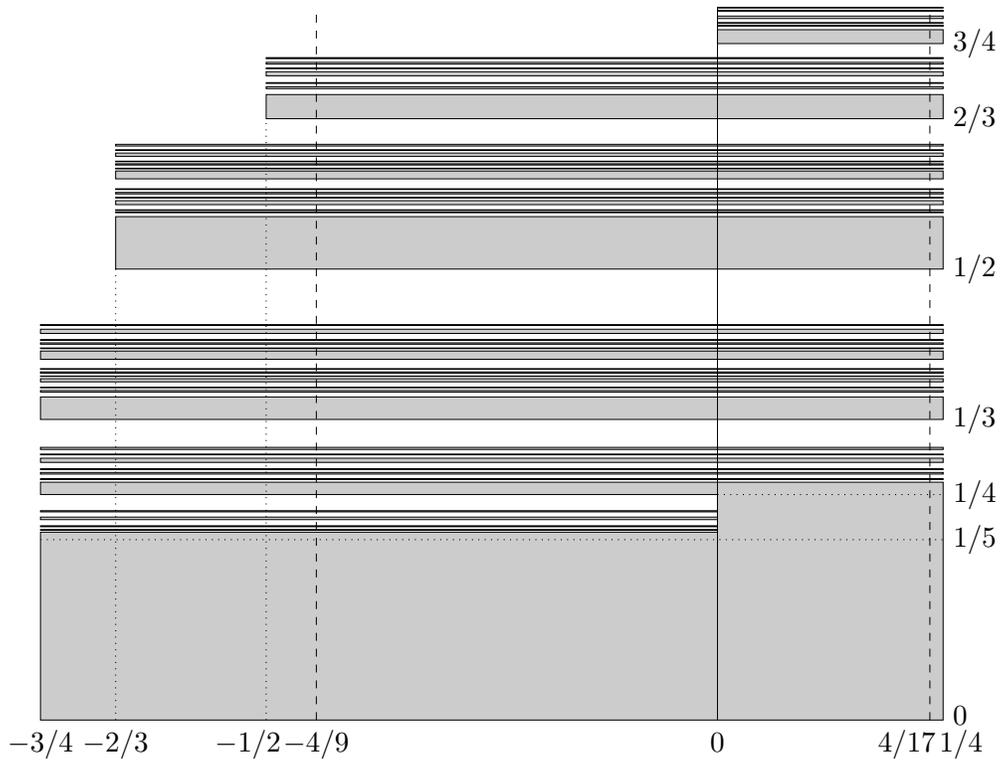

In the following, we show that $\mathcal{T}_\alpha$ and $\Omega_\alpha$ give indeed a natural extension of~$T_\alpha$.

\begin{Lem} \label{l:positivemeasure}
Let $\alpha \in (0,1]$.
We have 
\[
\big[0, \tfrac{1}{d_\alpha(\alpha)+1}\big]^2 \subset \Omega_\alpha \subseteq \mathbb{I}_\alpha \times [0,1],
\]
thus $0 < \mu(\Omega_\alpha) < \infty$.
\end{Lem}

\begin{proof}
The inclusion $\Omega_\alpha \subseteq \mathbb{I}_\alpha \times [0,1]$ follows from the inclusion $N_a \cdot [0,1] \subset [0,1]$, which holds for every $a \in \mathscr{A}$.
Therefore, $\Omega_\alpha$~is bounded away from $y=-1/x$, and its compactness yields that $\mu(\Omega_\alpha) < \infty$.  

It remains to show that $\Omega_\alpha$ contains the square $\big[0, \frac{1}{d_\alpha(\alpha)+1}\big]^2$, which implies $\mu(\Omega_\alpha) > 0$.
Every point in $\big[0, \frac{1}{d_\alpha(\alpha)+1}\big]^2$ can be approximated by points $\mathcal{T}_1^n(x_n,0)$, $n \ge 1$, with $x_n \in (0,\alpha]$, $T_1^{n-1}(x_n) \le \frac{1}{d_\alpha(\alpha)+1}$, $T_1^n(x_n) \le \frac{1}{d_\alpha(\alpha)+1}$.
By Lemma~\ref{l:1alpha}, there exist numbers $m_n \ge 1$ such that $\mathcal{T}_1^n(x_n,0) = \mathcal{T}_\alpha^{m_n}(x_n,0)$, from which we conclude that $\big[0, \frac{1}{d_\alpha(\alpha)+1}\big]^2 \subset \Omega_\alpha$.
\end{proof}

\begin{Lem} \label{l:bijective}
Let $\alpha \in (0,1]$.
Up to a set of $\mu$-measure zero, $\mathcal{T}_\alpha$~is a bijective map from $\Omega_\alpha$ to~$\Omega_\alpha$.
\end{Lem}

\begin{proof}   
For $a \in \mathscr{A}$, let $D_\alpha(a) := \{(x,y) \in \Omega_\alpha \mid x \in \Delta_\alpha(a)\}$.   
The map $\mathcal{T}_\alpha$ is one-to-one, continuous and $\mu$-preserving on each $D_\alpha(a)$. 
Now,  as the $\Delta_\alpha(a)$ partition $\mathbb{I}_\alpha \setminus \{0\}$,  $\mathcal{T}_\alpha$~is continuous on $\Omega_\alpha$ except for its intersection with a countable number of vertical lines.
Since $\Omega_{\alpha}$ is compact and bounded away from $y = -1/x$,  these lines are of $\mu$-measure zero.  
Thus, we find that
\begin{equation} \label{e:TOmega}
\mathcal{T}_\alpha(\Omega_\alpha) = \overline{\big\{\mathcal{T}_\alpha^{n+1}(x,0) \mid x \in [{\alpha-1},\alpha),\, n \ge 0\big\}}\,,
\end{equation}
up to a $\mu$-measure zero set, hence $\mathcal{T}_\alpha(\Omega_\alpha) = \Omega_\alpha$.
This implies that
\[
\sum_{a \in \mathscr{A}} \mu\big(\mathcal{T}_\alpha(D_\alpha(a))\big) = \sum_{a \in \mathscr{A}} \mu\big(D_\alpha(a)\big) = \mu\big(\Omega_\alpha\big) = \mu\big(\mathcal{T}_\alpha(\Omega_\alpha)\big) = \mu\bigg(\bigcup_{a \in \mathscr{A}} \mathcal{T}_\alpha(D_\alpha(a))\Big)\,,
\]
and thus
\[
\mu\big(\mathcal{T}_\alpha(D_\alpha(a)) \cap \mathcal{T}_\alpha(D_\alpha(a'))\big) = 0 \quad \mbox{for all}\ a, a' \in \mathscr{A}\ \mbox{with}\ a \ne a'\,.
\]
From its injectivity on the~$D_\alpha(a)$,  we conclude that $\mathcal{T}_\alpha$ is bijective on $\Omega_\alpha$ up to a set of measure zero.
\end{proof}    

Our candidate  $(\Omega_\alpha, \mathcal{T}_\alpha, \mathscr{B}_\alpha', \mu_\alpha)$ for a natural extension of $(\mathbb{I}_\alpha, T_\alpha, \mathscr{B}_\alpha, \nu_\alpha)$ is such that the factor map is projection onto the first coordinate,  call this map $\pi$.     The first three criteria of the definition of a natural extension  are clearly satisfied here:   (1)~$\pi$ is a surjective and measurable map that pulls-back $\mu_\alpha$ to $\nu_\alpha$; (2)~$\pi \circ \mathcal{T}_\alpha = T_\alpha \circ \pi$; and,  (3)~$\mathcal{T}_\alpha$ is an invertible transformation. 
It remains to show the \emph{minimality} of the extended system:   (4)~any invertible system that admits $(\mathbb{I}_\alpha, T_\alpha, \mathscr{B}_\alpha, \nu_\alpha)$ as a factor must itself be a factor of $(\Omega_\alpha, \mathcal{T}_\alpha, \mathscr{B}_\alpha', \mu_\alpha)$.    We employ the standard method to verify this last criterion, in that we verify that $\mathscr{B}_\alpha' = \bigvee_{n\ge 0}  \mathcal{T}_{\alpha}^{n} \pi^{-1} \mathscr{B}_\alpha$.   
  
\begin{proof}[\textbf{Proof of Theorem~\ref{t:natext}}]
Since $\mathcal{T}_\alpha$ is invertible, with $\mu_\alpha$ as an invariant probability measure, we must only show that $\mathscr{B}_\alpha' = \bigvee_{n\ge 0}  \mathcal{T}_{\alpha}^{n} \pi^{-1} \mathscr{B}_\alpha$,  where $\pi$ is the projection map to the first coordinate. 
As usual, we define rank $n$ cylinders as $\Delta_{\alpha}(v_{[1,n]})  = \bigcap_{j=1}^{n} \, T_{\alpha}^{-j+1}(\Delta_\alpha(v_j))$.   
Since $T_{\alpha}$ is expanding,  for any $v_{[1,\infty)} \in \mathscr{A}_0^\omega$ the Lebesgue measure of $\Delta_{\alpha}(v_{[1,n]})$ tends to zero as $n$ goes to infinity.
Thus $P_{\alpha}$, the collection of all of these cylinders, generates~$\mathscr{B}_\alpha$.  
Let $\mathcal{P}_{\alpha} = \pi^{-1} P_{\alpha}$; it suffices to show that $\bigvee_{n\in \mathbb{Z}}  \mathcal{T}_{\alpha}^{n} \mathcal{P}_{\alpha}$ separates points of~$\Omega_{\alpha}$.  
We know that $\bigvee_{n\ge 0}  T_{\alpha}^{n} P_{\alpha}$ separates points of~$\mathbb{I}_\alpha$,  thus $\bigvee_{n\ge 0}  \mathcal{T}_{\alpha}^{n} \mathcal{P}_{\alpha}$ separates points of the form $(x,y), (x', y')$ with $x \neq x'$.     It now suffices to show that powers of $\mathcal{T}_{\alpha}^{-1}$ on $\mathcal{P}_{\alpha}$ can separate points sharing the same $x$-value.
Now,  on some neighborhood of  $\mu_{\alpha}$-almost any point of $\Omega_{\alpha}$, there is $a \in \mathscr{A}$ such that $\mathcal{T}_{\alpha}^{-1}$ is given locally by  $(x,y) \mapsto (M_a^{-1} \cdot x, N_a^{-1} \cdot y)$. 
But, $N_a^{-1} \cdot y$ is an expanding map. 
Since $\mathcal{T}_{\alpha}^{-1}$ takes horizontal strips to vertical strips,  one can separate points. 
\end{proof}

With the help of the following lemma and Abramov's formula, we show that the product of the entropy and the measure of the natural extension domain is constant.

\begin{Lem} \label{l:equalreturn}
Let $\alpha \in (0,1]$, $\mathcal{T}_{1,\alpha}$~be the first return map of $\mathcal{T}_1$ on $\Omega_1 \cap \Omega_\alpha$, and $\mathcal{T}_{\alpha,1}$ be the first return map of $\mathcal{T}_\alpha$ on $\Omega_1 \cap \Omega_\alpha$.
For $\mu$-almost all $(x,y) \in \Omega_1 \cap \Omega_\alpha$, these two maps are defined and $\mathcal{T}_{1,\alpha}(x,y) = \mathcal{T}_{\alpha,1}(x,y)$.
\end{Lem}

\begin{proof}
Note first that $\Omega_1 \cap \Omega_\alpha = \{(x,y) \in \Omega_\alpha \mid x \ge 0\}$ since $\Omega_1 = [0,1]^2$.
The ergodicity of~$T_\alpha$ (see \cite{LuzziMarmi08}) yields that, for $\nu_{\alpha}$-almost every $x \in [0,\alpha]$, there exists some $m \ge 0$ such that $T_\alpha^m(x) \ge 0$, and thus there exists some $n \ge 0$ such that $\mathcal{T}_{\alpha,1}(x,y) = \mathcal{T}_1^n(x,y)$ by Lemma~\ref{l:alpha1}.   
Then we have $\mathcal{T}_{1,\alpha}(x,y) = \mathcal{T}_1^{n'}(x,y)$ with $1 \le n' \le n$, thus $\mathcal{T}_{1,\alpha}$ and $\mathcal{T}_{\alpha,1}$ are defined for $\nu_\alpha$-almost all~$x \in [0,\alpha]$, hence for $\mu$-almost all $(x,y) \in \Omega_1 \cap \Omega_\alpha$.

The ergodicity of~$T_1$ yields that, for $\nu_1$-almost every $x \in [0,\alpha]$, there exists some $n'' \ge 1$ such that $T_1^{n''-1}(x) \le \frac{1}{d_\alpha(\alpha)+1}$ and $T_1^{n''}(x) \le \frac{1}{d_\alpha(\alpha)+1}$.
By Lemma~\ref{l:1alpha}, we obtain that $\mathcal{T}_1^{n''}(x,y) = \mathcal{T}_\alpha^{m'}(x,y)$ for some $m' \ge 1$;   it follows that $n'' \ge n'$.
Applying  Lemma~\ref{l:1alpha} a second time, we find that $\mathcal{T}_1^{n''}(x,y) = \mathcal{T}_1^{n''-n'} \mathcal{T}_1^{n'}(x,y) = \mathcal{T}_\alpha^{m''} \mathcal{T}_1^{n'}(x,y)$ for some $m'' \ge 1$, with $m'' \le m'$.
Since $\mathcal{T}_\alpha$ is bijective $\mu$-almost everywhere, we obtain that $\mathcal{T}_1^{n'}(x,y) = \mathcal{T}_\alpha^{-m''} \mathcal{T}_1^{n''}(x,y) = \mathcal{T}_\alpha^{m'-m''}(x,y)$ for $\mu$-almost all $(x,y) \in \Omega_1 \cap \Omega_\alpha$.    Since  for these $(x,y)$ there is a power of $\mathcal{T}_\alpha$ that agrees with the first return of $\mathcal{T}_1$, it follows that $\mathcal{T}_{1,\alpha}(x,y) = \mathcal{T}_{\alpha,1}(x,y)$ holds here.    
\end{proof}

\begin{proof}[\textbf{Proof of Theorem~\ref{t:hmu}}]
It is well known that $h(T_1) = \pi^2/(6 \log 2)$ and that $\mu(\Omega_1) = \mu([0,1]^2) = \log 2$, thus $h(T_1)\, \mu(\Omega_1) = \pi^2/6$.
With the definitions of Lemma~\ref{l:equalreturn}, Abramov's formula~\cite{Abramov59} yields that
\[
h(\mathcal{T}_{1,\alpha}) = \frac{\mu(\Omega_1)}{\mu(\Omega_1\cap\Omega_\alpha)}\, h(\mathcal{T}_1) \quad \mbox{and} \quad h(\mathcal{T}_{\alpha,1}) = \frac{\mu(\Omega_\alpha)}{\mu(\Omega_1\cap\Omega_\alpha)}\, h(\mathcal{T}_\alpha)\,.
\]
Since $\mathcal{T}_{1,\alpha}$ and $\mathcal{T}_{\alpha,1}$ are equal (up to a set of measure zero), and a system and its natural extension have the same entropy \cite{Rohlin61}, we obtain that $h(T_\alpha)\, \mu(\Omega_\alpha) = h(T_1)\, \mu(\Omega_1)$.
\end{proof}

\section{Intervals of synchronizing orbits} \label{s:relation}

Luzzi and Marmi indicate in \cite[Remark~3]{LuzziMarmi08} that the natural extension domain can be described in an explicit way when one has an explicit relation between the $\alpha$-expansions of ${\alpha-1}$ and of~$\alpha$.
Such a relation is easily found for $\alpha \ge \sqrt{2}-1$.
Nakada and Natsui \cite{NakadaNatsui08} find the relation on some subintervals of $(0,\sqrt{2}-1)$, showing that it can be rather complicated.
The aim of this section is to provide a relation for every $\alpha \in (0,1]$, i.e., to prove Theorem~\ref{t:endpoints}.

\begin{Lem} \label{l:conversion}
For each $v \in \mathscr{A}_-^*$, we have $M_{\widehat{v}} = E\, W M_v\, E\, W$.
\end{Lem}

\begin{proof}
Let $a_{[1,2\ell+1]}$ be the characteristic sequence of $v \in \mathscr{A}_-^*$.
By~\eqref{e:Mn2} and~\eqref{e:Mn}, we have
\[
E\, W M_{(-1:2)}^{n-1}\, E\, W = M_{(-1:2+n)}\,, 
\]
thus
\begin{align*}
E\, W M_v\,E\, W & = E\, W M_{(-1:2)}^{a_{2\ell+1}-1}\, M_{(-1:2+a_{2\ell})} \cdots M_{(-1:2)}^{a_3-1}\, M_{(-1:2+a_2)}\, M_{(-1:2)}^{a_1-1}\, E\, W \\
& = M_{(-1:2+a_{2\ell+1})}\, M_{(-1:2)}^{a_{2\ell}-1} \cdots M_{(-1:2+a_3)}\, M_{(-1:2)}^{a_2-1}\, M_{(-1:2+a_1)} = M_{\widehat{v}}\,. \qedhere
\end{align*}
\end{proof}

\begin{Lem} \label{l:conversion2}
If $\alpha-1 = \llbracket v, x \rrbracket$ with $v \in \mathscr{A}_-^*$, $x \in \mathbb{R}$, then $\alpha = \llbracket{}^{(W)}\widehat{v}{}^{(-1)}, W \cdot x\rrbracket$.

If $\alpha-1 = \llbracket v\rrbracket$, with $v \in \mathscr{A}_-^\omega$, then $\alpha = \llbracket {}^{(W)}\widehat{v}\,\rrbracket$.
\end{Lem}

\begin{proof}
For any $v \in \mathscr{A}_-^*$, Lemma~\ref{l:conversion} implies that 
\begin{equation} \label{e:funMatrix}
M_v \cdot (\alpha-1) = M_v\, E \cdot \alpha = W E^{-1} M_{\widehat{v}}\, W \cdot \alpha\,,
\end{equation}
which proves the first statement.
Taking limits, the second statement follows.
\end{proof}

\begin{Lem} \label{l:digit}
Let $x \in \mathbb{I}_\alpha \setminus \{0\}$, and $a \in \mathscr{A}$.
If $|T_\alpha(x) -  M_a \cdot x| < 1$, then $T_\alpha(x) =  M_a \cdot x$ and $(\varepsilon(x):d_\alpha(x)) = a$.
\end{Lem}

\begin{proof}
Let $a = (\varepsilon':d') \in \mathscr{A}$, then we have $M_a \cdot x = \varepsilon'/x - d'$ and $T_\alpha(x) = \varepsilon(x)/x - d_\alpha(x)$.
We cannot have $\varepsilon' \ne \varepsilon(x)$ since this would imply $M_a \cdot x \le -2$, contradicting $|T_\alpha(x) -  M_a \cdot x| < 1$.
Since $d'$, $d_\alpha(x)$ are integers and $\varepsilon' = \varepsilon(x)$, $|T_\alpha(x) -  M_a \cdot x| < 1$ yields that $d' = d_\alpha(x)$.
\end{proof}

We will use Lemma~\ref{l:digit} mainly to deduce that $T_\alpha(x) =  M_a \cdot x$ from $M_a \cdot x \in [{\alpha-1}, \alpha)$ or from $T_\alpha(x) < 0$, $M_a \cdot x \in [-1,0)$.

\begin{Lem}\label{l:wMat} 
Let $x \in \mathbb{I}_\alpha$. 
If $W \cdot x \in \mathbb{I}_\alpha$, then $T_\alpha(W \cdot x) = T_\alpha(x)$ and $(\varepsilon(W \cdot x):d_\alpha(W \cdot x)) = {}^{(W)} (\varepsilon(x):d_\alpha(x))$.
\end{Lem}

\begin{proof}
For $x = 0$, this is clear since $W \cdot 0 = 0$.
For $x \ne 0$, we have 
\[
M_{{}^{(W)} (\varepsilon(x):d_\alpha(x))} \cdot (W \cdot x) = M_{(\varepsilon(x):d_\alpha(x))} W\, W \cdot x = M_{(\varepsilon(x):d_\alpha(x))} \cdot x = T_\alpha(x)\,,
\]
thus $T_\alpha(W \cdot x) = T_\alpha(x)$ and $(\varepsilon(W \cdot x):d_\alpha(W \cdot x)) = {}^{(W)} (\varepsilon(x):d_\alpha(x))$ by Lemma~\ref{l:digit}.
\end{proof}

\begin{Lem} \label{l:Gammav}
Let $\alpha \in (0,1]$, $v \in \mathscr{F}$.
Then 
\begin{equation} \label{e:vvhat}
\underline{b}{}^\alpha_{[1,|v|\,]} = v \quad \mbox{and} \quad \overline{b}{}^\alpha_{[1,|\widehat{v}|\,]} = {}^{(W)}\widehat{v}{}^{(-1)}
\end{equation}
hold if and only if $\alpha \in \Gamma_v$.

If $\alpha \in (\zeta_v, \eta_v)$, then we have
\begin{equation} \label{e:gtetav}
T_\alpha^n(\alpha-1) > \eta_v-1 \ \mbox{for all}\ 1 \le n \le |v|\,, \quad T_\alpha^n(\alpha) > \eta_v-1 \ \mbox{for all}\ 1 \le n \le |\widehat{v}|\,.
\end{equation}
Moreover, we have $M_v \cdot (\eta_v-1) = \eta_v$ (if $|v| \ge 1$) and $M_{v'} \cdot \zeta_v = \zeta_v$, with $v' = {}^{(W)}\widehat{v}{}^{(-1)}$.
\end{Lem}

\begin{proof}
Let $v \in \mathscr{F}$ with characteristic sequence $a_{[1,2\ell+1]}$, $\alpha \in (0,1]$.
If $v$ is the empty word, then $\widehat{v} = (-1:3)$, and ${}^{(W)}\widehat{v}{}^{(-1)} = (+1:1) = \overline{b}{}^\alpha_1$ if and only if $\alpha \in (g,1]$. 
If $\alpha \in (g,1)$, then $T_\alpha(\alpha) = \frac{1}{\alpha} - 1 > 0$, thus \eqref{e:gtetav} holds in this case.
We also have $M_{(+1:1)} \cdot g = g$.

Assume from now on that $|v| \ge 1$;  in particular, by Proposition~\ref{p:0toRCF}, $a_1 \ge 2$.
The characteristic sequence of $\zeta_v - 1 = \llbracket  (v\, \widehat{v}\,)^\omega\rrbracket$ is $(a_{[1,2\ell+1]})^\omega$, and that of $\eta_v - 1 = \llbracket  (v^{(+1)})^\omega \rrbracket$ is
\[
a'_{[1,\infty)} = \left\{\begin{array}{cl}\big(a_{[1,2\ell]}\, (a_{2\ell+1}-1)\, 1\big)^\omega & \mbox{if}\ a_{2\ell+1} \ge 2, \\[1ex] \big(a_{[1,2\ell)}\, (a_{2\ell}+1)\big)^\omega & \mbox{if}\ a_{2\ell+1} = 1.\end{array}\right.
\]

Write $v = v_{[1,|v|\,]}$.   We next show that $M_{v_{[1,n]}} \cdot (\zeta_v-1) \in (\eta_v - 1, 0)$ for $1 \le n \le |v|$.   For $1 \le n \le |v|$, the characteristic sequence of $M_{v_{[1,n]}} \cdot (\zeta_v-1)$ is $m\, a_{[2j+2,2\ell+1]}\, (a_{[1,2\ell+1]})^\omega$ for some $0 \le j \le \ell$, $1 \le m \le a_{2j+1}$, where $m = a_1$ is excluded when $j = 0$.
In these cases, we show that 
\begin{equation} \label{e:mc}
m\, a_{[2j+2,2\ell+1]}\, (a_{[1,2\ell+1]})^\omega <_{\mathrm{alt}} a'_{[1,\infty)}\,.
\end{equation}
Of course, it suffices to consider $m = a_{2j+1}$ when $j \ge 1$, and $m = a_1-1$ when $j = 0$. 
The case $j = 0$ is settled by $a_1-1 < a_1 = a'_1$ in case $\ell \ge 1$, and by $(a_1-1)\, a_1 <_{\mathrm{alt}} (a_1-1)\, 1 = a'_1\, a'_2$ in case $\ell = 0$.
For $1 \le j \le \ell$, we have $a_{[2j+1,2\ell+1]} \le_{\mathrm{alt}} a_{[1,2\ell-2j+1]}$, thus it only remains to consider the case $a_{[2j+1,2\ell+1]} = a_{[1,2\ell-2j+1]}$.
Since $a_1 \ge 2$, it is not possible that $j = \ell$ and $a_{2\ell+1} = 1$ in this case.
From $a_{[2\ell-2j+2,2\ell+1]} <_{\mathrm{alt}} a_{[1,2j]}$, $1 \le j \le \ell$, we infer that 
\begin{equation} \label{e:cineq}
\begin{array}{cl}a_{[1,2j+1]} \ge_{\mathrm{alt}} a_{[1,2j]}\, 1 \ge_{\mathrm{alt}} a_{[2\ell-2j+2,2\ell]}\, (a_{2\ell+1}-1)\, 1 = a'_{[2\ell-2j+2,2\ell+2]} & \mbox{if}\ a_{2\ell+1} \ge 2, \\[1ex]  a_{[1,2j)} \ge_{\mathrm{alt}} a_{[2\ell-2j+2,2\ell)}\, (a_{2\ell}+1) = a'_{[2\ell-2j+2,2\ell]} & \mbox{if}\ a_{2\ell+1} = 1,\end{array}
\end{equation}
and strict inequality implies that $a_{[2j+1,2\ell+1]} \, (a_{[1,2\ell+1]})^\omega <_{\mathrm{alt}} a'_{[1,\infty)}$.
In particular, this settles the case $j = \ell$.
In case $a_{2\ell+1} \ge 2$, $1 \le j < \ell$, we have $a_{[2j+2,2\ell+1]} <_{\mathrm{alt}} a_{[1,2\ell-2j]} = a'_{[1,2\ell-2j]}$, thus $a_{[2j+1,2\ell+1]}\, a_{[1,2\ell+1]} <_{\mathrm{alt}} a'_{[1,4\ell-2j+2]}$.
In case $a_{2\ell+1} = 1$, $1 \le j < \ell$, we have $a_{[2j,2\ell+1]} <_{\mathrm{alt}} a_{[1,2\ell-2j+2]} = a'_{[1,2\ell-2j+2]}$, thus~\eqref{e:mc} holds in all our cases.
Together with Corollary~\ref{c:altorder}, this yields that, indeed, $M_{v_{[1,n]}} \cdot (\zeta_v-1) \in (\eta_v - 1, 0)$ for $1 \le n \le |v|$.\\

We clearly have $M_{v_{[1,n]}} \cdot (\eta_v-1) < 0$ for $1 \le n < |v|$, and $M_{v^{(+1)}} \cdot (\eta_v-1) = \eta_v - 1$, thus $M_v \cdot (\eta_v-1) = \eta_v$.
Note that $x \mapsto M_a \cdot x$ is order preserving and expanding on $(-1,0)$ for any $a \in \mathscr{A}_-$.
For any $\alpha \in \Gamma_v$, we have therefore $M_{v_{[1,n]}} \cdot (\alpha-1) \in (\eta_v-1, 0)$, $1 \le n < |v|$, and $M_v \cdot (\alpha-1) \in (\eta_v-1, \alpha)$, thus $\underline{b}{}^\alpha_{[1,|v|\,]} = v$, $T_\alpha^n(\alpha-1) > \eta_v-1$ for $1 \le n \le |v|$.
Moreover, we have $\underline{b}{}^\alpha_{[1,|v|\,]} \ne v$ for all $\alpha \ge \eta_v$.

Since the characteristic sequence of $(v^{(+1)})^\omega$ is $a'_{[1,\infty)}$, that of $\widehat{(v^{(+1)})^\omega}$ is $1\, a'_{[1,\infty)}$, and we obtain that $\widehat{(v^{(+1)})^\omega} = (\widehat{v}{}^{(-1)})^\omega$.
By Lemma~\ref{l:conversion2}, we get that
\[
\eta_v = \big\llbracket {}^{(W)}\widehat{(v^{(+1)})^\omega} \big\rrbracket = \big\llbracket {}^{(W)} (\widehat{v}{}^{(-1)})^\omega \big\rrbracket = \big\llbracket v'\, (\widehat{v}{}^{(-1)})^\omega \big\rrbracket\,,
\]
with $v' = v'_{[1,|\widehat{v}|\,]} = {}^{(W)}\widehat{v}{}^{(-1)}$.
The characteristic sequence of $M_{v'_{[1,n]}} \cdot \eta_v$, $1 \le n \le |\widehat{v}|$, is therefore of the form $m\, a'_{[2j+1,\infty)}$ for some $1 \le m \le a'_{2j}$ with $1 \le j \le \ell$ if $a_{2\ell+1} = 1$, $1 \le j \le \ell+1$ if $a_{2\ell+1} \ge 2$.
We show that 
\begin{equation} \label{e:c2}
a'_{[2j,\infty)} \le_{\mathrm{alt}} a'_{[1,\infty)}\,.
\end{equation}
For $1 \le j \le \ell$, \eqref{e:cineq} and $a'_{[1,2\ell)} = a_{[1,2\ell)}$ imply that
\[
\begin{array}{cl}a'_{[2j,2\ell+2]} \le_{\mathrm{alt}} a_{[1,2\ell-2j+3]} = a'_{[1,2\ell-2j+3]},\ a'_{[1,2j)} = a_{[1,2j)} \ge_{\mathrm{alt}} a'_{[2\ell-2j+4,2\ell+2]} & \mbox{if}\ a_{2\ell+1} \ge 2, \\[1ex]  
a'_{[2j,2\ell]} \le_{\mathrm{alt}} a_{[1,2\ell-2j+1]} = a'_{[1,2\ell-2j+1]},\ a'_{[1,2j)} = a_{[1,2j)} \ge_{\mathrm{alt}} a'_{[2\ell-2j+2,2\ell]} & \mbox{if}\ a_{2\ell+1} = 1.\end{array}
\]
In this case, \eqref{e:c2} follows from $a'_{[1,\infty)} = (a'_{[1,2\ell+2]})^\omega$ and $a'_{[1,\infty)} = (a'_{[1,2\ell]})^\omega$ respectively. 
For $j = \ell+1$ (and $a_{2\ell+1} \ge 2$), \eqref{e:c2} is a consequence of $a'_{2\ell+2} = 1 < a'_1$ when $\ell \ge 1$ or $a_1 \ge 3$, and of $a'_{[1,\infty)} = 1^\omega$ when $\ell = 0$ and $a_1 = 2$.
Now, Corollary~\ref{c:altorder} yields that $M_{v'_{[1,n]}} \cdot \eta_v \in [\eta_v - 1, 0)$ for $1 \le n \le |\widehat{v}|$.

The equation
\[
\zeta_v = \big\llbracket {}^{(W)}\,\widehat{\!(v\, \widehat{v})^\omega\!}\,\big\rrbracket = \big\llbracket {}^{(W)} (\widehat{v}\, v)^\omega \big\rrbracket
\]
shows that $M_{v'_{[1,n]}} \cdot \zeta_v < 0$ for $1 \le n < |\widehat{v}|$, and $M_{v'} \cdot \zeta_v = \zeta_v$.
As $x \mapsto M_{v'_1} \cdot x$ is order reversing on $(0,1)$ and $x \mapsto M_{v'_n} \cdot x$ is order preserving on $(-1,0)$ for $2 \le n \le |\widehat{v}|$, we obtain for any $\alpha \in \Gamma_v$ that $M_{v'_{[1,n]}} \cdot \alpha \in (\eta_v-1, 0)$ for $1 \le n < |\widehat{v}|$, $M_{v'} \cdot \alpha \in (\eta_v-1, \alpha)$, thus $\overline{b}{}^\alpha_{[1,|\widehat{v}|\,]} = v'$, and $T_\alpha^n(\alpha) > \eta_v-1$ for $1 \le n \le |\widehat{v}|$.
We also obtain that $\overline{b}{}^\alpha_{[1,|\widehat{v}|\,]} \ne v'$ for all $\alpha \le \zeta_v$, which concludes the proof of the lemma.
\end{proof}

\begin{Lem} \label{l:balpha}
For any $\alpha \in \Gamma$, there exists a unique $v \in \mathscr{F}$ such that $\alpha \in \Gamma_v$.
\end{Lem}

\begin{proof}
Let $\alpha \in \Gamma$, and $a_{[1,\infty)}$ be the characteristic sequence of $\alpha-1$. 
If $T_\alpha(\alpha) \ge 0$, then using \eqref{e:Mn}
\[
0 \ge W \cdot T_\alpha(\alpha) = W M_{(+1:d_\alpha(\alpha))} \cdot \alpha = M_{(-1:2)}^{d_\alpha(\alpha)-1} \cdot (\alpha-1) = T_\alpha^{d_\alpha(\alpha)-1}(\alpha-1)\,,
\] 
thus $\alpha \in \Gamma_{(-1:2)^{d_\alpha(\alpha)-1}}$ by Lemma~\ref{l:Gammav}.
Assume from now on that $T_\alpha(\alpha) < 0$.
Then the characteristic sequence of $T_\alpha(\alpha)$ is $a_{[2,\infty)}$ by Lemma~\ref{l:conversion2}.

If $T_\alpha^n(\alpha-1) \ge 0$ for some $n \ge 1$, and $n$ is minimal with this property, then the by-excess expansion of $\alpha-1$ starts with $\underline{b}{}^\alpha_{[1,n]}{}^{(+1)}$.
Since this word does not end with $(-1:2)$, there exists some $m \ge 1$ such that the characteristic sequence of $\underline{b}{}^\alpha_{[1,n]}{}^{(+1)}$ is $a_{[1,2m]}\, 1$.
Therefore, the characteristic sequence of $\underline{b}{}^\alpha_{[1,n]}$ is $a_{[1,2m)}\, (a_{2m}-1)\, 1$ (if $a_{2m} \ge 2$) or $a_{[1,2m-2]}\, (a_{2m-1}+1)$ (if $a_{2m} = 1$).
Set $m = \infty$ if $T_\alpha^n(\alpha-1) < 0$ for all $n \ge 1$.

Similarly, if $T_\alpha^n(\alpha) \ge 0$ for some $n \ge 2$, and $n$ is minimal with this property, then the $0$-expansion of $T_\alpha(\alpha)$ starts with $\overline{b}{}^\alpha_{[2,n]}{}^{(+1)}$.
Therefore, the characteristic sequence of $\overline{b}{}^\alpha_{[2,n]}{}^{(+1)}$ is $a_{[2,2m'+1]}\, 1$ for some $m' \ge 1$, and that of $\overline{b}{}^\alpha_{[2,n]}$ is $a_{[2,2m']}\, ({a_{2m'+1}-1})\, 1$ (if $a_{2m'+1} \ge 2$) or $a_{[2,2m')}\, (a_{2m'}+1)$ (if $a_{2m'+1} = 1$).
Set $m' = \infty$ if $T_\alpha^n(\alpha) < 0$ for all $n \ge 1$.

Let $v \in \mathscr{A}_-^*$ be the word with characteristic sequence
\[
a'_{[1,2\ell+1]} = \left\{\begin{array}{ll}a_{[1,2m)}\, (a_{2m}-1)\, 1 & \mbox{if}\ m \le m',\ a_{2m} \ge 2, \\[1ex] a_{[1,2m-2]}\, (a_{2m-1}+1) & \mbox{if}\ m \le m',\ a_{2m} = 1, \\[1ex] a_{[1,2m'+1]} & \mbox{if}\ m > m'.\end{array}\right.
\]
We show that~\eqref{e:vvhat} holds. 
Suppose first $m \le m'$.
Then $\underline{b}{}^\alpha_{[1,|v|\,]} = v$ by the definition of~$v$, and the characteristic sequence of $\widehat{v}^{(-1)}$ is~$1\, a_{[1,2m]}$.
Removing the first letter of $\widehat{v}^{(-1)}$ yields a word with characteristic sequence~$a_{[2,2m]}$, and $m \le m'$ implies that $\overline{b}{}^\alpha_{[2,\infty)}$ starts with this word.
By Lemma~\ref{l:conversion2} and since $T_\alpha(\alpha) < 0$, Lemma~\ref{l:digit} shows that $\overline{b}{}^\alpha_1$ is equal to the first letter of~${}^{(W)}\widehat{v}^{(-1)}$.
Therefore, we also have $\overline{b}{}^\alpha_{[1,|\widehat{v}|\,]} = {}^{(W)}\widehat{v}{}^{(-1)}$. 
Suppose now $m > m'$. 
Then $\underline{b}{}^\alpha_{[1,\infty)}$ starts with~$v$.
As for $m \le m'$, the first letter of~${}^{(W)}\widehat{v}^{(-1)}$ is equal to~$\overline{b}{}^\alpha_1$.
Since the characteristic sequence of $\widehat{v}^{(-1)}$ is $1\, a_{[1,2m']}\, ({a_{2m'+1}-1})\, 1$ (if $a_{2m'+1} \ge 2$) or $1\, a_{[1,2m')}\, (a_{2m'}+1)$ (if $a_{2m'+1} = 1$), we obtain that $\overline{b}{}^\alpha_{[1,|\widehat{v}|\,]} = {}^{(W)}\widehat{v}{}^{(-1)}$.
Therefore, $\alpha$ and $v$ satisfy~\eqref{e:vvhat}. 

Next we show that $v \in \mathscr{F}$, i.e., $a'_{[2j,2\ell+1]} <_{\mathrm{alt}} a'_{[1,2\ell-2j+2]}$ and $a'_{[2j+1,2\ell+1]} \le_{\mathrm{alt}} a'_{[1,2\ell-2j+1]}$ for all $1 \le j \le \ell$.
Since the characteristic sequence of $T_\alpha^{1+a_2+a_4+\cdots+a_{2j-2}}(\alpha)$ is $a_{[2j,\infty)}$ for all $1 \le j \le \ell$, we have $a_{[2j,\infty)} \le_{\mathrm{alt}} a_{[1,\infty)}$ by Corollary~\ref{c:altorder}.
If $m \le m'$, this yields that $a'_{[2j,2\ell+1]} <_{\mathrm{alt}} a_{[2j,2\ell+1]} \le_{\mathrm{alt}} a_{[1,2\ell-2j+2]} = a'_{[1,2\ell-2j+2]}$.
If $m > m'$, then we have $a_{[2\ell+2,\infty)} >_{\mathrm{alt}} a_{[1,\infty)}$ because $a_{[2\ell+2,\infty)}$ is the characteristic sequence of $T_\alpha^{|\widehat{v}|}(\alpha) - 1$, thus $a_{[2j,\infty)} \le_{\mathrm{alt}} a_{[1,\infty)}$ implies that $a_{[2j,2\ell+1]} <_{\mathrm{alt}} a_{[1,2\ell-2j+2]}$.
In this case, we obtain that $a'_{[2j,2\ell+1]} = a_{[2j,2\ell+1]} <_{\mathrm{alt}} a_{[1,2\ell-2j+2]} = a'_{[1,2\ell-2j+2]}$.
Consider now $a'_{[2j+1,2\ell+1]}$, $1 \le j \le \ell$. 
If $j = \ell$, $m \le m'$, and $a_{2m} \ge 2$, then $a'_{2\ell+1} = 1 \le a_1'$. 
In all other cases we have $a_{[2j+1,\infty)} \le_{\mathrm{alt}} a_{[1,\infty)}$ because $a_{[2j+1,\infty)}$ is the characteristic sequence of $T_\alpha^{a_1+a_3+\cdots+a_{2j-1}}(\alpha-1)$.
If $m > m'$, then this implies that $a'_{[2j+1,2\ell+1]} = a_{[2j+1,2\ell+1]} \le_{\mathrm{alt}} a_{[1,2\ell-2j+1]} = a'_{[1,2\ell-2j+1]}$.
If $m \le m'$, then the fact that $a_{[2m+1,\infty)}$ is the characteristic sequence of $T_\alpha^{|v|}(\alpha-1) - 1$ implies that $a_{[2j+1,2m]} <_{\mathrm{alt}} a_{[1,2m-2j]}$, thus $a'_{[2j+1,2\ell+1]} \le_{\mathrm{alt}} a_{[1,2\ell-2j+1]} = a'_{[1,2\ell-2j+1]}$.
This proves that $v \in \mathscr{F}$.

By Lemma~\ref{l:Gammav}, we have shown that $\alpha \in \Gamma$ implies that $\alpha \in \Gamma_v$ for some $v \in \mathscr{F}$.
Suppose that $\alpha \in \Gamma_v$ and $\alpha \in \Gamma_w$ for two different $v, w \in \mathscr{F}$. 
Since $\underline{b}{}^\alpha_{[1,|v|\,]} = v$ and $\underline{b}{}^\alpha_{[1,|w|\,]} = w$ by Lemma~\ref{l:Gammav}, $v$ is a prefix of~$w$ or $w$ is a prefix of~$v$. 
Then ${}^{(W)}\widehat{v}^{(-1)}$ is not a prefix of~${}^{(W)}\widehat{w}^{(-1)}$, and ${}^{(W)}\widehat{w}^{(-1)}$ is not a prefix of~${}^{(W)}\widehat{v}^{(-1)}$, thus $\overline{b}{}^\alpha_{[1,|\widehat{v}|\,]} \ne {}^{(W)}\widehat{v}{}^{(-1)}$ or $\overline{b}{}^\alpha_{[1,|\widehat{w}|\,]} \ne {}^{(W)}\widehat{w}{}^{(-1)}$.
Again by Lemma~\ref{l:Gammav}, this implies that $\alpha \not\in \Gamma_v$ or $\alpha \not\in \Gamma_w$ .
Thus $\alpha$ lies in a unique $\alpha \in \Gamma_v$.
\end{proof}

\begin{Lem} \label{l:notinGamma}
We have $\alpha \in (0,1] \setminus \Gamma$ if and only if $\alpha \in (0,g]$ and the characteristic sequence $a_{[1,\infty)}$ of $\alpha-1$~satisfies $a_{[n,\infty)} \le_{\mathrm{alt}} a_{[1,\infty)}$ for all $n \ge 2$.
If $\alpha \in (0,1] \setminus \Gamma$, then $\overline{b}{}^\alpha_{[1,\infty)} = {}^{(W)}\widehat{\underline{b}{}^\alpha_{[1,\infty)}}$.
\end{Lem}

\begin{proof}
We have $\alpha \in (0,1] \setminus \Gamma$ if and only if $\underline{b}{}^\alpha_{[1,\infty)} \in \mathscr{A}_-^\omega$ and $\overline{b}{}^\alpha_{[2,\infty)} \in \mathscr{A}_-^\omega$, which in turn is equivalent to $\underline{b}{}^\alpha_{[1,\infty)} \in \mathscr{A}_-^\omega$ and $\overline{b}{}^\alpha_{[1,\infty)}  = {}^{(W)}\widehat{\underline{b}{}^\alpha_{[1,\infty)}}$ by Lemmas~\ref{l:conversion2} and~\ref{l:digit}.
Since $\Gamma_v = (g,1]$ for the empty word~$v$, we have $(0,1] \setminus \Gamma \subset (0,g]$.

Let $\alpha \in (0,g]$, and $a_{[1,\infty)}$ be the characteristic sequence of~$\alpha-1$.
If $\underline{b}{}^\alpha_{[1,\infty)} \in \mathscr{A}_-^\omega$, then $a_{[2j+1,\infty)}$ is the characteristic sequence of $T_\alpha^{a_1+a_3+\cdots+a_{2j-1}}(\alpha-1)$ for all $j \ge 0$, thus $a_{[2j+1,\infty)} \le_{\mathrm{alt}} a_{[1,\infty)}$ by Corollary~\ref{c:altorder}.
If moreover $\overline{b}{}^\alpha_{[1,\infty)}  = {}^{(W)}\widehat{\underline{b}{}^\alpha_{[1,\infty)}}$, then $a_{[2j,\infty)}$ is the characteristic sequence of $T_\alpha^{1+a_2+a_4+\cdots+a_{2j-2}}(\alpha-1)$ for all $j \ge 1$, thus $a_{[2j,\infty)} \le_{\mathrm{alt}} a_{[1,\infty)}$.

On the other hand, if $a_{[2j+1,\infty)} \le_{\mathrm{alt}} a_{[1,\infty)}$ for all $j \ge 1$, then we obtain recursively from Corollary~\ref{c:altorder}, Lemma~\ref{l:digit} and the fact that $M_{(-1:2)}$ is increasing on $(-1,0)$ that $\underline{b}{}^\alpha_{[a_1+a_3+\cdots+a_{2j-3}+1, a_1+a_3+\cdots+a_{2j-1}]} = (-1:2)^{a_{2j-1}-1} (-1:2+a_{2j})$, thus $\underline{b}{}^\alpha_{[1,\infty)} \in \mathscr{A}_-^\omega$.
If moreover $a_{[2j,\infty)} \le_{\mathrm{alt}} a_{[1,\infty)}$ holds for all $j \ge 1$, then we obtain in the same way that $\underline{b}{}^\alpha_{[a_2+a_4+\cdots+a_{2j-2}+2, a_2+a_4+\cdots+a_{2j}+1]} = (-1:2)^{a_{2j}-1} (-1:2+a_{2j+1})$ for all $j \ge 1$, thus $\overline{b}{}^\alpha_{[2,\infty)} \in \mathscr{A}_-^\omega$.
This implies $\alpha \in (0,1] \setminus \Gamma$. 
\end{proof}

By the following lemma, the orbits of ${\alpha-1}$ and $\alpha$ synchronize for almost all $\alpha \in (0,1]$.

\begin{Lem} \label{l:zeromeasure}
The set $(0,1] \setminus \Gamma$ has zero Lebesgue measure.
\end{Lem}

\begin{proof}
By Lemma~\ref{l:notinGamma} and Proposition~\ref{p:0toRCF}, we have 
\begin{align*}
(0,1] \setminus \Gamma & = \big\{\alpha \in (0,1] \mid T_1^n(\alpha) \ge \alpha\ \mbox{for all}\ n \ge 1\big\} \\
& \subset \bigcup_{d\ge1} \big\{\alpha \in [1/d, 1] \mid T_1^n(\alpha) \ge 1/d\ \mbox{for all}\ n \ge 1\big\}\,.
\end{align*}
Since $T_1$ is ergodic, this set is the countable union of null sets.
\end{proof}

We remark that Lemma~\ref{l:zeromeasure} was also proved in~\cite{Carminati-Tiozzo:10}.
Furthermore, they showed that the Hausdorff measure of $(0,1] \setminus \Gamma$ is~$1$.

Putting everything together, we obtain the main result of this section.

\begin{proof}[\textbf{Proof of Theorem~\ref{t:endpoints}}]
Lemma~\ref{l:balpha} shows that $\Gamma$ is the disjoint union of the intervals~$\Gamma_v$, $v \in \mathscr{F}$.
For any $\alpha \in \Gamma_v$, we have $\underline{b}{}^\alpha_{[1,|v|\,]} = v$ and $\overline{b}{}^\alpha_{[1,|\widehat{v}|\,]} = {}^{(W)}\widehat{v}{}^{(-1)}$ by Lemma~\ref{l:Gammav}, thus $T_\alpha^{|v|}({\alpha-1}) = W \cdot T_\alpha^{|\widehat{v}|}(\alpha)$ by Lemma~\ref{l:conversion2}. 
Then Lemma~\ref{l:wMat} gives $T_\alpha^{|v|+1}({\alpha-1}) = T_\alpha^{|\widehat{v}|+1}(\alpha)$ and $\overline{b}{}^\alpha_{|\widehat{v}|+1}  = {}^{(W)}\underline{b}{}^\alpha_{|v|+1}$.
The statements on $(0,1] \setminus \Gamma$ are shown in Lemmas~\ref{l:notinGamma} and~\ref{l:zeromeasure}.
\end{proof}

\begin{Rmk} \label{r:RvLv}
Let $L_v := M_v E$ and $R_v := E^{-1} M_{\widehat{v}} W$.   Then Lemma~\ref{l:conversion}  implies that for each $v \in \mathscr{F}$,  $L_v = W R_v$ and Theorem~\ref{t:endpoints} implies that the graphs of $L_v \cdot x$ and $R_v \cdot x$ cross above $(\zeta_v, \eta_v)$, with common zero at~$\chi_v$.   See Figure~\ref{f:rich}.
\end{Rmk}

\begin{figure}[ht]
\centering
\includegraphics{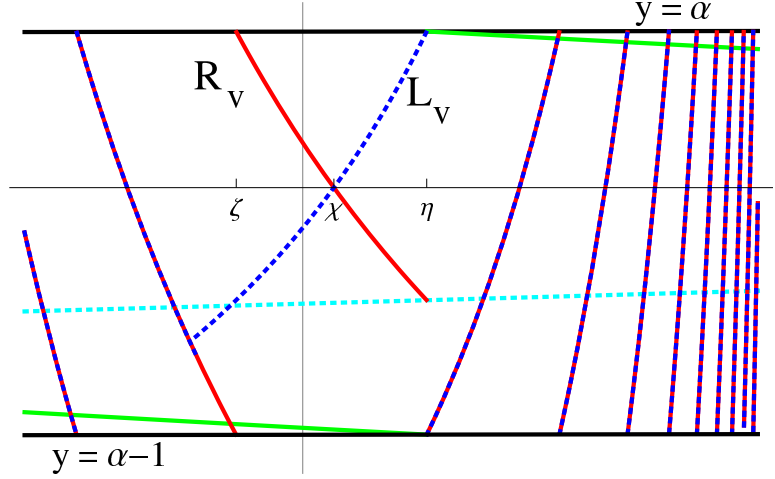}
\caption{Graphs of $\alpha \mapsto T_{\alpha}^4(\alpha)$ in solid red, $\alpha \mapsto T_{\alpha}^4({\alpha-1})$ in dotted blue, $\alpha \mapsto T_{\alpha}^2(\alpha)$ in solid green, and $\alpha \mapsto T_{\alpha}^2({\alpha-1})$ in dotted cyan,    near $\Gamma_v$ for $v = (-1:2)(-1:3)(-1:4)(-1:2) = \Theta\big(({-1:2})({-1:3})\big)$. 
On $\Gamma_v = (0.3867\dots, 0.3874\dots)$,  $R_v \cdot \alpha$ and $L_v \cdot \alpha$  agree with $T_{\alpha}^4(\alpha)$ and $T_{\alpha}^4({\alpha-1})$, respectively; have  a common zero at $\chi = \chi_v$;  and,  meet the graph of the identity function at $\zeta= \zeta_v$ and $\eta = \eta_v$, respectively.     To aid comparison with Figure~\ref{f:2342}, gridline is at $\alpha = 113/292$. 
For $\alpha \in \Gamma_{(-1:2)(-1:3)} = (0.3874\dots, 0.4142\dots)$, one has $T_{\alpha}^4(\alpha) = T_{\alpha}^4({\alpha-1})$  --- compare with Figure~\ref{f:23}, whereas  to the left of $\zeta = \zeta_v$, one sees that there is a gap before once again these agree. 
The transcendental $\tau_v$ lies in this gap.}
\label{f:rich}
\end{figure}

We give examples realizing some of the various cases that arise in Lemma~\ref{l:balpha}.

\begin{Ex}
If $\alpha = 1/r$ for some positive integer~$r$, then $\underline{b}{}^\alpha_{[1,r)} = (-1:2)^{r-1}$ and $\overline{b}{}^\alpha_1 = (+1:r)$ imply that $\alpha \in \Gamma_{(-1:2)^{r-1}}$.
We have $T_\alpha^{r-1}(\alpha) = 0 = T_\alpha(\alpha)$.
\end{Ex}

\begin{Ex}
If $\alpha = 37/97$, then $\underline{b}{}^\alpha_{[1,4]} = (-1:2) (-1:3) (-1:3) (-1:2)$ and $T_\alpha^4({\alpha-1}) = 1/4$.
The characteristic sequence of $v = \underline{b}{}^\alpha_{[1,4]}$ is $21112$, thus $\widehat{v} = (-1:4) (-1:3) (-1:4)$.
Since $\overline{b}{}^\alpha_{[1,3]} = (+1:3) (-1:3) (-1:3) = {}^{(W)}\widehat{v}^{(-1)}$, we have $\alpha \in \Gamma_v$.
Note that $T_\alpha^3(\alpha) = -1/5 = W \cdot T_\alpha^4({\alpha-1})$, $T_\alpha^5(\alpha-1) = 0 = T_\alpha^4(\alpha)$, and that $|v| = 4 > 3 = |\widehat{v}|$.
\end{Ex}

\begin{Ex}
If $\alpha = 58/195$, then $\underline{b}{}^\alpha_{[1,5]} = (-1:2) (-1:2) (-1:4) (-1:4) (-1:5)$ and $T_\alpha^5({\alpha-1}) = 0$, $\overline{b}{}^\alpha_{[1,5]} = (+1:4) (-1:2) (-1:3) (-1:2) (-1:2)$ and $T_\alpha^5(\alpha) = 1/4$.
The characteristic sequence of $\underline{b}{}^\alpha_{[1,5]}{}^{(+1)}$ is $3212141$, and that of $\overline{b}{}^\alpha_{[2,5]}{}^{(+1)}$ is $21211$. 
This yields that $m = 3$, $m' = 2$ in Lemma~\ref{l:balpha}, hence $\alpha \in \Gamma_v$, where $v = (-1:2) ({-1:2}) (-1:4) (-1:4)$ has the characteristic sequence $32121$.
Since $\widehat{v} = (-1:5) (-1:2) (-1:3) (-1:2) (-1:3)$, we have $|v| = 4 < 5 = |\widehat{v}|$.
\end{Ex}

Theorem~\ref{t:tauTransc} shows that $v \in \mathscr{F}$ for which $|v| = |\widehat{v}|$ abound.
In the following examples, we exhibit families of words showing that strict inequality (in each direction) also arises infinitely often.  
Note that \cite{NakadaNatsui08} also give infinite families realizing each of the three types of behavior.    

\begin{Ex} \label{ex:infFamInF}
Let $v = (-1:2)^m\, (-1:3)^{\ell}\, (-1:2)$ for some positive integers $m$ and~$\ell$.
Then the characteristic sequence of~$v$ is $a_{[1,2\ell+1]} = (m+1)\, 1^{2\ell-1}\, 2$, thus $v \in \mathscr{F}$.
Since $\widehat{v} = (-1:3+m)\, (-1:3)^{\ell-1}\, (-1:4)$, we have $|v| - |\widehat{v}| = m$.
\end{Ex}

\begin{Ex}\label{ex:twoTwoFourFour}
Let $v = (-1:2)^{m+1}\, (-1:4)^{\ell}$ for some positive integers $m$ and~$\ell$.
Then the characteristic sequence is $a_{[1,2\ell+1]} = (m+2)\, (2\,1)^{\ell}$, thus $\widehat{v} = (-1:4+m)\, \big((-1:2)\,(-1:3)\big)^\ell$ and $|\widehat{v}| - |v| = \ell - m$.
Again, membership of $v$ in $\mathscr{F}$ follows trivially.
\end{Ex}

In the central range $[g^2, g]$, however, we always have equality $|v| = |\widehat{v}|$.

\begin{Lem} \label{l:longConstIntvl}
For any $v \in \mathscr{F}$ with $\Gamma_v \subset [g^2, g]$, we have $|v| = |\widehat{v}|$.
\end{Lem}

\begin{proof}
Let $a_{[1,2\ell+1]}$ be the characteristic sequence of $v \in \mathscr{F}$ with $\Gamma_v \subset [g^2, g]$.
Then we have $a_1 = 2$ because $\underline{b}{}^\alpha_1 = (-1:2)$ and $\underline{b}{}^\alpha_2 \neq (-1:2)$ for each $\alpha \in [g^2,g)$.
This implies that $a_{[1,2\ell+1]} \in \{1,2\}^*$.
Since $\alpha-1 \ge -g$ and the characteristic sequence of $-g$ is $2\,1^\omega$, Corollary~\ref{c:altorder} yields that $a_{[1,2j+1]} \ne 2\,1^{2j-1}\,2$ for all $1 \le j \le \ell$.
Therefore, the number of~$1$s between any two $2$s in $a_{[1,2\ell+1]}$ is even.
Moreover, $a_{[2j,2\ell+1]} = 2\,1^{2\ell-2j+1}$ is impossible for $1 \le j \le \ell$.
Since $a_{[1,2\ell+1]}$ is of odd length, we obtain that $a_{[1,2\ell+1]} \in 2 (11)^* \,(2 (11)^* 2 (11)^*)^*$.
We have $|v| = \sum_{j=0}^\ell a_{2j+1} - 1$ and $|\widehat{v}| = \sum_{j=1}^\ell a_{2j} + 1$, thus $|v| = |\widehat{v}|$.
\end{proof}

Immediately to the right of $[g^2,g]$ lies the interval $\Gamma_v = (g,1]$ with the empty word~$v$, where $|v| = 0 < 1 = |\widehat{v}|$.
Example~\ref{ex:infFamInF} (with $m=1$) provides intervals~$\Gamma_v$ arbitrarily close to the left of $[g^2,g]$ with $|v| > |\widehat{v}|$.
The following example shows that the opposite inequality also occurs arbitrarily close to the left of $[g^2,g]$.

\begin{Ex}\label{ex:closetog}
Let  $m$ be a positive integer and set 
\[
v = (-1:2)\, (-1:3)^m\, (-1:2)\, (-1:4)\, (-1:3)^m\, (-1:4)\, (-1:3)^m\, (-1:4)\, (-1:2)\,.
\] 
Then the characteristic sequence is $a_{[1,6m+7]} = 21^{2m-1}221^{2m+1}21^{2m+1}22$, thus $v \in \mathscr{F}$, and 
\[
\widehat{v} = (-1:4)\, (-1:3)^{m-1}\, (-1:4)\, (-1:2)\, (-1:3)^{m+1}\, (-1:2)\, (-1:3)^{m+1}\, (-1:2)\, (-1:4)
\] 
shows that $|\widehat{v}| = |v| + 1$.
\end{Ex}

\section{Structure of the natural extension domains} \label{sec:struct-natur-extens}

For an explicit description of $\Omega_{\alpha}$,  we require  detailed knowledge of the effects of $\mathcal T_{\alpha}$ on the regions fibered above non-full cylinders determined by the $T_{\alpha}$-orbits of ${\alpha-1}$ and $\alpha$. 
To this end,  we use the languages $\mathscr{L}_\alpha$ and~$\mathscr{L}_\alpha'$ defined in Section~\ref{s:justTheFacts}.
Throughout the section, let 
\[
k = \left\{\begin{array}{cl}|v| + 1 & \mbox{if}\ \alpha \in \Gamma_v,\, v \in \mathscr{F}, \\ \infty & \mbox{if}\ \alpha \in (0,1] \setminus \Gamma,\end{array}\right. \quad k' = \left\{\begin{array}{cl}|\widehat{v}| + 1 & \mbox{if}\ \alpha \in \Gamma_v,\, v \in \mathscr{F}, \\ \infty & \mbox{if}\ \alpha \in (0,1] \setminus \Gamma.\end{array}\right.
\]
We make use of the extended languages $\mathscr{L}_\alpha^\times$ and~$\mathscr{L}_\alpha^{\prime\times}$, defined by
\[
\mathscr{L}^\times_\alpha := \big(\widetilde{U}_{\alpha,3} \cup U_{\alpha,1}\, U_{\alpha,2}^*\, \widetilde{U}_{\alpha,4}\big)^*\,, \quad \mathscr{L}^{\prime\times}_\alpha := \mathscr{L}^\times_\alpha\, U_{\alpha,1}\, U_{\alpha,2}^*\,,
\]
where $U_{\alpha,1} := \big\{\underline{b}{}^\alpha_{[1,j]} \mid 0 \le j < k\big\}$, $U_{\alpha,2} := \big\{\overline{b}{}^\alpha_{[1,j]} \mid 1 \le j < k'\big\}$ as in Section~\ref{s:justTheFacts}, and
\begin{gather*}
\begin{aligned}
\widetilde{U}_{\alpha,3} & := \big\{\underline{b}{}^\alpha_{[1,j)}\, a \mid j \ge 1,\, a \in \mathscr{A},\, \underline{b}{}^\alpha_j \prec a \prec \overline{b}{}^\alpha_1\big\} \\
\widetilde{U}_{\alpha,4} & := \big\{\overline{b}{}^\alpha_{[1,j)} \, a \mid j \ge 2,\, a \in \mathscr{A},\, \overline{b}{}^\alpha_j \prec a \prec \overline{b}{}^\alpha_1\big\}
\end{aligned}
\quad \mbox{if}\ \alpha \in (0,1] \setminus \Gamma, \\[1ex]
\begin{aligned}
\widetilde{U}_{\alpha,3} & := \big\{\underline{b}{}^\alpha_{[1,j)}\, a \mid 1 \le j < k,\, a \in \mathscr{A},\, \underline{b}{}^\alpha_j \prec a \prec \overline{b}{}^\alpha_1\big\} \cup \big\{\underline{b}{}^\alpha_{[1,k)}\, a \mid a \in \mathscr{A}_+,\, a \prec \overline{b}{}^\alpha_1\big\} \\
\widetilde{U}_{\alpha,4} & := \big\{\overline{b}{}^\alpha_{[1,j)} \, a \mid 2 \le j < k',\, a \in \mathscr{A},\, \overline{b}{}^\alpha_j \prec a \prec \overline{b}{}^\alpha_1\big\} \cup \big\{\overline{b}{}^\alpha_{[1,k')}\, a \mid a \in \mathscr{A}_+,\, a \prec \overline{b}{}^\alpha_1\big\}
\end{aligned}
\quad \mbox{if}\ \alpha \in \Gamma.
\end{gather*}
Let
\[
\Psi^\times_\alpha := \big\{N_w \cdot 0 \mid w \in \mathscr{L}^\times_\alpha\big\} \quad \mbox{and} \quad \Psi^{\prime\times}_\alpha := \big\{N_w \cdot 0 \mid w \in \mathscr{L}^{\prime\times}_\alpha\big\}\,.
\]

The languages introduced above  allow us to view the region $\Omega_{\alpha}$ as being the union of pieces,  each of which fibers over a subinterval whose left endpoint is in the $T_{\alpha}$-orbit of $\alpha$ or of ${\alpha-1}$.  
We will see in Lemma~\ref{l:language} that $\mathscr{L}^{\prime\times}_\alpha$ is the language of the $\alpha$-expansions avoiding $(+1:\infty)$ if either $\alpha \in (0,1] \setminus \Gamma$ or $T_\alpha^{k-1}({\alpha-1}) = T_\alpha^{k'-1}(\alpha) = 0$.
For other~$\alpha$, $\mathscr{L}^{\prime\times}_\alpha$~is slightly different from the language of the $\alpha$-expansions.   However,   any $\alpha \in \Gamma$ lies in some $\Gamma_v$ and hence shares various properties with $\chi_v$.  We thus can exploit the fact that   $T_\alpha^{k-1}({\chi_v}) = T_\alpha^{k'-1}(\chi_v) = 0$ to aid in the description of~$\Omega_\alpha$.

From their definitions, we clearly have   $\Psi^\times_\alpha \subset \Psi^{\prime\times}_\alpha$.  Using these languages,  we describe $\Omega_{\alpha}$ in terms of its fibering over $\mathbb I_{\alpha}$.   For example,   Corollary ~\ref{c:increasing} shows that the fiber in $\Omega_{\alpha}$ above any $x \in \mathbb I_{\alpha}$ is squeezed between  the closures of $\Psi^\times_\alpha$ and $\Psi^{\prime\times}_\alpha$.  Thus,   $\mathbb I_{\alpha} \times \overline{\Psi^\times_\alpha} \subseteq \Omega_{\alpha} \subseteq \mathbb I_{\alpha} \times \overline{\Psi^{\prime\times}_\alpha}$.   Note also that Lemma~\ref{l:Psialpha} shows that $\Psi_{\alpha} = \overline{\Psi^\times_\alpha}$  and $\Psi'_{\alpha} = \overline{\Psi^{\prime \times}_\alpha}$.

\begin{Prop} \label{p:Omega}
Let $\alpha \in (0,1]$.
Then we have
\begin{multline} \label{e:TnPsi}
\bigcup_{n \ge 0} \mathcal{T}_\alpha^n\big([{\alpha-1},\alpha) \times \{0\}\big) \\
= [{\alpha-1}, \alpha) \times \Psi^\times_\alpha \ \cup \bigcup_{1 \le j < k} \big[T_\alpha^j({\alpha-1}), \alpha\big) \times N_{\underline{b}{}^\alpha_{[1,j]}} \cdot \Psi^\times_\alpha \ \cup \bigcup_{1 \le j < k'} \big(T_\alpha^j(\alpha), \alpha\big) \times N_{\overline{b}{}^\alpha_{[1,j]}} \cdot \Psi^{\prime\times}_\alpha\,. 
\end{multline}
\end{Prop}

Here, $(x,x')$ denotes the open interval between $x$ and~$x'$  (and not a point in~$\mathbb{R}^2$), and the map $\mathcal{T}_\alpha$ always acts on products of two sets in~$\mathbb{R}$.

The following lemmas are used in the proof of the proposition.

\begin{Lem} \label{l:Ldecomposition}
For any $\alpha \in (0,1]$, $\mathscr{L}^{\prime\times}_\alpha$~admits the partition
\[
\mathscr{L}^{\prime\times}_\alpha = \mathscr{L}^\times_\alpha\ \cup \bigcup_{1 \le j < k} \mathscr{L}^\times_\alpha\, \underline{b}{}^\alpha_{[1,j]}\ \cup \bigcup_{1 \le j < k'} \mathscr{L}^{\prime\times}_\alpha\, \overline{b}{}^\alpha_{[1,j]}\,.
\]
\end{Lem}

\begin{proof}
In the factorization $\mathscr{L}^{\prime\times}_\alpha = \mathscr{L}^\times_\alpha\, U_{\alpha,1}\, U_{\alpha,2}^*$, there are two cases: the exponent of $U_{\alpha,2}$ being zero or not.
In the first case, the element of $U_{\alpha,1}$ can be the empty word $\underline{b}{}^\alpha_{[1,0]}$, which gives~$\mathscr{L}^\times_\alpha$, or a word $\underline{b}{}^\alpha_{[1,j]}$, $1 \le j < k$.
In the second case,  we can factor exactly one power of $U_{\alpha,2}$ to the right.
Since the decomposition of every $w \in \mathscr{L}^{\prime\times}_\alpha$ into factors in $\mathscr{L}^\times_\alpha, U_{\alpha,1}, U_{\alpha,2}^*$ (in this order) is unique, this proves the lemma.
\end{proof}

By Lemma~\ref{l:Ldecomposition}, we can write (\ref{e:TnPsi}) as
\begin{equation} \label{e:PrimeForm}
\bigcup_{n \ge 0} \mathcal{T}_\alpha^n\big([{\alpha-1},\alpha) \times \{0\}\big) = \bigcup_{w \in \mathscr{L}^{\prime\times}_\alpha} J^\alpha_w \times \{N_w \cdot 0\}\,,
\end{equation}
where
\[
J^\alpha_w := \left\{\begin{array}{cl}[{\alpha-1}, \alpha) & \mbox{if}\ w \in \mathscr{L}^\times_\alpha, \\[1ex] \big[T_\alpha^j({\alpha-1}), \alpha\big) & \mbox{if}\ w \in \mathscr{L}^\times_\alpha\, \underline{b}{}^\alpha_{[1,j]},\, 1 \le j < k\,, \\[1ex]  \big(T_\alpha^j(\alpha), \alpha\big) & \mbox{if}\ w \in \mathscr{L}^{\prime\times}_\alpha\, \overline{b}{}^\alpha_{[1,j]},\, 1 \le j < k'\,.\end{array}\right.
\]
From now on, denote by $\Delta_\alpha(w)$, $w \in \mathscr{A}^*$, the set of $x \in [{\alpha-1}, \alpha)$ with $\alpha$-expansion starting with~$w$.
This only differs from previous definitions in that $\Delta_\alpha(w)$ never contains the point~$\alpha$.   

\begin{Lem} \label{l:Jw}
Let $\alpha \in (0,1] \setminus \Gamma$ or $\alpha = \chi_v$, $v \in \mathscr{F}$.
Then 
\[
J^\alpha_w = T_\alpha^{|w|}\big(\Delta_\alpha(w)\big) = M_w \cdot \Delta_\alpha(w) \quad \mbox{for all}\ w \in \mathscr{L}^{\prime\times}_\alpha\,. 
\]
\end{Lem}

\begin{proof}
The second equality follows immediately from the definitions.

The first equality clearly holds if $w$ is the empty word.
We proceed by induction on~$|w|$.
The definition of $\mathscr{L}^{\prime\times}_\alpha$ implies that every $w' \in \mathscr{L}^{\prime\times}_\alpha$ with $|w'| \ge 1$ can be written as $w' = w a$ with $w \in \mathscr{L}^{\prime\times}_\alpha$, $a \in \mathscr{A}$.
Let first $w \in \mathscr{L}^\times_\alpha\, \underline{b}{}^\alpha_{[1,j)}$, $1 \le j < k$, which implies $\underline{b}{}^\alpha_j \preceq a \preceq \overline{b}{}^\alpha_1$.
Since $T_\alpha^{j-1}({\alpha-1}) < 0$, we have
\[
J^\alpha_w = \big[T_\alpha^{j-1}({\alpha-1}), \alpha\big) = \big[T_\alpha^{j-1}({\alpha-1}), \tfrac{-1}{d_{\alpha,j}({\alpha-1})+\alpha}\big)\ \cup \bigcup_{\substack{a \in \mathscr{A}:\\ \underline{b}{}^\alpha_j \prec a \preceq \overline{b}{}^\alpha_1}} \Delta_{\alpha}(a)\ \cup \{0\}\,.
\]
Then, $J^\alpha_w = T_\alpha^{|w|}(\Delta_\alpha(w))$ implies that 
\begin{align*}
T_\alpha^{|w|+1}\big(\Delta_\alpha(w \underline{b}{}^\alpha_j)\big) & = T_\alpha\big(J^\alpha_w \cap \Delta_\alpha(\underline{b}{}^\alpha_j)\big) = \big[T_\alpha^j({\alpha-1}), \alpha\big) = J_{w \underline{b}{}^\alpha_j}\,, \\
T_\alpha^{|w|+1}\big(\Delta_\alpha(w a)\big) & = [{\alpha-1}, \alpha) = J_{w a} \quad (\underline{b}{}^\alpha_j \prec a \prec \overline{b}{}^\alpha_1)\,, \\ T_\alpha^{|w|+1}\big(\Delta_\alpha(w \overline{b}{}^\alpha_1)\big) & = \big(T_\alpha(\alpha), \alpha\big) = J_{w \overline{b}{}^\alpha_1}\,.
\end{align*}
If $w \in \mathscr{L}^{\prime\times}_\alpha\, \overline{b}{}^\alpha_{[1,j)}$, $2 \le j < k'$, then similar arguments yield that $T_\alpha^{|w|+1}(\Delta_\alpha(w a)) = J_{w a}$ for $\overline{b}{}^\alpha_j \preceq a \preceq \overline{b}{}^\alpha_1$.
Finally, if $w \in \mathscr{L}^\times_\alpha\, \underline{b}{}^\alpha_{[1,k)}$ or $w \in \mathscr{L}^{\prime\times}_\alpha\, \overline{b}{}^\alpha_{[1,k')}$ (which is possible only for $\alpha \in \Gamma$), then $\alpha = \chi_v$ yields that $J^\alpha_w = [0, \alpha)$ and $J^\alpha_w = (0, \alpha)$ respectively.  Here, $w a \in \mathscr{L}^{\prime\times}_\alpha$ is equivalent to $a \in \mathscr{A}_+$, $a \preceq \overline{b}{}^\alpha_1$, and we obtain again that $T_\alpha^{|w|+1}(\Delta_\alpha(w a)) = J_{w a}$.
\end{proof}

\begin{Lem} \label{l:language}
Let $\alpha \in (0,1] \setminus \Gamma$ or $\alpha = \chi_v$, $v \in \mathscr{F}$.
Then 
\begin{equation} \label{e:cylinders}
[{\alpha-1}, \alpha) = \bigcup_{w \in \mathscr{L}^{\prime\times}_\alpha \cap \mathscr{A}^n} \Delta_\alpha(w)\ \cup\ \{x \in [{\alpha-1}, \alpha) \mid T_\alpha^{n-1}(x) = 0\}
\end{equation}
for all $n \ge 1$, i.e., $\mathscr{L}^{\prime\times}_\alpha$ is the language of the $\alpha$-expansions of $x \in [\alpha-1, \alpha)$ avoiding~$(1:\infty)$.
\end{Lem}

\begin{proof}
We have $\mathscr{L}^{\prime\times}_\alpha \cap \mathscr{A} = \{a \in \mathscr{A} \mid \underline{b}{}^\alpha_1 \preceq a \preceq \overline{b}{}^\alpha_1\}$, thus (\ref{e:cylinders}) holds for $n = 1$.
In the proof of Lemma~\ref{l:Jw}, we have seen that
\[
T_\alpha^{|w|}\big(\Delta_\alpha(w)\big) \setminus \{0\} = \bigcup_{\substack{a \in \mathscr{A}:\\ w a \in \mathscr{L}^{\prime\times}_\alpha}} T_\alpha^{|w|}\big(\Delta_\alpha(w a)\big)
\]
for all $w \in \mathscr{L}^{\prime\times}_\alpha$.
By applying $M_w^{-1}$, we obtain the corresponding subdivision of $\Delta_\alpha(w)$, which yields inductively (\ref{e:cylinders}) for all $n \ge 1$.
\end{proof}

Lemmas~\ref{l:Jw} and~\ref{l:language} show that~\eqref{e:PrimeForm} and thus~\eqref{e:TnPsi} hold if $\alpha \in (0,1] \setminus \Gamma$ or $\alpha = \chi_v$, $v \in \mathscr{F}$.
For general $\alpha \in \Gamma$, note that
\begin{equation} \label{e:Tn1}
\mathcal{T}_\alpha\big(J^\alpha_w \times \{N_w \cdot 0\}\big) = \{0\} \times \{0\}\ \cup \bigcup_{\substack{a \in \mathscr{A}:\\ w a \in \mathscr{L}^{\prime\times}_\alpha}} J^\alpha_{w a} \times \{N_{w a} \cdot 0\}
\end{equation}
holds for all $w \in \mathscr{L}^{\prime\times}_\alpha \setminus \big(\mathscr{L}^\times_\alpha\, \underline{b}{}^\alpha_{[1,k)} \cup \mathscr{L}^{\prime\times}_\alpha\, \overline{b}{}^\alpha_{[1,k')}\big)$, by arguments similar to the proof of Lemma~\ref{l:Jw}.
For $w \in \mathscr{L}^\times_\alpha\, \underline{b}{}^\alpha_{[1,k)} \cup \mathscr{L}^{\prime\times}_\alpha\, \overline{b}{}^\alpha_{[1,k')}$, we use the following two lemmas.

\begin{Lem} \label{l:LLprime}
Let $\alpha \in (0,1]$, $w \in \mathscr{A}^*$ with $|w| \ge 1$.
Then the membership of $w$ in $\mathscr{L}^\times_\alpha$ is equivalent to that of $w^{(-1)}$ in $\mathscr{L}^{\prime\times}_\alpha$. 
Furthermore, we have $\Psi^{\prime\times}_{\alpha} = {}^t\hspace{-.1em}E \cdot \Psi^\times_{\alpha}$.
\end{Lem}

\begin{proof}
The equivalence between $w \in \mathscr{L}^\times_\alpha$ and $w^{(-1)} \in \mathscr{L}^{\prime\times}_\alpha$ follows directly from the definition of $\mathscr{L}^\times_\alpha$ and~$\mathscr{L}^{\prime\times}_\alpha$. 
Then we find that 
\[
\Psi^{\prime\times}_{\alpha} = \{N_w \cdot 0 \mid w \in \mathscr{L}^{\prime\times}_\alpha\} = \{N_{w^{(-1)}} \cdot 0 \mid w \in \mathscr{L}^\times_\alpha\} = \{ {}^t\hspace{-.1em}E\, N_w\cdot 0 \mid w \in \mathscr{L}^\times_\alpha\} = {}^t\hspace{-.1em}E  \cdot \Psi^\times_{\alpha}\,. \qedhere
\]
\end{proof}

\begin{Lem} \label{l:wwprime}
Let $\alpha \in \Gamma$ and $w = \underline{b}{}^\alpha_{[1,k)}$, $w' = \overline{b}{}^\alpha_{[1,k')}$, or $w = u\, \underline{b}{}^\alpha_{[1,k)}$, $w' = u{}^{(-1)}\, \overline{b}{}^\alpha_{[1,k')}$ with $u \in \mathscr{L}^\times_\alpha$, $|u| \ge 1$.
Then we have $w, w' \in \mathscr{L}^{\prime\times}_\alpha$ and
\begin{multline} \label{e:Tn2}
\mathcal{T}_\alpha\big(J^\alpha_w \times \{N_w \cdot 0\}\big)\ \cup\ \mathcal{T}_\alpha\big(J^\alpha_{w'} \times \{N_{w'} \cdot 0\}\big) \\
= \{0\} \times \{0\}\ \cup \bigcup_{\substack{a \in \mathscr{A}:\\ w a \in \mathscr{L}^{\prime\times}_\alpha}} J^\alpha_{w a} \times \{N_{w a} \cdot 0\}\ \cup \bigcup_{\substack{a \in \mathscr{A}:\\ w' a \in \mathscr{L}^{\prime\times}_\alpha}} J^\alpha_{w' a} \times \{N_{w' a} \cdot 0\}\,. 
\end{multline}
\end{Lem}

\begin{proof}
By Theorem~\ref{t:endpoints}, we have $\sgn(T_\alpha^{k-1}({\alpha-1})) = - \sgn(T_\alpha^{k'-1}(\alpha))$.
We can assume that $T_\alpha^{k-1}({\alpha-1}) < 0$, the case $T_\alpha^{k'-1}(\alpha) < 0$ being symmetric, and the case $T_\alpha^{k-1}({\alpha-1}) = 0$ being trivial since~\eqref{e:Tn1} holds for $w$ and $w'$ in this case (except for the point $\{0\} \times \{0\}$ not belonging to $\mathcal{T}_\alpha(J^\alpha_{w'} \times \{N_{w'} \cdot 0\})$).
Then
\begin{multline*}
\mathcal{T}_\alpha\big(J^\alpha_w \times \{N_w \cdot 0\}\big) = \big[T_\alpha^k({\alpha-1}), \alpha\big) \times \{N_{w \underline{b}{}^\alpha_k} \cdot 0\}\ \cup\ \{0\} \times \{0\}  \\
\cup\ [{\alpha-1}, \alpha) \times \big\{N_{w a} \cdot 0 \mid \underline{b}{}^\alpha_k \prec a \prec \overline{b}{}^\alpha_1\}\ \cup\ \big(T_\alpha(\alpha), \alpha\big) \times \{N_{w \overline{b}{}^\alpha_1} \cdot 0\}
\end{multline*}
and, if $\overline{b}{}^\alpha_{k'} \prec \overline{b}{}^\alpha_1$, 
\begin{multline*}
\mathcal{T}_\alpha\big(J^\alpha_{w'} \times \{N_{w'} \cdot 0\}\big) = \big[{\alpha-1}, T_\alpha^{k'}(\alpha)\big) \times \{N_{w' \overline{b}{}^\alpha_{k^{{}_\prime}}} \cdot 0\}  \\
\cup\ [{\alpha-1}, \alpha) \times \big\{N_{w' a} \cdot 0 \mid \overline{b}{}^\alpha_{k'} \prec a \prec \overline{b}{}^\alpha_1\}\ \cup\ \big(T_\alpha(\alpha), \alpha\big) \times \{N_{w' \overline{b}{}^\alpha_1} \cdot 0\}\,,
\end{multline*}
whereas $\mathcal{T}_\alpha\big(J^\alpha_{w'} \times \{N_{w'} \cdot 0\}\big) = \big(T_\alpha(\alpha), T_\alpha^{k'}(\alpha)\big) \times \{N_{w' \overline{b}{}^\alpha_1} \cdot 0\}$  if $\overline{b}{}^\alpha_{k'} = \overline{b}{}^\alpha_1$.

Theorem~\ref{t:endpoints} gives that $T_\alpha^k({\alpha-1}) = T_\alpha^{k'}(\alpha)$ and $\underline{b}{}^\alpha_k = {}^{(W)}\overline{b}{}^\alpha_{k'}$ by Lemma~\ref{l:wMat}.
If $w = \underline{b}{}^\alpha_{[1,k)}$, $w' = \overline{b}{}^\alpha_{[1,k')}$, then we have $M_{w' {}^{(W)}a} = M_{w a} E$ for any $a \in \mathscr{A}$, whereas
\[
M_{w' {}^{(W)}a} = M_{u{}^{(-1)}\, \overline{b}{}^\alpha_{[1,k^{{}_\prime})} {}^{(W)}a} = M_a W M_{\overline{b}{}^\alpha_{[1,k^{{}_\prime})}} E^{-1} M_u = M_a M_{\underline{b}{}^\alpha_{[1,k)}} M_u = M_{w a}
\]
otherwise. 
In all cases, this yields that $N_{w a} \cdot 0 = N_{w' {}^{(W)}a} \cdot 0$.
Applying this for $a \in \mathscr{A}_-$, we obtain that
\begin{multline*}
\hspace{-3mm}\mathcal{T}_\alpha\big(J^\alpha_w \times \{N_w \cdot 0\}\big) \cup \mathcal{T}_\alpha\big(J^\alpha_{w'} \times \{N_{w'} \cdot 0\}\big) = \big(T_\alpha(\alpha), \alpha\big) \times \{N_{w \overline{b}{}^\alpha_1} \cdot 0,\, N_{w' \overline{b}{}^\alpha_1} \cdot 0\} \cup \{0\} \times \{0\}  \\
\cup\ [{\alpha-1}, \alpha) \times \big\{N_{w a} \cdot 0 \mid a \in \mathscr{A}_+,\, a \prec \overline{b}{}^\alpha_1\big\}\ \cup\ [{\alpha-1}, \alpha) \times \big\{N_{w' a} \cdot 0 \mid a \in \mathscr{A}_+,\, a \prec \overline{b}{}^\alpha_1\big\}\,,
\end{multline*}
which is precisely~\eqref{e:Tn2}.
\end{proof}

\begin{proof}[\textbf{Proof of Proposition~\ref{p:Omega}}]
We have already noted that~\eqref{e:TnPsi} is equivalent to~\eqref{e:PrimeForm}, and that~\eqref{e:PrimeForm} follows from Lemmas~\ref{l:Jw} and~\ref{l:language} for $\alpha \in (0,1] \setminus \Gamma$ or $\alpha = \chi_v$, $v \in \mathscr{F}$.
For general $\alpha \in \Gamma$, we already know that~\eqref{e:Tn1} holds for $w \in \mathscr{L}^{\prime\times}_\alpha \setminus \big({\mathscr{L}^\times_\alpha\, \underline{b}{}^\alpha_{[1,k)} \cup \mathscr{L}^{\prime\times}_\alpha\, \overline{b}{}^\alpha_{[1,k')}}\big)$.
Together with Lemma~\ref{l:wwprime}, this gives inductively that 
\[
\bigcup_{\substack{w \in \mathscr{L}^{\prime\times}_\alpha:\\ |w| \le n m/m'}} J^\alpha_w \times \{N_w \cdot 0\} \subseteq \bigcup_{0 \le j \le n} \mathcal{T}_\alpha^j\big([{\alpha-1},\alpha) \times \{0\}\big) \subseteq \bigcup_{\substack{w \in \mathscr{L}^{\prime\times}_\alpha:\\ |w| \le n m'/m}} J^\alpha_w \times \{N_w \cdot 0\}
\]
for every $n \ge 0$, where $m = \min(k,k')$ and $m' = \max(k,k')$.
This shows again~\eqref{e:PrimeForm}, hence the proposition.
\end{proof}

For $\alpha \in (0,1]$, $x \in \mathbb{I}_\alpha$, the \emph{$x$-fiber} is 
\[
\Phi_\alpha(x) := \{y \mid (x,y) \in \Omega_\alpha\}\,.
\]
The description of $\Omega_{\alpha}$ as the union of pieces fibering above the various $J^\alpha_w$ shows both that fibers are constant between points in the union of the orbits of $\alpha$ and ${\alpha-1}$ and that a fiber contains every fiber to its left.  The maximal fiber is therefore~$\Phi_\alpha(\alpha)$, which  by~\eqref{e:TnPsi} and Lemma~\ref{l:Ldecomposition} equals $\overline{\Psi^{\prime\times}_\alpha}$.
To be precise, we state the following.

\begin{Cor} \label{c:increasing}
Let $\alpha \in (0,1]$.
If $x, x' \in \mathbb{I}_\alpha$, $x \le x'$, then 
\[
\overline{\Psi^\times_\alpha} \subseteq \Phi_\alpha(\alpha-1) \subseteq \Phi_\alpha(x) \subseteq \Phi_\alpha(x') \subseteq \Phi_\alpha(\alpha) = \overline{\Psi^{\prime\times}_\alpha}\,.
\]
If $(x,x'\,] \cap \big(\big\{T_{\alpha}^j({\alpha-1}) \mid 0 \le j < k\big\} \cup \big\{T_{\alpha}^j(\alpha) \mid 1 \le j < k'\big\}\big) = \emptyset$, then $\Phi_\alpha(x) = \Phi_\alpha(x')$.
\end{Cor}

\begin{Rmk} \label{r:fiberRels} 
The inclusion $\overline{\Psi^\times_\alpha} \subseteq \Phi_\alpha(\alpha-1)$ can be strict only when $\alpha-1 \in \overline{\big\{T_{\alpha}^j(\alpha) \mid j \ge 1\big\}}$ or $\alpha-1 \in \overline{\big\{T_{\alpha}^j({\alpha-1}) \mid j \ge 1\big\}}$, which implies that $\alpha \in (0,1] \setminus \Gamma$.
Furthermore,  Lemma~\ref{l:LLprime} implies that $\overline{\Psi^\times_{\alpha}}  = {}^t\hspace{-.1em}E^{-1}  \cdot \overline{\Psi^{\prime\times}_{\alpha}}\,$; since $\overline{\Psi^{\prime\times}_\alpha} \subseteq [0,1]$ and ${}^t\hspace{-.1em}E^{-1} \cdot y = y/(y+1)$ takes $[0, 1]$ to $[0,1/2]$, we find that  $\overline{\Psi^\times_\alpha} \subseteq [0,1/2]$.  
Compare this with Figures~\ref{f:empty},~\ref{f:2} and~\ref{f:23}.
\end{Rmk}

Now we give a description of $\Omega_\alpha$ which provides good approximations of $\Omega_\alpha$ and of~$\mu(\Omega_\alpha)$.
We show how the languages $\mathscr{L}^\times_\alpha$ and~$\mathscr{L}^{\prime\times}_\alpha$ can be replaced by the restricted languages $\mathscr{L}_\alpha$ and~$\mathscr{L}'_\alpha$.
To this matter, we define the alphabet
\[
\mathscr{A}_\alpha := \big\{a \in \mathscr{A}_- \mid \overline{b}{}^\alpha_1 \preceq a \preceq {}^{(W)}\overline{b}{}^\alpha_1\big\} \cup \big\{\,\overline{b}{}^\alpha_1\,\big\} \,.
\] 
This set can also be written as $\{\,(-1:d') \mid 2 \le d' \le d_\alpha(\alpha)+1\} \cup \{(+1:d_\alpha(\alpha))\}$. 

\begin{Lem} \label{l:Aalpha}
Let $\alpha \in (0,1]$.
Then $\underline{b}{}^\alpha_j \in \mathscr{A}_\alpha$ for all $1 \le j < k$, $\overline{b}{}^\alpha_j \in \mathscr{A}_\alpha$ for all $1 \le j < k'$.
We have $\mathscr{L}_\alpha = \mathscr{L}^\times_\alpha \cap \mathscr{A}_\alpha^*$ and $\mathscr{L}'_\alpha = \mathscr{L}^{\prime\times}_\alpha \cap \mathscr{A}_\alpha^*$.
\end{Lem}

\begin{proof} 
Let first $\alpha \in (0,1] \setminus \Gamma$, and $a_{[1,\infty)}$ be the characteristic sequence of~$\alpha-1$. 
We have $a_1 = d_\alpha(\alpha)-1$ since $(+1:d_\alpha(\alpha)) = \overline{b}{}^\alpha_1 = {}^{(W)}(-1:2+a_1) = (+1:1+a_1)$ by Theorem~\ref{t:endpoints}.
Moreover, Theorem~\ref{t:endpoints} implies that $a_n \le a_1$ for all $n \ge 1$ and that $a_{[2,\infty)}$ is the characteristic sequence of $\overline{b}{}^\alpha_{[2,\infty)}$, thus $\underline{b}{}^\alpha_{[1,\infty)} \in \mathscr{A}_\alpha^\omega$ and $\overline{b}{}^\alpha_{[1,\infty)} \in \mathscr{A}_\alpha^\omega$.

Let now $\alpha \in \Gamma_v$, and $a_{[1,2\ell+1]}$ be the characteristic sequence of $v \in \mathscr{F}$.
If $\ell = 0$, then we have $\underline{b}{}^\alpha_{[1,k)} = (-1:2)^{a_1-1}$, $\overline{b}{}^\alpha_{[1,k')} = \overline{b}{}^\alpha_1 = (+1:a_1)$, and these words are in~$\mathscr{A}_\alpha^*$.
If $\ell \ge 1$, then $a_1 = d_\alpha(\alpha)-1$ as in the case $\alpha \in (0,1] \setminus \Gamma$.
Again, we have $a_n \le a_1$ for all $1 \le n \le 2\ell+1$, thus $\underline{b}{}^\alpha_{[1,k)} \in \mathscr{A}_\alpha^*$ and $\overline{b}{}^\alpha_{[1,k')} \in \mathscr{A}_\alpha^*$.

The equations $\mathscr{L}_\alpha = \mathscr{L}_\alpha^\times \cap \mathscr{A}_\alpha^*$ and $\mathscr{L}'_\alpha = \mathscr{L}^{\prime\times}_\alpha \cap \mathscr{A}_\alpha^*$ are now immediate consequences of the definitions.
\end{proof}

\begin{Lem} \label{l:Psialpha}
For any $\alpha \in (0,1]$, we have $\Psi_\alpha = \overline{\Psi^\times_\alpha}$ and $\Psi'_\alpha = \overline{\Psi^{\prime\times}_\alpha}$.
\end{Lem}

\begin{proof}
We know from Lemma~\ref{l:positivemeasure} and Corollary~\ref{c:increasing} that $\big[0,\frac{1}{d_\alpha(\alpha)+1}\big] \subset \overline{\Psi^{\prime\times}_\alpha}$. 
The last letter of any $w \in \mathscr{L}^{\prime\times}_\alpha$ with $N_w \cdot 0 \in \big(0,\frac{1}{d_\alpha(\alpha)+1}\big)$ is not in~$\mathscr{A}_\alpha$, thus $w \in \mathscr{L}^\times_\alpha$ by Lemmas~\ref{l:Ldecomposition} and~\ref{l:Aalpha}.
This implies that $\big[0,\frac{1}{d_\alpha(\alpha)+1}\big] \subset \overline{\Psi^\times_\alpha}$.
Since $\mathscr{L}^\times_\alpha\, \mathscr{L}^{\prime\times}_\alpha = \mathscr{L}^{\prime\times}_\alpha$ and $\mathscr{L}'_\alpha \subset \mathscr{L}^{\prime\times}_\alpha$, we obtain that 
$N_w \cdot \big[0, \tfrac{1}{d_\alpha(\alpha)+1}\big] \subset \overline{\Psi^{\prime\times}_\alpha}$ for all $w \in \mathscr{L}'_\alpha$, thus $\Psi'_\alpha \subseteq \overline{\Psi^{\prime\times}_\alpha}$.
For the other inclusion, write any $w' \in \mathscr{L}^{\prime\times}_\alpha$ as $w' = u w$, with $w \in \mathscr{A}_\alpha^*$ and $u$ empty or ending with a letter in $\mathscr{A} \setminus \mathscr{A}_\alpha$.
Then we have $u \in \mathscr{L}^\times_\alpha$ by Lemmas~\ref{l:Ldecomposition} and~\ref{l:Aalpha}, and $w \in \mathscr{L}^{\prime\times}_\alpha \cap \mathscr{A}_\alpha^* = \mathscr{L}'_\alpha$, thus $N_{w'} \cdot 0 = N_w N_u \cdot 0 \in N_w \cdot \big[0, \tfrac{1}{d_\alpha(\alpha)+1}\big] \subset \Psi'_\alpha$.
Since $\Psi'_\alpha$ is closed, this shows that $\Psi'_\alpha = \overline{\Psi^{\prime\times}_\alpha}$.
In the same way, $\mathscr{L}^\times_\alpha\, \mathscr{L}^\times_\alpha = \mathscr{L}^\times_\alpha$ and $\mathscr{L}^\times_\alpha \cap \mathscr{A}_\alpha^* = \mathscr{L}_\alpha$ imply that $\Psi_\alpha = \overline{\Psi^\times_\alpha}$.
\end{proof}

\begin{Lem} \label{l:Ualpha}
For any $\alpha \in (0,1]$, we have
\begin{align*}
\Omega_\alpha & = \mathbb{I}_\alpha \times \Psi_\alpha\ \cup \overline{\bigcup_{1 \le j < k} \big[T_\alpha^j(\alpha-1), \alpha\big] \times N_{\underline{b}{}^\alpha_{[1,j]}} \cdot \Psi_\alpha}\ \cup \overline{\bigcup_{1 \le j < k'} \big[T_\alpha^j(\alpha), \alpha\big] \times N_{\overline{b}{}^\alpha_{[1,j)}} \cdot \Psi'_\alpha} \\
& = \overline{\bigcup_{w \in \mathscr{L}'_\alpha} J^\alpha_{\vphantom{I}w} \times N_w \cdot \big[0, \tfrac{1}{d_\alpha(\alpha)+1}\big]}\,.
\end{align*}
For any $w \in \mathscr{L}'_\alpha$, we have $N_w \cdot \big(0, \tfrac{1}{d_\alpha(\alpha)+1}\big) \cap\, \overline{\bigcup_{w' \in\mathscr{L}'_\alpha \setminus \{w\}} N_{w'} \cdot \big[0, \tfrac{1}{d_\alpha(\alpha)+1}\big]} = \emptyset$. 
\end{Lem}

\begin{proof}
The first equation follows from Proposition~\ref{p:Omega} and Lemma~\ref{l:Psialpha}.
The decomposition $\mathscr{L}'_\alpha = \mathscr{L}_\alpha \cup\, \bigcup_{1 \le j < k} \mathscr{L}_\alpha\, \underline{b}{}^\alpha_{[1,j]} \cup\, \bigcup_{1 \le j < k'} \mathscr{L}'_\alpha\, \overline{b}{}^\alpha_{[1,j]}$ gives the second equation.

To show the disjointness of $N_w \cdot \big(0, \tfrac{1}{d_\alpha(\alpha)+1}\big)$ and $\overline{\bigcup_{w' \in\mathscr{L}'_\alpha \setminus \{w\}} N_{w'} \cdot \big[0, \tfrac{1}{d_\alpha(\alpha)+1}\big]}$, note first that $\alpha \in \Gamma_v$, $v \in \mathscr{F}$, implies that $d_\alpha(\alpha) = d_{\chi_v}(\chi_v)$ and $\mathscr{L}'_\alpha = \mathscr{L}'_{\chi_v}$.
Therefore, we can assume that $\alpha = \chi_v$ or $\alpha \in (0,1] \setminus \Gamma$.
Then Lemma~\ref{l:Jw} yields that 
\begin{equation} \label{e:invw}
\mathcal{T}_\alpha^{|w|}\big(\Delta_\alpha(w) \times \big[0, \tfrac{1}{d_\alpha(\alpha)+1}\big]\big) = J^\alpha_w \times N_w \cdot \big[0, \tfrac{1}{d_\alpha(\alpha)+1}\big] \qquad \big(w \in \mathscr{L}'_\alpha\big).
\end{equation}
Since $\mathcal{T}_\alpha$ is bijective (up to a set of measure zero) by Lemma~\ref{l:bijective}, the disjointness of the cylinders $\Delta_\alpha(w)$ and $\Delta_\alpha(w')$ yields that $\mu\big(J^\alpha_w \times N_w \cdot \big[0, \tfrac{1}{d_\alpha(\alpha)+1}\big]\, \cap\, J^\alpha_{w'} \times N_{w'} \cdot \big[0, \tfrac{1}{d_\alpha(\alpha)+1}\big]\big) = 0$ for all $w, w' \in \mathscr{L}'_\alpha$ with $|w| = |w'|$, $w \ne w'$.
For all $w, w' \in \mathscr{L}'_\alpha$ with $|w| < |w'|$, we have 
\begin{equation} \label{e:invw2}
\mathcal{T}_\alpha^{|w|}\Big(T_\alpha^{|w'|-|w|}(\Delta_\alpha(w')) \times N_{w'_{[1,|w'|-|w|\,]}} \cdot \big[0, \tfrac{1}{d_\alpha(\alpha)+1}\big]\Big) = J^\alpha_{w'} \times N_{w'} \cdot \big[0, \tfrac{1}{d_\alpha(\alpha)+1}\big]\,.
\end{equation}
The inclusion $N_a \cdot \big[0, 1\big] \subset \big[\tfrac{1}{d_\alpha(\alpha)+1}, 1\big]$ for all $a \in \mathscr{A}_\alpha$ gives that
\[
\Delta_\alpha(w) \times \big[0, \tfrac{1}{d_\alpha(\alpha)+1}\big)\ \cap \overline{\bigcup_{\substack{w'\in\mathscr{L}'_\alpha:\\|w'|>|w|}} T_\alpha^{|w'|-|w|}(\Delta_\alpha(w')) \times N_{w'_{[1,|w'|-|w|\,]}} \cdot \big[0, \tfrac{1}{d_\alpha(\alpha)+1}\big]} = \emptyset\,.
\]
As $\mathcal{T}_\alpha$ is bijective and continuous $\mu$-almost everywhere, applying $\mathcal{T}_\alpha^{|w|}$ yields that $J^\alpha_w \times N_w \cdot \big[0, \tfrac{1}{d_\alpha(\alpha)+1}\big]$ and $\overline{\bigcup_{w'\in\mathscr{L}'_\alpha:\,|w'|>|w|} J^\alpha_{w'} \times N_{w'} \cdot \big[0, \tfrac{1}{d_\alpha(\alpha)+1}\big]}$ are $\mu$-disjoint.
We have shown that
\[
\mu\bigg(J^\alpha_w \times N_w \cdot \big[0, \tfrac{1}{d_\alpha(\alpha)+1}\big]\ \cap \overline{\bigcup_{\substack{w'\in\mathscr{L}'_\alpha:\\|w'|\ge|w|,\,w'\ne w}} J^\alpha_{w'} \times N_{w'} \cdot \big[0, \tfrac{1}{d_\alpha(\alpha)+1}\big]}\bigg) = 0\,.
\]
Inverting the roles of $w$ and~$w'$, we also obtain for all $w'\in\mathscr{L}'_\alpha$ with $|w'|<|w|$ that $J^\alpha_w \times N_w \cdot \big[0, \tfrac{1}{d_\alpha(\alpha)+1}\big]$ and $J^\alpha_{w'} \times N_{w'} \cdot \big[0, \tfrac{1}{d_\alpha(\alpha)+1}\big]$ are $\mu$-disjoint.
Since $(0, \alpha] \subseteq J^\alpha_w$ for all $w \in \mathscr{L}'_\alpha$, this yields that the intersection of $N_w \cdot \big[0, \tfrac{1}{d_\alpha(\alpha)+1}\big]$ and $\overline{\bigcup_{w'\in\mathscr{L}'_\alpha \setminus \{w\}} N_{w'} \cdot \big[0, \tfrac{1}{d_\alpha(\alpha)+1}\big]}$ has zero Lebesgue measure, thus $N_w \cdot \big(0, \tfrac{1}{d_\alpha(\alpha)+1}\big) \cap\, \overline{\bigcup_{w' \in\mathscr{L}'_\alpha \setminus \{w\}} N_{w'} \cdot \big[0, \tfrac{1}{d_\alpha(\alpha)+1}\big]} = \emptyset$.
\end{proof}

We study now the sets
\begin{equation} \label{e:defUpsilon}
\Xi_{\alpha,n} := \overline{\bigcup_{\substack{w\in\mathscr{L}'_\alpha:\\|w|\ge n}}  J^\alpha_{\vphantom{I}w} \times N_w \cdot \big[0, \tfrac{1}{d_\alpha(\alpha)+1}\big]} \quad (n \ge 0)\,,
\end{equation}
which are obtained from $\Omega_\alpha$ by removing finitely many rectangles. 
In the following, we can and usually do ignore  various sets of measure zero.

\begin{Lem} \label{l:exponentialbound}
Let $\alpha \in (0,1] \setminus \Gamma$ or $\alpha = \chi_v$, $v \in \mathscr{F}$. 
For any $n \ge 0$, we have
\[
\mu(\Xi_{\alpha,n}) \le \mu(\Omega_\alpha)\, \big(\tfrac{d_\alpha(\alpha)}{d_\alpha(\alpha)+\alpha}\big)^n.
\]
\end{Lem}

\begin{proof}
For any $n \ge 0$, we have
\begin{equation} \label{e:diffXi}
\mathcal{T}_\alpha^{-n}\big(\Xi_{\alpha,n} \setminus \Xi_{\alpha,n+1}\big) = \bigcup_{\substack{w\in\mathscr{L}'_\alpha:\\|w|=n}} \mathcal{T}_\alpha^{-n}\Big(J^\alpha_w \times N_w \cdot \big[0, \tfrac{1}{d_\alpha(\alpha)+1}\big]\Big) = X_{\alpha,n} \times \big[0, \tfrac{1}{d_\alpha(\alpha)+1}\big]
\end{equation}
by Lemma~\ref{l:Ualpha} and~\eqref{e:invw}, with $X_{\alpha,n} := \bigcup_{w \in \mathscr{L}'_\alpha:\, |w|=n} \Delta_\alpha(w)$.
For any $w' \in \mathscr{L}'_\alpha$ with $|w'|>n$, we have $T_\alpha^{|w'|-n}(\Delta_\alpha(w')) \subset \Delta_\alpha(w'_{[\,|w'|-n+1,|w'|\,]}) \subset X_{\alpha,n}$.
Therefore, \eqref{e:invw2} implies that 
\[
\mathcal{T}_\alpha^{-n}(\Xi_{\alpha,n+1}) \subset X_{\alpha,n} \times \big[\tfrac{1}{d_\alpha(\alpha)+1}, 1\big]\,.
\]
With Theorem~\ref{t:natext}, we obtain that 
\begin{multline*}
\frac{\mu(\Xi_{\alpha,n} \setminus \Xi_{\alpha,n+1})}{\mu(\Xi_{\alpha,n})} \ge \frac{\mu\big(X_{\alpha,n} \times \big[0, \frac{1}{d_\alpha(\alpha)+1}\big]\big)}{\mu\big(X_{\alpha,n} \times [0,1]\big)} \ge \min_{x\in\mathbb{I}_\alpha} \frac{\int_0^{1/(d_\alpha(\alpha)+1)} \frac{1}{(1+xy)^2}\, dy}{\int_0^1 \frac{1}{(1+xy)^2}\, dy} \\
= \min_{x\in\mathbb{I}_\alpha} \frac{\frac{y}{1+xy}\, \big|_{y=0}^{1/(d_\alpha(\alpha)+1)}}{\frac{y}{1+xy}\,\big|_{y=0}^1} = \min_{x\in\mathbb{I}_\alpha} \frac{1+x}{d_\alpha(\alpha)+1+x} = \frac{\alpha}{d_\alpha(\alpha)+\alpha}\,.
\end{multline*}
This implies that $\mu(\Xi_{\alpha,n+1}) \le \frac{d_\alpha(\alpha)}{d_\alpha(\alpha)+\alpha}\, \mu(\Xi_{\alpha,n})$.
Since $\Xi_{\alpha,0} = \Omega_\alpha$, this proves the lemma.
\end{proof}

\begin{Rmk}
Since $[0,\alpha] \times \Psi'_\alpha \subset \Omega_\alpha$ for $\alpha \in (0,1] \setminus \Gamma$ or $\alpha = \chi_v$, $v \in \mathscr{F}$, Lemma~\ref{l:exponentialbound} implies that the Lebesgue measure of $\overline{\bigcup_{w \in \mathscr{L}'_\alpha:\, |w|\ge n} N_w \cdot \big[0, \tfrac{1}{d_\alpha(\alpha)+1}\big]}$ is at most of the order $\big(\tfrac{d_\alpha(\alpha)}{d_\alpha(\alpha)+\alpha}\big)^n$ for these~$\alpha$.
For $\alpha \in \Gamma_v$, $v \in \mathscr{F}$, we obtain that $\mu(\Xi_{\alpha,n}) \le c_\alpha\, \big(\frac{d_\alpha(\alpha)}{d_\alpha(\alpha)+\chi_v}\big)^n$ for some $c_\alpha > 0$.
A~calculation similar to the proof of Lemma~\ref{l:exponentialbound} shows that we can choose $c_\alpha = \frac{1+\chi_v}{\chi_v\alpha (1+\alpha)}\, \mu(\Omega_{\chi_v})$.
\end{Rmk}

Lemmas~\ref{l:Ualpha} and~\ref{l:exponentialbound} and the estimate $\mu(\Omega_\alpha) \le \mu(\mathbb{I}_\alpha \times [0,1]) = \log\big({1+\frac{1}{\alpha}}\big)$ give the following bound for the error of an approximation of $\mu(\Omega_\alpha)$ by a sum of measures of rectangles which are contained in~$\Omega_\alpha$.

\begin{Cor} \label{c:approximation}
Let $\alpha \in (0,1] \setminus \Gamma$ or $\alpha = \chi_v$, $v \in \mathscr{F}$. 
Then we have, for any $n \ge 0$, 
\[
0\ \le\ \mu(\Omega_\alpha)\ - \sum_{\substack{w\in\mathscr{L}'_\alpha:\\|w|<n}} \mu\big(J^\alpha_w \times N_w \cdot \big[0, \tfrac{1}{d_\alpha(\alpha)+1}\big]\big) \le \big(\tfrac{d_\alpha(\alpha)}{d_\alpha(\alpha)+\alpha}\big)^n\,\log\big(1+\tfrac{1}{\alpha}\big)\,.
\]
\end{Cor}

\begin{proof}[\textbf{Proof of Theorem~\ref{t:shapeOmega}}]
Equation~\eqref{e:shapeOmega} is proved in Lemma~\ref{l:Ualpha} and implies that the density of the invariant measure $\nu_\alpha$ is continous on any interval $(x,x')$ satisfying $T_{\alpha}^j({\alpha-1}) \not\in (x,x')$ for all $0 \le j < k$ and $T_{\alpha}^j(\alpha) \not\in (x,x')$ for all $0 \le j < k'$.
The equation $\Psi'_\alpha = \overline{\bigcup_{Y \in \mathscr{C}_\alpha} Y}$ follows from $\mathscr{L}'_\alpha = \mathscr{L}_\alpha \cup\, \bigcup_{1 \le j < k} \mathscr{L}_\alpha\, \underline{b}{}^\alpha_{[1,j]} \cup\, \bigcup_{1 \le j < k'} \mathscr{L}'_\alpha\, \overline{b}{}^\alpha_{[1,j]}$ and the compactness of~$\Psi'_\alpha$.
By Lemma~\ref{l:Ualpha}, $N_w \cdot \big(0, \tfrac{1}{d_\alpha(\alpha)+1}\big)$ is disjoint from the rest of the intervals constituting~$\Psi'_\alpha$.
Taking the closure in unions of such intervals does not increase the measure, by Lemma~\ref{l:exponentialbound}.
Therefore, the disjointness of the decomposition $\mathscr{L}'_\alpha = \mathscr{L}_\alpha \cup\, \bigcup_{1 \le j < k} \mathscr{L}_\alpha\, \underline{b}{}^\alpha_{[1,j]} \cup\, \bigcup_{1 \le j < k'} \mathscr{L}'_\alpha\, \overline{b}{}^\alpha_{[1,j]}$ implies that, for any $Y \in \mathscr{C}_\alpha$, the Lebesgue measure of $Y \cap\, \overline{\bigcup_{Y' \in \mathscr{C}_\alpha \setminus \{Y\}} Y'}$ is zero.
Finally, $\Psi'_\alpha = {}^t\hspace{-.1em}E \cdot \Psi_\alpha$ follows from Lemmas~\ref{l:LLprime} and~\ref{l:Psialpha}.
\end{proof}

\section{Evolution of the natural extension along a synchronizing interval} \label{sec:evol-omeg-along}

Given $v \in \mathscr{F}$,  both $\underline{b}{}^\alpha_{[1,|v|\,]}$ and $\overline{b}{}^\alpha_{[1,|\widehat{v}|\,]}$ are invariant within the interval~$\Gamma_v$.  The same is hence   true for $\Psi_\alpha$ and~$\Psi'_\alpha$, which we accordingly denote by $\Psi_v$ and $\Psi'_v$, respectively.
The evolution of the natural extension domain,  and of the entropy, is now straightforward to describe along such an interval.
The following lemma is mainly a rewording of~\eqref{e:shapeOmega}, but addresses the endpoints of~$\Gamma_v$.

\begin{Lem}\label{l:simplerDescr}
Let $v = v_{[1,|v|\,]} \in \mathscr{F}$, $v' = v'_{[1,|\widehat{v}|\,]} = {}^{(W)}\widehat{v}{}^{(-1)}$.
For any $\alpha \in [\zeta_v, \eta_v]$, we have
\[
\Omega_\alpha =  \bigcup_{0 \le j \le |v|} \overline{\big[M_{v_{[1,j]}} \cdot ({\alpha-1}), \alpha\big)} \times N_{v_{[1,j]}} \cdot \Psi_v\ \cup \bigcup_{1 \le j \le |\widehat{v}|} \overline{\big(M_{v'_{[1,j]}} \cdot \alpha, \alpha\big)} \times N_{v'_{[1,j]}} \cdot \Psi'_v\,.
\]
\end{Lem}

\begin{proof}
Since $M_{v_{[1,j]}} \cdot ({\alpha-1}) \in [{\alpha-1},\alpha]$ for all $0 \le j \le |v|$, and $M_{v'_{[1,j]}} \cdot \alpha \in [{\alpha-1},\alpha]$ for all $1 \le j \le |\widehat{v}|$, the equation follows from the proof of Theorem~\ref{t:shapeOmega}.
\end{proof}

\begin{Rmk} \label{r:empty}
Note that $(M_{v'} \cdot \zeta_v, \zeta_v)$ is the empty interval by Lemma~\ref{l:Gammav}, therefore the contribution from $N_{v'} \cdot \Psi'_v$ vanishes at $\alpha = \zeta_v$.
Similarly, if $v$ is not the empty word, then $[M_v \cdot (\eta_v-1), \eta_v)$ is the empty interval and there is no contribution from $N_v \cdot \Psi_v$ at $\alpha = \eta_v$.  
\end{Rmk}

\begin{Ex}
If $v$ is the empty word, then $\Omega_\alpha =  \mathbb{I}_\alpha \times \Psi_v\, \cup\, \overline{\big(M_{(+1:1)} \cdot \alpha, \alpha\big)} \times N_{(+1:1)} \cdot \Psi'_v$.
Here, we know from \cite{Nakada81} that $\Psi_v = [0,1/2]$, $\Psi'_v = [0,1]$, $\Omega_\alpha =  \mathbb{I}_\alpha \times [0,1/2]\, \cup\, [T_\alpha(\alpha), \alpha] \times [1/2,1]$ if $\alpha \in (g,1]$, and $\Omega_g =  \mathbb{I}_g \times [0,1/2]$, see Figure~\ref{f:empty}.   
\end{Ex}

\begin{figure}[ht]
\begin{tikzpicture}[scale=4,fill=black!20]
\small
\begin{scope}[shift={(-1.118,0)}]
\node[below=4ex] at (.118,.0) {$\alpha = g$};
\filldraw(-.382,.0)--(-.382,.5)--(.618,.5)--(.618,.0)--cycle;
\draw(0,0)node[below]{$0$}--(.0,1.0);
\node[below] at (-.382,.0) {$\vphantom{0}-\!g^2$};
\node[below] at (.618,.0) {$\vphantom{0}g$};
\node[right] at (.618,.0) {$0$};
\node[right] at (.618,.5) {$1/2$};
\end{scope}

\node[below=4ex] at (.3,.0) {$\alpha = 4/5$};
\filldraw(-.2,.0)--(-.2,.5)--(.25,.5)--(.25,1.0)--(.8,1.0)--(.8,.0)--cycle;
\draw(0,0)node[below]{$\vphantom{()}0$}--(.0,1.0);
\draw[thin,dotted](.25,.0)--(.25,.5);
\node[below] at (-.2,.0) {$\vphantom{(}\alpha\!-\!1$};
\node[below] at (.8,.0) {$\vphantom{(}\alpha$};
\node[below] at (.25,.0) {$T_\alpha(\alpha)$};
\draw[dashed](.5556,.0)--(.5556,1.0);
\node[below] at (.5556,0) {$\frac{1}{1+\alpha}$};
\node[right] at (.8,.0) {$0$};
\draw[thin,dotted](.25,.5)--(.8,.5)node[right]{$1/2$};
\node[right] at (.8,1) {$1$};

\begin{scope}[shift={(1.1,0)}]
\node[below=4ex] at (.5,.0) {$\alpha = 1$};
\filldraw(0,0)--(0,1)--(1,1)--(1,0)--cycle;
\node[below] at (0,0) {$0$};
\node[below] at (1,0) {$1$};
\node[right] at (1,0) {$0$};
\node[right] at (1,1) {$1$};
\end{scope}
\end{tikzpicture}
\caption{The natural extension domain $\Omega_\alpha$ for $\alpha \in [g,1]$.}
\label{f:empty}
\end{figure}
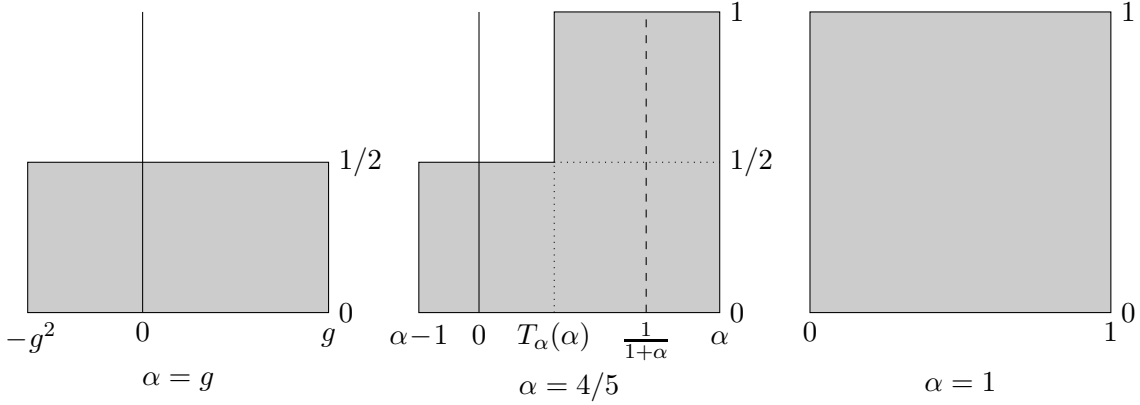

\begin{Ex}
For $\alpha \in \overline{\Gamma_{(-1:2)}} = [\sqrt{2}-1, g]$, the natural extension domain is
\[
\Omega_\alpha = \mathbb{I}_\alpha \times \Psi_{(-1:2)}\ \cup\ \overline{\big[M_{(-1:2)} \cdot (\alpha-1), \alpha\big)} \times N_{(-1:2)} \cdot \Psi_{(-1:2)}\ \cup\ \overline{\big(M_{(+1:2)} \cdot \alpha, \alpha\big)} \times N_{(+1:2)} \cdot \Psi'_{(-1:2)}\,.
\]
We know from \cite{Nakada81,MoussaCassaMarmi:99} that $\Psi_{(-1:2)} = [0,g^2]$ and $\Psi'_{(-1:2)} = [0,g]$, 
\[
\Omega_\alpha =  \mathbb{I}_\alpha \times \big[0,g\big]\ \cup\ \big[T_\alpha({\alpha-1}), \alpha\big] \times \big[1/2,g\big]\ \cup\ \big[T_\alpha(\alpha), \alpha\big] \times \big[g^2,1/2\big] \quad \big(\alpha \in \Gamma_{(-1:2)}\big).
\]
In Figure~\ref{f:2}, one can see how $[T_\alpha({\alpha-1}), \alpha] \times [1/2,g]$ shrinks and $[T_\alpha(\alpha), \alpha] \times [g^2,1/2]$ grows when $\alpha$ increases.
\end{Ex}

\begin{figure}[ht]
\begin{tikzpicture}[scale=4,fill=black!20]
\small
\begin{scope}[shift={(-1.2444,0)}]
\node[below=4ex] at (-.0556,.0) {$\alpha = 4/9$};
\filldraw(-.5556,.0)--(-.5556,.382)--(.25,.382)--(.25,.5)--(-.2,.5)--(-.2,.618)--(.4444,.618)--(.4444,.0)--cycle;
\draw(0,0)--(.0,.618)node[above=.5ex]{$0$};
\draw[thin,dotted](-.2,.0)--(-.2,.5);
\draw[thin,dotted](.25,.0)--(.25,.382);
\node[below] at (-.5556,.0) {$\vphantom{(}\alpha\!-\!1$};
\node[below] at (-.2,.0) {$T_\alpha(\alpha\!-\!1)$};
\node[below] at (.4444,.0) {$\vphantom{(}\alpha$};
\node[below] at (.25,.0) {$T_\alpha(\alpha)$};
\draw[dashed](.4091,.0)--(.4091,.618);
\node[above] at (.4091,.618) {$\frac{1}{2+\alpha}$};
\draw[dashed](-.4091,.0)--(-.4091,.618);
\node[above] at (-.4091,.618) {$\frac{-1}{2+\alpha}$};
\node[right] at (.4444,.0) {$0$};
\draw[thin,dotted](.25,.5)--(.4444,.5)node[right]{$1/2$};
\draw[thin,dotted](-.5556,.3333)--(.4444,.3333)node[right]{$1/3$};
\end{scope}

\node[below=4ex] at (.0,.0) {$\alpha = \chi_{(-1:2)} = 1/2$};
\filldraw(-.5,.0)--(-.5,.382)--(.0,.382)--(.0,.5)--(.0,.5)--(.0,.618)--(.5,.618)--(.5,.0)--cycle;
\draw(0,0)node[below]{$\vphantom{/}0$}--(.0,.618);
\node[below] at (-.5,.0) {$-1/2$};
\node[below] at (.5,.0) {$1/2$};
\draw[dashed](.4,.0)--(.4,.618);
\node[above] at (.4,.618) {$\frac{1}{2+\alpha}$};
\draw[dashed](-.4,.0)--(-.4,.618);
\node[above] at (-.4,.618) {$\frac{-1}{2+\alpha}$};
\node[right] at (.5,.01) {$0$};
\draw[thin,dotted](.0,.5)--(.5,.5)node[right]{$1/2$};
\draw[thin,dotted](-.5,.3333)--(.5,.3333)node[right]{$1/3$};

\begin{scope}[shift={(1.2556,0)}]
\node[below=4ex] at (.0556,.0) {$\alpha = 5/9$};
\filldraw(-.4444,.0)--(-.4444,.382)--(-.2,.382)--(-.2,.5)--(.25,.5)--(.25,.618)--(.5556,.618)--(.5556,.0)--cycle;
\draw(0,0)--(.0,.618)node[above=.5ex]{$0$};
\draw[thin,dotted](.25,.0)--(.25,.5);
\draw[thin,dotted](-.2,.0)--(-.2,.382);
\node[below] at (-.4444,.0) {$\vphantom{(}\alpha\!-\!1$};
\node[below] at (.25,.0) {$T_\alpha(\alpha\!-\!1)$};
\node[below] at (.5556,.0) {$\vphantom{(}\alpha$};
\node[below] at (-.2,.0) {$T_\alpha(\alpha)$};
\draw[dashed](.3913,.0)--(.3913,.618);
\node[above] at (.3913,.618) {$\frac{1}{2+\alpha}$};
\draw[dashed](-.3913,.0)--(-.3913,.618);
\node[above] at (-.3913,.618) {$\frac{-1}{2+\alpha}$};
\node[right] at (.5556,.0) {$0$};
\draw[thin,dotted](.25,.5)--(.5556,.5)node[right]{$1/2$};
\draw[thin,dotted](-.4444,.3333)--(.5556,.3333)node[right]{$1/3$};
\end{scope}
\end{tikzpicture}
\caption{The natural extension domain $\Omega_\alpha$ for $\alpha \in \Gamma_{(-1:2)} = (\sqrt{2}-1, g)$.}
\label{f:2}
\end{figure}
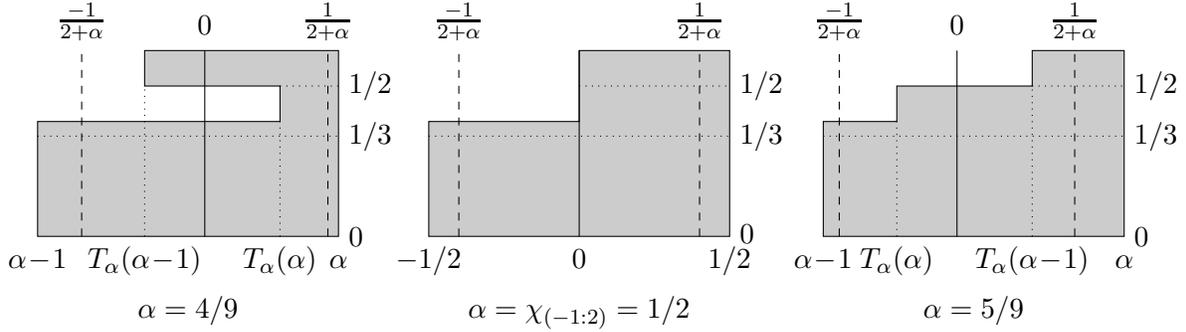

\begin{figure}[ht]
\begin{tikzpicture}[scale=6.5,fill=black!20]
\small
\begin{scope}[shift={(-1.1858,0)},ultra thin]
\node[below=4ex] at (-.1,.0) {$\alpha = 2/5$};
\filldraw(-.6,.0)--(-.6,.2763)--(-.5,.2763)--(-.5,.2857)--(-.6,.2857)--(-.6,.2923)--(-.5,.2923)--(-.5,.2927)--(-.6,.2927)--(-.6,.2929)--(-.5,.2929)--(-.5,.2941)--(-.6,.2941)--(-.6,.2957)--(-.5,.2957)--(-.5,.3333)--(-.6,.3333)--(-.6,.3819)--(.4,.3819)--(.4,.0)--cycle;
\filldraw(.0,.4)--(.0,.4132)--(.4,.4132)--(.4,.4)--cycle;
\filldraw(.0,.4138)--(.0,.4141)--(.4,.4141)--(.4,.4138)--cycle;
\filldraw(.0,.4167)--(.0,.4198)--(.4,.4198)--(.4,.4167)--cycle;
\filldraw(-.3333,.5)--(-.3333,.5802)--(.0,.5802)--(.0,.5833)--(-.3333,.5833)--(-.3333,.5856)--(.0,.5856)--(.0,.5862)--(-.3333,.5862)--(-.3333,.5867)--(.0,.5867)--(.0,.6)--(-.3333,.6)--(-.3333,.618)--(.4,.618)--(.4,.5)--cycle;
\draw(0,0) node[below]{$\vphantom{(}0$}--(.0,.618);
\draw[thin,dotted](-.3333,.0)--(-.3333,.5);
\draw[thin,dotted](-.5,.0)--(-.5,.2763);
\node[below] at (-.62,.0) {$\vphantom{(}\alpha\!-\!1$};
\node[below] at (-.275,.0) {$T_\alpha(\alpha\!-\!1)$};
\node[below] at (.4,.0) {$\vphantom{(}\alpha$};
\node[below] at (-.475,.0) {$T_\alpha(\alpha)$};
\draw[dashed](-.4167,.0)--(-.4167,.618);
\node[above] at (-.4167,.618) {$\frac{-1}{2+\alpha}$};
\draw[dashed](-.2941,.0)--(-.2941,.618);
\node[above] at (-.2941,.618) {$\frac{-1}{3+\alpha}$};
\draw[dashed](.2941,.0)--(.2941,.618);
\node[above] at (.2941,.618) {$\frac{1}{3+\alpha}$};
\node[right] at (.4,.0) {$0$};
\node[right] at (.4,.5) {$1/2$};
\draw[thin,dotted](-.5,.3333)--(.4,.3333)node[right]{$1/3$};
\draw[thin,dotted](-.6,.25)--(.4,.25)node[right]{$1/4$};
\end{scope}

\node[below=4ex] at (-.0858,.0) {$\alpha = \sqrt{2} - 1$};
\filldraw(-.5858,.0)--(-.5858,.3819)--(.4142,.3819)--(.4142,.0)--cycle;
\filldraw(-.2929,.5)--(-.2929,.618)--(.4142,.618)--(.4142,.5)--cycle;
\draw(0,0)node[below]{$\vphantom{/}0$}--(.0,.618);
\draw[thin,dotted](-.2929,.0)--(-.2929,.5);
\node[below] at (-.5858,.0) {$\sqrt{2}-2$};
\node[below] at (-.2929,.0) {$1-1/\sqrt{2}$};
\node[below] at (.4142,.0) {$\sqrt{2}-1$};
\draw[dashed](-.4142,.0)--(-.4142,.618);
\node[above] at (-.4142,.618) {$\frac{-1}{2+\alpha}$};
\node[above] at (-.2929,.618) {$\frac{-1}{3+\alpha}$};
\node[above] at (.4142,.618) {$\frac{1}{2+\alpha}$};
\draw[dashed](.2929,.0)--(.2929,.618);
\node[above] at (.2929,.618) {$\frac{1}{3+\alpha}$};
\node[right] at (.4142,.0) {$0$};
\node[right] at (.4142,.5) {$1/2$};
\draw[thin,dotted](-.5858,.3333)--(.4142,.3333)node[right]{$1/3$};
\draw[thin,dotted](-.5858,.25)--(.4142,.25)node[right]{$1/4$};
\end{tikzpicture}
\caption{The natural extension domain $\Omega_\alpha$ for $\alpha = 2/5$ and $\alpha = \sqrt{2}-1$.}
\label{f:23}
\end{figure}
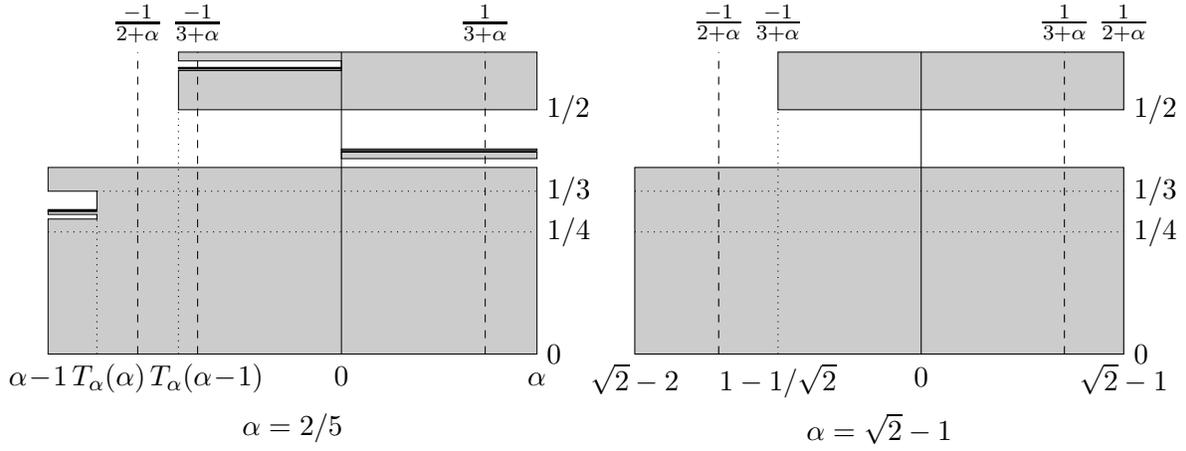

Figure~\ref{f:23} shows the fractal structure appearing in the interval $\Gamma_{(-1:2)(-1:3)}$, which is immediately to the left of $\sqrt{2}-1$.
An even more complicated example of a natural extension domain is shown in Figure~\ref{f:2342}, see also Figures~\ref{f:g2} and~\ref{f:4}.

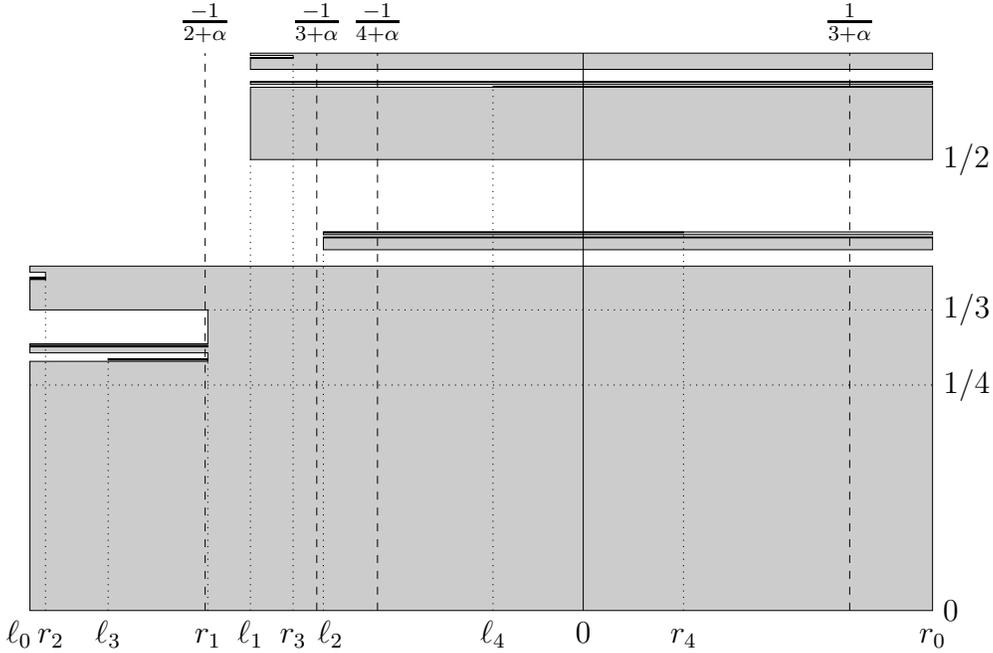
\begin{figure}[ht]
\begin{tikzpicture}[scale=12,fill=black!20,ultra thin]
\filldraw(-.613,.0)--(-.613,.2763)--(-.4159,.2763)--(-.4159,.2778)--(-.5263,.2778)--(-.5263,.2787)--(-.4159,.2787)--(-.4159,.2791)--(-.5263,.2791)--(-.5263,.2792)--(-.4159,.2792)--(-.4159,.2857)--(-.613,.2857)--(-.613,.2923)--(-.4159,.2923)--(-.4159,.2927)--(-.613,.2927)--(-.613,.2929)--(-.4159,.2929)--(-.4159,.2941)--(-.613,.2941)--(-.613,.2957)--(-.4159,.2957)--(-.4159,.3333)--(-.613,.3333)--(-.613,.3671)--(-.5957,.3671)--(-.5957,.3673)--(-.613,.3673)--(-.613,.3675)--(-.5957,.3675)--(-.5957,.3684)--(-.613,.3684)--(-.613,.3692)--(-.5957,.3692)--(-.5957,.3696)--(-.613,.3696)--(-.613,.3697)--(-.5957,.3697)--(-.5957,.375)--(-.613,.375)--(-.613,.3819)--(.387,.3819)--(.387,.0)--cycle;
\filldraw(-.2879,.4)--(-.2879,.4132)--(.387,.4132)--(.387,.4)--cycle;
\filldraw(-.2879,.4138)--(-.2879,.4141)--(.387,.4141)--(.387,.4138)--cycle;
\filldraw(-.2879,.4167)--(-.2879,.4188)--(.1111,.4188)--(.1111,.4194)--(-.2879,.4194)--(-.2879,.4198)--(.387,.4198)--(.387,.4167)--cycle;
\filldraw(-.3687,.5)--(-.3687,.5802)--(.387,.5802)--(.387,.5)--cycle;
\filldraw(-.1,.5806)--(-.1,.581)--(.387,.581)--(.387,.5806)--cycle;
\filldraw(-.3687,.5833)--(-.3687,.5856)--(.387,.5856)--(.387,.5833)--cycle;
\filldraw(-.3687,.5862)--(-.3687,.5867)--(.387,.5867)--(.387,.5862)--cycle;
\filldraw(-.3687,.6)--(-.3687,.6124)--(-.3214,.6124)--(-.3214,.6129)--(-.3687,.6129)--(-.3687,.6132)--(-.3214,.6132)--(-.3214,.6154)--(-.3687,.6154)--(-.3687,.618)--(.387,.618)--(.387,.6)--cycle;
\draw(0,0)node[below]{$\vphantom{\ell_0}0$}--(.0,.618);
\draw[thin,dotted](-.3687,.0)--(-.3687,.5);
\draw[thin,dotted](-.2879,.0)--(-.2879,.4);
\draw[thin,dotted](-.5263,.0)--(-.5263,.2778);
\draw[thin,dotted](-.1,.0)--(-.1,.5806);
\draw[thin,dotted](-.4159,.0)--(-.4159,.2763);
\draw[thin,dotted](-.5957,.0)--(-.5957,.3671);
\draw[thin,dotted](-.3214,.0)--(-.3214,.6124);
\draw[thin,dotted](.1111,.0)--(.1111,.4188);
\node[below] at (-.625,.0) {$\ell_{0}$};
\node[below] at (-.3687,.0) {$\ell_{1}$};
\node[below] at (-.28,.0) {$\ell_{2}$};
\node[below] at (-.5263,.0) {$\ell_{3}$};
\node[below] at (-.1,.0) {$\ell_{4}$};
\node[below] at (.387,.0) {$\vphantom{\ell_0}r_{0}$};
\node[below] at (-.4159,.0) {$\vphantom{\ell_0}r_{1}$};
\node[below] at (-.59,.0) {$\vphantom{\ell_0}r_{2}$};
\node[below] at (-.3214,.0) {$\vphantom{\ell_0}r_{3}$};
\node[below] at (.1111,.0) {$\vphantom{\ell_0}r_{4}$};
\draw[dashed](-.4189,.0)--(-.4189,.618);
\node[above] at (-.4189,.618) {$\frac{-1}{2+\alpha}$};
\draw[dashed](-.2952,.0)--(-.2952,.618);
\node[above] at (-.2952,.618) {$\frac{-1}{3+\alpha}$};
\draw[dashed](-.2279,.0)--(-.2279,.618);
\node[above] at (-.2279,.618) {$\frac{-1}{4+\alpha}$};
\draw[dashed](.2952,.0)--(.2952,.618);
\node[above] at (.2952,.618) {$\frac{1}{3+\alpha}$};
\node[right] at (.387,.0) {$0$};
\node[right] at (.387,.5) {$1/2$};
\draw[thin,dotted](-.4159,.3333)--(.387,.3333)node[right]{$1/3$};
\draw[thin,dotted](-.613,.25)--(.387,.25)node[right]{$1/4$};
\end{tikzpicture}
\caption{The domain $\Omega_\alpha$ for $\alpha = 113/292 \in \Gamma_{(-1:2)(-1:3)(-1:4)(-1:2)}$, with
$\ell_j = T_\alpha^j(\alpha-1)$ and $r_j = T_\alpha^j(\alpha)$.}
\label{f:2342}
\end{figure}

Now, we can evaluate the measure of $\Omega_\alpha$, $\alpha \in \Gamma_v$,  as a function of the measures of $\Omega_{\eta_v}$ and of $[\alpha, \eta_v] \times \overline{\Psi'_v}$.
When we compare with $\Omega_{\zeta_v}$, it is even sufficient to know the density~$\nu_{\zeta_v}$.   

Compare \cite{Kraaikamp-Schmidt-Smeets:10} for similar arguments.

\begin{proof}[\textbf{Proof of Theorem~\ref{t:muOmega}}]
Let $v = v_{[1,|v|\,]} \in \mathscr{F}$, $v' = v'_{[1,|\widehat{v}|\,]} = {}^{(W)}\widehat{v}{}^{(-1)}$, and compare $\Omega_\alpha$, $\alpha \in [\zeta_v, \eta_v]$, with~$\Omega_{\eta_v}$.
By Lemma~\ref{l:simplerDescr}, Remark~\ref{r:empty} and since $\Psi'_\alpha$ is the disjoint union of the elements of~$\mathscr{C}_\alpha$ (Theorem~\ref{t:shapeOmega}), we have
\begin{align*}
\Omega_\alpha \setminus \Omega_{\eta_v} & = \bigcup_{0 \le j \le |v|} M_{v_{[1,j]}} \cdot [\alpha-1, \eta_v-1] \times N_{v_{[1,j]}} \cdot \Psi_v\ \ \setminus\ \ [\alpha, \eta_v] \times N_{v_{[1,|v|\,]}} \cdot \Psi_v\,, \\
\Omega_{\eta_v} \setminus \Omega_\alpha & = \bigcup_{0 \le j \le |\widehat{v}|} M_{v'_{[1,j]}} \cdot [\alpha, \eta_v] \times N_{v'_{[1,j]}} \cdot \Psi'_v\ \ \setminus\ \ [\alpha, \eta_v] \times N_{v_{[1,|v|\,]}} \cdot \Psi_v\,,
\end{align*}
up to sets of measure zero, thus
\begin{multline*}
\mu(\Omega_\alpha) - \mu(\Omega_{\eta_v}) \\
= \sum_{0 \le j \le |v|} \mu\big(M_{v_{[1,j]}} \cdot [\alpha-1, \eta_v-1] \times N_{v_{[1,j]}} \cdot \Psi_v\big) - \sum_{0 \le j \le |\widehat{v}|} \mu\big(M_{v'_{[1,j]}} \cdot [\alpha, \eta_v] \times N_{v'_{[1,j]}} \cdot \Psi'_v\big)\,.
\end{multline*}
From Theorem~\ref{t:shapeOmega}, we have $\Psi_v = {}^t\hspace{-.1em}E^{-1} \cdot \Psi'_v$, thus
\begin{equation} \label{e:minus1}
\mu\big([{\alpha-1}, \eta_v-1] \times \Psi_v\big) = \mu\big(E \cdot [\alpha, \eta_v] \times {}^t\hspace{-.1em}E^{-1} \cdot \Psi'_v\big) = \mu\big([\alpha, \eta_v] \times \Psi'_v\big)
\end{equation}
by~\eqref{e:mu}.
Applying~\eqref{e:mu} with $M_{v_{[1,j]}}$, $1 \le j \le |v|$, and $M_{v'_{[1,j]}}$, $1 \le j \le |\widehat{v}|$, gives
\[
\mu(\Omega_\alpha) = \mu(\Omega_{\eta_v}) + \big(|v| - |\widehat{v}|\big)\, \mu\big([\alpha-1, \eta_v-1] \times \Psi_v\big)\,,
\]
in particular $\mu(\Omega_{\zeta_v}) = \mu(\Omega_{\eta_v}) + \big(|v| - |\widehat{v}|\big)\, \mu\big([\zeta_v-1, \eta_v-1] \times \Psi_v\big)$.
Therefore, we also have
\[
\mu(\Omega_\alpha) = \mu(\Omega_{\zeta_v}) + \big(|\widehat{v}| - |v|\big)\, \mu\big([\zeta_v-1, \alpha-1] \times \Psi_v\big)\,.
\]
Since $M_{v_{[1,j]}} \cdot (\zeta_v-1) \ge \eta_v-1$ for all $1 \le j \le |v|$ and $M_{v'_{[1,j]}} \cdot \zeta_v \ge \eta_v-1$ for all $1 \le j \le |\widehat{v}|$ by Lemma~\ref{l:Gammav}, Lemma~\ref{l:simplerDescr} yields that the fibers $\Phi_{\zeta_v}(x)$ are equal to $\Psi_v$ for all $x \in [\zeta_v-1, \eta_v-1)$.
This gives
\[
\mu\big([\alpha-1, \eta_v-1] \times \Psi_v\big) = \mu(\Omega_{\zeta_v})\, \nu_{\zeta_v}\big([\zeta_v-1, \alpha-1]\big)\,,
\]
which proves the formula for~$\mu(\Omega_\alpha)$. 
The monotonicity relations for $\alpha \mapsto \mu(\Omega_\alpha)$ are an obvious consequence, and the inverse relations for $\alpha \mapsto h(T_\alpha)$ follow from Theorem~\ref{t:hmu}.
\end{proof}

\section{Continuity of entropy and measure of the natural extension domain}\label{s:continuity}

By Theorem~\ref{t:muOmega}, the normalizing constant $\mu(\Omega_\alpha)$ is continuous on $[\zeta_v, \eta_v]$ for every $v \in \mathscr{F}$.
We now prove that there is a synchronizing interval immediately to the left of~$\Gamma_v$, which implies that $\mu(\Omega_\alpha)$ is continuous on the left of $\zeta_v$ as well.  
Recall that $\Theta(v) := v\,\widehat{v}^{(-1)}$.  

\begin{Lem} \label{l:folding}
For every $v \in \mathscr{F}$, we have $\Theta(v) \in \mathscr{F}$.
The left endpoint of the interval~$\Gamma_v$ is the right endpoint of the interval~$\Gamma_{\Theta(v)}$, i.e., $\zeta_v = \eta_{\Theta(v)}$.
Moreover, 
\[
|\Theta(v)| = |\widehat{\Theta(v)} | = |v| + |\widehat{v}|\,.
\]
\end{Lem}

\begin{proof}
Let $v \in \mathscr{F}$ with characteristic sequence $a_{[1,2\ell+1]}$.
If $|v| = 0$, then $\Theta(v) = (-1:2)$, and all statements are true.
If $v$ is non-empty, then the characteristic sequence of $\Theta(v)$ is 
\[
a_{[1,2\ell+1]}\, a_{[1,2\ell]}\, (a_{2\ell+1}-1)\, 1 \quad \mbox{if}\ a_{2\ell+1} \ge 2, \qquad a_{[1,2\ell+1]}\, a_{[1,2\ell)}\, (a_{2\ell}+1) \quad \mbox{if}\ a_{2\ell+1} = 1\,.
\]
In both cases, we have $\Theta(v) \in \mathscr{F}$ and $\widehat{\Theta(v)} = \widehat{v}\, v^{(+1)}$.
The equality of the lengths of $\Theta(v)$ and $\widehat{\Theta(v)}$ with the sum of those of $v$ and $\widehat{v}$ then follows directly.
The definitions of $\zeta_v$ and $\eta_v$ yield that $\zeta_v = \eta_{\Theta(v)}$.
\end{proof}

\begin{Rmk}\label{r:fold}  One can think of $\Theta(v)$ as giving a folding operation on the set of labels of intervals of synchronizing orbits.    In terms of Remark ~\ref{r:RvLv},  the fixed point of $R_v$ is also of course the fixed point of 
\[ (R_v)^2 = E^{-1} M_{\widehat{v}} W E^{-1} M_{\widehat{v}} W = E^{-1} M_{\widehat{v}} M_v E = M_{v\, \widehat{v} {}^{(-1)} } E = L_{\Theta(v)}\,. \]
Compare this with Figure~\ref{f:rich}.
\end{Rmk}

From Lemma~\ref{l:folding} and Theorem~\ref{t:muOmega}, we deduce the following result.

\begin{Cor}\label{c:foldThenConst}   
If $v \in \mathscr{F}$, then, irrespective of the behavior of the entropy function $\alpha \mapsto h(T_{\alpha})$ on~$\Gamma_v$, this function is constant immediately to the left, that is on $[\eta_{\Theta(v)}, \zeta_v]$.
\end{Cor}

It remains to consider $\alpha \in (0,1] \setminus \Gamma$ that is not the left endpoint of an interval~$\Gamma_v$, $v \in \mathscr{F}$.   For this, let $Z = \{\zeta_v \mid v \in \mathscr{F}\}$.

\begin{Lem} \label{l:delta}
Let  $\alpha \in (0,1] \setminus (\Gamma \cup Z)$. 
For every $n \ge 1$, there exists some $\delta > 0$ such that
\[
\underline{b}{}^{\alpha'}_{[1,n)} = \underline{b}{}^\alpha_{[1,n)} \ \mbox{and} \ \overline{b}{}^{\alpha'}_{[1,n)} = \overline{b}{}^\alpha_{[1,n)} \ \mbox{for all}\ \alpha' \in \left\{\begin{array}{cl}[\alpha, \alpha + \delta) & \mbox{if $\alpha = \eta_v$ for some $v \in \mathscr{F}$,} \\ (\alpha-\delta, \alpha + \delta) & \mbox{else.}\end{array}\right.
\]
\end{Lem}

\begin{proof}
Let $\alpha \in (0,1] \setminus \Gamma$, i.e., $T_\alpha^m({\alpha-1}) < 0$ and $T_\alpha^m(\alpha) < 0$ for all $m \ge 1$.

If $T_\alpha^m({\alpha-1}) > {\alpha-1}$ and $T_\alpha^m(\alpha) > {\alpha-1}$ for all $m \ge 1$, then due to the continuity of $x \mapsto M_w\cdot x$ for general $w$, we clearly have   for each $n \ge 1$, some $\delta > 0$ such that $\underline{b}{}^{\alpha'}_{[1,n)} = \underline{b}{}^\alpha_{[1,n)}$ and $\overline{b}{}^{\alpha'}_{[1,n)} = \overline{b}{}^\alpha_{[1,n)}$ for all $\alpha' \in (\alpha-\delta, \alpha + \delta)$.

If $T_\alpha^m(\alpha) = {\alpha-1}$ for some $m \ge 1$, then $M_{\overline{b}{}^\alpha_{[1,m]}} \cdot \alpha' < \alpha'-1$ for all $\alpha' > \alpha$.
Let $m$ be minimal with this property, then $T_{\alpha'}^m(\alpha') \ge 0$ for all $\alpha' > \alpha$ sufficiently close to $\alpha$, which implies that $\alpha \in Z$.

Finally, suppose that $T_\alpha^m(\alpha) > {\alpha-1}$ for all $m \ge 1$, and $T_\alpha^m({\alpha-1}) = {\alpha-1}$ for some $m \ge 1$.
Similarly to the preceding paragraph, this implies that $\alpha = \eta_v$ for some $v \in \mathscr{F}$.
Now we have, for each $n \ge 1$, some $\delta > 0$ such that $\underline{b}{}^{\alpha'}_{[1,n)} = \underline{b}{}^\alpha_{[1,n)}$ and $\overline{b}{}^{\alpha'}_{[1,n)} = \overline{b}{}^\alpha_{[1,n)}$ for all $\alpha' \in [\alpha, \alpha + \delta)$.
\end{proof}

\begin{proof}[\textbf{Proof of Theorem~\ref{t:continuous}}]
By the remarks of the beginning of the section, we only have to consider the continuity of $\mu(\Omega_\alpha)$ at $\alpha \in (0,1] \setminus (\Gamma \cup Z)$. 
Moreover, we only have to show right continuity if $\alpha = \eta_v$ for some $v \in \mathscr{F}$.
By the monotonicity on every interval~$\Gamma_v$, it suffices to compare $\mu(\Omega_\alpha)$ with~$\mu(\Omega_{\alpha'})$, $\alpha' \in (0,1] \setminus \Gamma$.

If $\{w \in \mathscr{L}'_\alpha:\, |w| < n\} = \{w \in \mathscr{L}'_{\alpha'}:\, |w| < n\}$, $n \ge 2$, then $d_{\alpha'}(\alpha') = d_\alpha(\alpha)$, and Corollary~\ref{c:approximation} yields that
\begin{multline*}
\big|\mu(\Omega_\alpha) - \mu(\Omega_{\alpha'})\big| \le \sum_{\substack{w\in\mathscr{L}'_\alpha:\\|w|<n}} \Big|\mu\big(J^\alpha_w \times N_w \cdot \big[0, \tfrac{1}{d_\alpha(\alpha)+1}\big]\big) - \mu\big(J^{\alpha'}_w \times N_w \cdot \big[0, \tfrac{1}{d_\alpha(\alpha)+1}\big]\big)\Big| \\[-2ex]
+ \big(\tfrac{d_\alpha(\alpha)}{d_\alpha(\alpha)+\alpha}\big)^n\, \log\big(1+\tfrac{1}{\alpha}\big) +  \big(\tfrac{d_\alpha(\alpha)}{d_\alpha(\alpha)+\alpha'}\big)^n\,\log\big(1+\tfrac{1}{\alpha'}\big)\,.
\end{multline*}
Fix $\epsilon > 0$, choose $n \ge 2$ and an interval around~$\alpha$ such that $\big(\frac{d_\alpha(\alpha)}{d_\alpha(\alpha)+\alpha'}\big)^n\, \log\big(1+\frac{1}{\alpha'}\big) < \epsilon/3$ for every $\alpha'$ in this interval.
Lemma~\ref{l:delta} gives some $\delta > 0$ such that ${\{w \in \mathscr{L}'_\alpha:\, |w| < n\}} = \{w \in \mathscr{L}'_{\alpha'}:\, |w| < n\}$ for all $\alpha' \in (\alpha-\delta, \alpha+\delta)$ and $\alpha' \in [\alpha, \alpha+\delta)$ respectively.
Since $\{w \in \mathscr{L}'_\alpha:\, |w| < n\}$ is a finite set, and $J^\alpha_w = M_w \cdot \Delta_\alpha(w)$, $J^{\alpha'}_w = M_w \cdot \Delta_{\alpha'}(w)$ by Lemma~\ref{l:Jw}, 
\[
\sum_{\substack{w\in\mathscr{L}'_\alpha:\\|w|<n}} \Big|\mu\big(J^\alpha_w \times N_w \cdot \big[0, \tfrac{1}{d_\alpha(\alpha)+1}\big]\big) - \mu\big(J^{\alpha'}_w \times N_w \cdot \big[0, \tfrac{1}{d_\alpha(\alpha)+1}\big]\big)\Big| < \frac{\epsilon}{3}
\]
for $\alpha'$ sufficiently close to~$\alpha$.
This shows the continuity of $\alpha \mapsto \mu(\Omega_\alpha)$.
\end{proof}

\section{Constancy of entropy on $[g^2,g]$} \label{sec:constancy-entropy-g2}

Lemma~\ref{l:longConstIntvl}, Theorems~\ref{t:continuous} and~\ref{t:endpoints} show that the entropy is constant on intervals covering almost all points in $[g^2,g]$.  To show that the entropy is constant on the whole interval $[g^2,g]$, we must exclude that the function $\alpha \mapsto h(T_\alpha)$  forms a ``devil's staircase''.
To this end, we improve some of the previous~results.

For simplicity, we assume in the following proposition that $\alpha \in (0,1] \setminus \Gamma$ although the statement can be proved for general $\alpha \in (0,1]$. 
Note that this description, together with Lemma~\ref{l:Ualpha}, is useful for drawing figures approximating the natural extension domains.

\begin{Prop} \label{p:dalpha}
For any $\alpha \in (0,1]\setminus \Gamma$, we have 
\[
\bigcup_{\substack{w\in\mathscr{L}'_\alpha:\\\overline{b}{}^\alpha_1w\in\mathscr{L}'_\alpha}} J^\alpha_{\overline{b}{}^\alpha_1w} \times N_w \cdot \big[0, \tfrac{1}{d_\alpha(\alpha)}\big]\ \subset\ \Omega_\alpha\,.
\]
\end{Prop}

\begin{proof} 
Let $\alpha \in (0,1]\setminus \Gamma$, and $a_{[1,\infty)}$ be the characteristic sequence of~$\alpha-1$.
We first prove that $J^\alpha_{\overline{b}{}^\alpha_1} \times \big[0, \tfrac{1}{d_\alpha(\alpha)}\big] \subset \Omega_\alpha$. 
We already know both that $\mathbb{I}_{\alpha} \times \big[0, \tfrac{1}{d_\alpha(\alpha)+1}\big] \subset \Omega_\alpha$ and $\big[0, \tfrac{1}{d_\alpha(\alpha)}\big] \subset \Psi'_{\alpha}$, with $\Psi'_{\alpha}$ being the closure of $\bigcup_{w' \in \mathscr{L}'_\alpha} N_{w'}\cdot \big[0, \tfrac{1}{d_\alpha(\alpha)+1}\big]$.
It thus suffices to show that $J^\alpha_{\overline{b}{}^\alpha_1} \subseteq J^\alpha_{w'}$ for all $w' \in \mathscr{L}'_\alpha$ with $N_{w'} \cdot \big[0, \tfrac{1}{d_\alpha(\alpha)+1}\big]\, \cap\,  \big(\tfrac{1}{d_\alpha(\alpha)+1}, \tfrac{1}{d_\alpha(\alpha)}\big) \ne \emptyset$, i.e., for all $w'$ ending with $(-1:d_\alpha(\alpha)+1)$ or~$\overline{b}{}^\alpha_1$.
If $w'$ ends with~$\overline{b}{}^\alpha_1$, then $J^\alpha_{w'} = J^\alpha_{\overline{b}{}^\alpha_1}$;  thus we need consider only $w'$ ending with $(-1:d_\alpha(\alpha)+1)$.
Furthermore,  we need only consider $w' \in \mathscr{L}'_\alpha \setminus \mathscr{L}_\alpha$, since $J^\alpha_{w'} = [\alpha-1,\alpha)$ otherwise. 
This means that $w' \in \mathscr{L}_\alpha\, \underline{b}{}^\alpha_{[1,j]}$ for some $j \ge 1$ or $w' \in \mathscr{L}'_\alpha\, \overline{b}{}^\alpha_{[1,j]}$ for some $j \ge 2$.
Let first $w' \in \mathscr{L}_\alpha\, \underline{b}{}^\alpha_{[1,j]}$.
Since $\alpha \in (0,1]\setminus \Gamma$, we have $d_\alpha(\alpha) \ge 2$, thus $w'$ does not end with~$(-1:2)$.
Therefore, the characteristic sequence of $\underline{b}{}^\alpha_{[j+1,\infty)}$ is $a_{[2n+1,\infty)}$ for some $n \ge 1$, and that of $\underline{b}{}^\alpha_{[j,\infty)}$ is $1\, a_{[2n,\infty)}$, with $a_{2n} = d_\alpha(\alpha)-1 = a_1$.
Since $a_{[2n,\infty)} \le_{\mathrm{alt}} a_{[1,\infty)}$, we obtain that $a_{[2n+1,\infty)} \ge_{\mathrm{alt}} a_{[2,\infty)}$, thus $T_\alpha^j(\alpha-1) \le T_\alpha(\alpha)$, i.e., $J^\alpha_{\overline{b}{}^\alpha_1} \subseteq J^\alpha_{w'}$.
If $w' \in \mathscr{L}_\alpha\, \overline{b}{}^\alpha_{[1,j]}$, then the characteristic sequences of $\overline{b}{}^\alpha_{[j,\infty)}$ and $\overline{b}{}^\alpha_{[j+1,\infty)}$ are $1\, a_{[2n-1,\infty)}$ and $a_{[2n,\infty)}$ respectively for some $n \ge 2$, with $a_{2n-1} = a_1$, thus we obtain that $T_\alpha^j(\alpha) \le T_\alpha(\alpha)$.
Therefore, $J^\alpha_{\overline{b}{}^\alpha_1} \subseteq J^\alpha_{w'}$ holds for all $w' \in \mathscr{L}'_\alpha$ ending with $(-1:d_\alpha(\alpha)+1)$ or~$\overline{b}{}^\alpha_1$, hence $J^\alpha_{\overline{b}{}^\alpha_1} \times \big[0, \tfrac{1}{d_\alpha(\alpha)}\big] \subset \Omega_\alpha$. 

From $J^\alpha_{\overline{b}{}^\alpha_1} \times \big[0, \tfrac{1}{d_\alpha(\alpha)}\big] \subset \Omega_\alpha$, we infer that $J^\alpha_{\overline{b}{}^\alpha_1 w} \times N_w \cdot \big[0, \tfrac{1}{d_\alpha(\alpha)}\big] \subset \mathcal{T}_\alpha^{|w|}\big(J^\alpha_{\overline{b}{}^\alpha_1} \times \big[0, \tfrac{1}{d_\alpha(\alpha)}\big]\big) \subset \Omega_\alpha$ for any $w \in \mathscr{L}'_\alpha$ with $\overline{b}{}^\alpha_1 w \in \mathscr{L}'_\alpha$.
\end{proof}

\begin{Lem} \label{l:positivemeasure2}
For any $\alpha \in [g^2, \sqrt{2}-1]$, we have 
\[
\Omega_\alpha\ \subset\ \mathbb{I}_\alpha \times \big[0, g^2\big]\ \cup\  \big[T_\alpha(\alpha-1), \alpha\big] \times \big(\big[0, \tfrac{1}{3-g}\big] \cup \big[\tfrac{1}{2}, g\big]\big)\,.
\]
\end{Lem}

\begin{proof}
Since $\underline{b}{}^{g^2}_{[1,\infty)} = (-1:2)\, (-1:3)^\omega$, no word in $(-1:2)\, (-1:3)^*\, (-1:2)$ occurs in $\underline{b}{}^\alpha_{[1,\infty)}$ for $\alpha \ge g^2$.
Hence the maximal height of a fiber in $\Omega_\alpha$ is $\lim_{n\to\infty} N_{(-1:3)^n\, (-1:2)} \cdot 0 = g$, i.e., $\Omega_\alpha \subseteq \mathbb{I}_\alpha \times [0,g]$.
Further, since $(+1:2) \notin \mathscr L'_{\alpha}$ and  $N_{(-1:3)} \cdot \Psi'_\alpha \subseteq N_{(-1:3)} \cdot [0,g]$, we have that  $\big(\frac{1}{3-g}, \frac{1}{2}\big) \cap \Psi'_\alpha = \emptyset$.
For $x < T_\alpha(\alpha-1) = M_{(-1:2)} \cdot (\alpha-1)$ and $n \ge 0$, we have $M_{(-1:2)\, (-1:3)^n} \cdot x < M_{(-1:2)\, (-1:3)^n\, (-1:2)} \cdot (\alpha-1) < \alpha-1$, thus $\max\{y \mid (x,y) \in \Omega_\alpha\} = \lim_{n\to\infty} N_{(-1:3)^n} \cdot 0 = g^2$ for $x \in [\alpha-1, T_\alpha(\alpha-1))$.
\end{proof}

With a little more effort, it can be shown that $\mathbb{I}_\alpha \times \big[0, \frac{1}{3+g}\big]\ \cup\  \big[T_\alpha(\alpha), \alpha\big] \times \big[0, g^2\big] \subset \Omega_\alpha$ for any $\alpha \in [g^2, \sqrt{2}-1]$.
However, the statement of Lemma~\ref{l:positivemeasure2} is sufficient for the following.

Instead of the sets $\Xi_{\alpha,n}$ defined in~\eqref{e:defUpsilon}, we study now
\[
\Xi'_{\alpha,n} := \Xi_{\alpha,n}\ \setminus \hspace{-1em} \bigcup_{\substack{w \in \mathscr{L}'_\alpha:\\ \overline{b}{}^\alpha_1 w \in \mathscr{L}'_\alpha,\, |w|<n}} \hspace{-1em} J^\alpha_{\overline{b}{}^\alpha_1 w} \times N_w \cdot \big[\tfrac{1}{4}, \tfrac{1}{3}\big] \qquad (n \ge 0).
\]

\begin{Lem}[cf.\ Lemma~\ref{l:exponentialbound}] \label{l:exponentialbound2}
Let $\alpha \in [g^2, \sqrt{2}-1] \setminus \Gamma$.
Then we have, for any $n \ge 0$,
\[
\mu(\Xi'_{\alpha,n}) \le \mu(\Omega_\alpha)\, \big(\tfrac{1}{\sqrt{5}}\big)^n.
\]
\end{Lem}

\begin{proof}
The proof runs along the same lines as that of Lemma~\ref{l:exponentialbound}.
For any $n \ge 0$,  we have
\[
\Xi'_{\alpha,n} \setminus \Xi'_{\alpha,n+1} = \Bigg(\big(\Xi_{\alpha,n} \setminus \Xi_{\alpha,n+1}\big)\, \cup\, \bigg(\Xi_{\alpha,n+1}\, \cap \hspace{-1.5em} \bigcup_{\substack{w \in \mathscr{L}'_\alpha:\\ \overline{b}{}^\alpha_1 w \in \mathscr{L}'_\alpha,\, |w|=n}} \hspace{-1.8em} J^\alpha_{\overline{b}{}^\alpha_1 w} \times N_w \cdot \big[\tfrac{1}{4}, \tfrac{1}{3}\big]\bigg)\Bigg)\, \setminus \hspace{-1.5em} \bigcup_{\substack{w \in \mathscr{L}'_\alpha:\\ \overline{b}{}^\alpha_1 w \in \mathscr{L}'_\alpha,\, |w|<n}} \hspace{-1.8em} J^\alpha_{\overline{b}{}^\alpha_1 w} \times N_w \cdot \big[\tfrac{1}{4}, \tfrac{1}{3}\big]\,.
\]
We show first that the intersection with $\Xi_{\alpha,n+1}$ can be omitted in this equation.
Let $w \in \mathscr{L}'_\alpha$ with $\overline{b}{}^\alpha_1 w \in \mathscr{L}'_\alpha$, $|w| = n$.
Proposition~\ref{p:dalpha} shows that $J^\alpha_{\overline{b}{}^\alpha_1 w} \times N_w \cdot \big[\tfrac{1}{4}, \tfrac{1}{3}\big] \subset \Omega_\alpha$.
The set $N_w \cdot \big(\tfrac{1}{4}, \tfrac{1}{3}\big]$ is disjoint from $N_{w'} \cdot \big[0, \tfrac{1}{4}\big]$ for every $w' \in \mathscr{L}'_\alpha$ with $|w| = n$.
If, for $w' \in \mathscr{L}'_\alpha$ with $|w'| < n$, $N_{w'} \cdot \big[0, \tfrac{1}{4}\big]$ overlaps with $N_w \cdot \big(\tfrac{1}{4}, \tfrac{1}{3}\big]$, then it also overlaps with $N_w \cdot \big(0, \tfrac{1}{4})$, contradicting Theorem~\ref{t:shapeOmega}.
Therefore, we have $J^\alpha_{\overline{b}{}^\alpha_1 w} \times N_w \cdot \big[\tfrac{1}{4}, \tfrac{1}{3}\big] \subset \Xi_{\alpha,n+1}$ (up to a set of measure zero) for every $w \in \mathscr{L}'_\alpha$ with $\overline{b}{}^\alpha_1 w \in \mathscr{L}'_\alpha$, $|w| = n$, thus
\[
\Xi'_{\alpha,n} \setminus \Xi'_{\alpha,n+1} = \Bigg(\big(\Xi_{\alpha,n} \setminus \Xi_{\alpha,n+1}\big)\, \cup\, \hspace{-1.5em} \bigcup_{\substack{w \in \mathscr{L}'_\alpha:\\ \overline{b}{}^\alpha_1 w \in \mathscr{L}'_\alpha,\, |w|=n}} \hspace{-1.8em} J^\alpha_{\overline{b}{}^\alpha_1 w} \times N_w \cdot \big[\tfrac{1}{4}, \tfrac{1}{3}\big]\Bigg)\, \setminus \hspace{-1.5em} \bigcup_{\substack{w \in \mathscr{L}'_\alpha:\\ \overline{b}{}^\alpha_1 w \in \mathscr{L}'_\alpha,\, |w|<n}} \hspace{-1.8em} J^\alpha_{\overline{b}{}^\alpha_1 w} \times N_w \cdot \big[\tfrac{1}{4}, \tfrac{1}{3}\big]\,.
\]
Let $X_{\alpha,n} := \bigcup_{w \in \mathscr{L}'_\alpha:\, |w|=n} \Delta_\alpha(w)$ as in the proof of Lemma~\ref{l:exponentialbound}, and set
\[
X'_{\alpha,n}\ :=\ X_{\alpha,n}\ \setminus \hspace{-1em} \bigcup_{\substack{w \in \mathscr{L}'_\alpha:\\ \overline{b}{}^\alpha_1 w \in \mathscr{L}'_\alpha,\, |w|<n}} \hspace{-1em} T_\alpha^{|w|+1-n}\big(\Delta\big(\overline{b}{}^\alpha_1 w\big)\big)\,.
\]
By arguments in the proofs of Lemmas~\ref{l:Ualpha} and~\ref{l:exponentialbound} and since $T_\alpha\big(\Delta\big(\overline{b}{}^\alpha_1 w\big)\big) = J^\alpha_{\overline{b}{}^\alpha_1} \cap \Delta_\alpha(w)$, we obtain that
\[
X'_{\alpha,n} \times \big[0, \tfrac{1}{4}\big]\ \cup\ \big(X'_{\alpha,n} \cap J^\alpha_{\overline{b}{}^\alpha_1}\big) \times \big[\tfrac{1}{4}, \tfrac{1}{3}\big]\ \subset\ \mathcal{T}_\alpha^{-n}\big(\Xi'_{\alpha,n} \setminus \Xi'_{\alpha,n+1}\big)\,.
\]
Since $\mathcal{T}_\alpha^{-n}(\Xi'_{\alpha,n}) \subset X_{\alpha,n} \times [0,1]$ and 
\[
\mathcal{T}_\alpha^n\Big(T_\alpha^{|w|+1-n}(\Delta(\overline{b}{}^\alpha_1 w)) \times [0,1]\, \cap\, \Omega_\alpha\Big) \subset J^\alpha_{\overline{b}{}^\alpha_1 w} \times N_{\overline{b}{}^\alpha_1 w} \cdot [0,1] = J^\alpha_{\overline{b}{}^\alpha_1 w} \times N_w \cdot \big[\tfrac{1}{4}, \tfrac{1}{3}\big]
\]
for any $w \in \mathscr{L}'_\alpha$ with $|w| < n$, we have
\[
\mathcal{T}_\alpha^{-n}(\Xi'_{\alpha,n}) \subset X'_{\alpha,n} \times [0,1]\ \cap\ \Omega_\alpha\,.
\]
With Lemma~\ref{l:positivemeasure2} and $T_\alpha(\alpha) \le -g^2 \le T_\alpha(\alpha-1)$, i.e., $J^\alpha_{\overline{b}{}^\alpha_1} \supseteq J^\alpha_{\underline{b}{}^\alpha_1}$, we obtain that
\[
\frac{\mu(\Xi'_{\alpha,n} \setminus \Xi'_{\alpha,n+1})}{\mu(\Xi'_{\alpha,n})} \ge \frac{\mu\big(\big(X'_{\alpha,n} \setminus J^\alpha_{\underline{b}{}^\alpha_1}\big) \times \big[0, \frac{1}{4}\big]\big) + \mu\big(\big(X'_{\alpha,n} \cap J^\alpha_{\underline{b}{}^\alpha_1}\big) \times \big[0, \frac{1}{3}\big]\big)}{\mu\big(\big(X'_{\alpha,n} \setminus J^\alpha_{\underline{b}{}^\alpha_1}\big) \times \big[0, g^2\big]\big) + \mu\big(\big(X'_{\alpha,n} \cap J^\alpha_{\underline{b}{}^\alpha_1}\big) \times \big(\big[0, \frac{1}{3-g}\big] \cup \big[\frac{1}{2}, g\big]\big)\big)}\,.
\]
Using that $\frac{p+p'}{q+q'} \ge \min\big\{\frac{p}{q}, \frac{p'}{q'}\big\}$ for all $p,p',q,q'>0$, the estimates
\begin{gather*}
\frac{\mu\big(\big(X'_{\alpha,n} \setminus J^\alpha_{\underline{b}{}^\alpha_1}\big) \times \big[0,\frac{1}{4}\big]\big)}{\mu\big(\big(X'_{\alpha,n} \setminus J^\alpha_{\underline{b}{}^\alpha_1}\big) \times \big[0,g^2\big]\big)} \ge \min_{x \in \mathbb{I}_\alpha} \frac{\int_0^{1/4} \frac{1}{(1+xy)^2}\, dy}{\int_0^{g^2} \frac{1}{(1+xy)^2}\, dy} = \min_{x \in \mathbb{I}_\alpha} \frac{x+2+g}{x+4} \ge \frac{2}{4-g}\,, \\
\frac{\mu\big(\big(X'_{\alpha,n} \cap J^\alpha_{\underline{b}{}^\alpha_1}\big) \times \big[0, \frac{1}{3}\big]\big)}{\mu\big(\big(X'_{\alpha,n} \cap J^\alpha_{\underline{b}{}^\alpha_1}\big) \times \big(\big[0, \frac{1}{3-g}\big] \cup \big[\frac{1}{2}, g\big]\big)\big)} \ge \min_{x \in [T_\alpha(\alpha-1), \alpha]} \frac{\frac{1}{x+3}}{\frac{1}{x+1+g}-\frac{1}{x+2}+\frac{1}{x+3-g}} \ge \frac{2g}{2g+1}\,,
\end{gather*}
yield that 
\[
\frac{\mu(\Xi'_{\alpha,n+1})}{\mu(\Xi'_{\alpha,n})} = 1 - \frac{\mu(\Xi'_{\alpha,n} \setminus \Xi'_{\alpha,n+1})}{\mu(\Xi'_{\alpha,n})} \le 1 - \min\Big\{\frac{2}{4-g},\,\frac{2g}{2g+1}\Big\} = \frac{1}{2g+1} = \frac{1}{\sqrt{5}}\,.
\]
Since $\Xi'_{\alpha,0} = \Omega_\alpha$, this proves the lemma.
\end{proof}

\begin{Lem} \label{l:Psidifference}
There exist constants $C_1, C_2 > 0$ such that
\[
\mu\big(\big(\mathbb{I}_\alpha \cup \mathbb{I}_{\alpha'}\big) \times \big(\Psi_\alpha \setminus \Psi_{\alpha'}\big)\big) \le C_1\, \big(\tfrac{1}{\sqrt{5}}\big)^n\,, \quad \mu\big(\big(\mathbb{I}_\alpha \cup \mathbb{I}_{\alpha'}\big) \times \big(\Psi'_\alpha \setminus \Psi'_{\alpha'}\big)\big) \le C_2\, \big(\tfrac{1}{\sqrt{5}}\big)^n\,, 
\]
for all $\alpha, \alpha' \in [g^2, \sqrt{2}-1] \setminus \Gamma$, $n \ge 1$, such that $\underline{b}{}^{\alpha'}_{[1,n)} = \underline{b}{}^\alpha_{[1,n)}$.
\end{Lem}

\begin{proof}
For any $\alpha, \alpha' \in [g^2, \sqrt{2}-1] \setminus \Gamma$, $n \ge 1$, with $\underline{b}{}^{\alpha'}_{[1,n)} = \underline{b}{}^\alpha_{[1,n)}$, we also have $\overline{b}{}^{\alpha'}_{[1,n)} = \overline{b}{}^\alpha_{[1,n)}$. 
Therefore, Lemma~\ref{l:Ualpha} and Proposition~\ref{p:dalpha} yield that
\[
\Psi'_\alpha \setminus \Psi'_{\alpha'} \subset Y'_{\alpha,n} := \overline{\bigcup_{\substack{w \in \mathscr{L}'_\alpha:\\ |w|\ge n}} N_w \cdot \big[0, \tfrac{1}{4}\big]\ \setminus \hspace{-1em} \bigcup_{\substack{w \in \mathscr{L}'_\alpha:\\ \overline{b}{}^\alpha_1 w \in \mathscr{L}'_\alpha,\, |w|<n}} \hspace{-1em} N_w \cdot \big[\tfrac{1}{4}, \tfrac{1}{3}\big]}\,.
\]
Since $[0, \alpha] \times Y'_{\alpha,n} \subset \Xi'_{\alpha,n}$ and $\mu(\Xi'_{\alpha,n}) \le \mu(\Omega_\alpha)\, \big(\tfrac{1}{\sqrt{5}}\big)^n$ by Lemma~\ref{l:exponentialbound2}, there exists a constant $C_2 > 0$ such that $\mu\big(\big(\mathbb{I}_\alpha \cup \mathbb{I}_{\alpha'}\big) \times \big(\Psi'_\alpha \setminus \Psi'_{\alpha'}\big)\big) \le C_2\, \big(\tfrac{1}{\sqrt{5}}\big)^n$, cf.\ the proof of Lemma~\ref{l:exponentialbound}.
As in~\eqref{e:minus1}, we have $\mu\big(\big(\mathbb{I}_\alpha \cup \mathbb{I}_{\alpha'}\big) \times \big(\Psi_\alpha \setminus \Psi_{\alpha'}\big)\big) = \mu\big(\big(\big[\alpha, \alpha+1] \cup [\alpha', \alpha'+1]\big) \times \big(\Psi'_\alpha \setminus \Psi'_{\alpha'}\big)\big)$, which yields the constant~$C_1$.
\end{proof}

In view of the equation $\Omega_\alpha = \overline{\bigcup_{j\ge 0} J^\alpha_{\underline{b}{}^\alpha_{[1,j]}} \times N_{\underline{b}{}^\alpha_{[1,j]}} \cdot \Psi_\alpha}\, \cup\, \overline{\bigcup_{j\ge 1} J^\alpha_{\overline{b}{}^\alpha_{[1,j]}} \times N_{\overline{b}{}^\alpha_{[1,j]}} \cdot \Psi'_\alpha}$, which holds for $\alpha \in (0,1] \setminus \Gamma$ by Theorem~\ref{t:shapeOmega}, we consider, for any $n \ge 1$, 
\[
\Upsilon_{\alpha,n} := \overline{\bigcup_{j\ge n} J^\alpha_{\underline{b}{}^\alpha_{[1,j]}} \times N_{\underline{b}{}^\alpha_{[1,j]}} \cdot \Psi_\alpha\ \cup\ J^\alpha_{\overline{b}{}^\alpha_{[1,j]}} \times N_{\overline{b}{}^\alpha_{[1,j]}} \cdot \Psi'_\alpha}\,. 
\]

\begin{Lem} \label{l:Upsilon}
There exists a constant $C_3 > 0$ such that
\[
\mu(\Upsilon_{\alpha,n}) \le C_3\, (3g^5)^n
\]
for all $\alpha \in [g^2, \sqrt{2}-1] \setminus \Gamma$, $n \ge 1$.
\end{Lem}

\begin{proof}
For any $n \ge 0$, we have
\[
\frac{\mu\big(J^\alpha_{\underline{b}{}^\alpha_{[1,n+1]}} \times N_{\underline{b}{}^\alpha_{[1,n+1]}} \cdot \Psi_\alpha\big)}{\mu\big(J^\alpha_{\underline{b}{}^\alpha_{[1,n]}} \times N_{\underline{b}{}^\alpha_{[1,n]}} \cdot \Psi_\alpha\big)} = \frac{\mu\big(\big(J^\alpha_{\underline{b}{}^\alpha_{[1,n]}} \cap \Delta_\alpha\big(\underline{b}{}^\alpha_{n+1}\big)\big) \times N_{\underline{b}{}^\alpha_{[1,n]}} \cdot \Psi_\alpha\big)}{\mu\big(J^\alpha_{\underline{b}{}^\alpha_{[1,n]}} \times N_{\underline{b}{}^\alpha_{[1,n]}} \cdot \Psi_\alpha\big)}
\]
by~\eqref{e:mu}.
If $\underline{b}{}^\alpha_{n+1} = (-1:2)$, i.e., $T_\alpha^n(\alpha-1) \in \big[\alpha-1, \frac{-1}{2+\alpha}\big)$, then 
\begin{multline*}
\frac{\mu\big(\big(J^\alpha_{\underline{b}{}^\alpha_{[1,n]}} \cap \Delta_\alpha\big(\underline{b}{}^\alpha_{n+1}\big)\big) \times N_{\underline{b}{}^\alpha_{[1,n]}} \cdot \Psi_\alpha\big)}{\mu\big(J^\alpha_{\underline{b}{}^\alpha_{[1,n]}} \times N_{\underline{b}{}^\alpha_{[1,n]}} \cdot \Psi_\alpha\big)} \le \min_{y\in[0,g^2]} \frac{\int_{T_\alpha^n(\alpha-1)}^{-1/(2+\alpha)} \frac{1}{(1+xy)^2\,dx}}{\int_{T_\alpha^n(\alpha-1)}^\alpha \frac{1}{(1+xy)^2\,dx}} \\
= \min_{y\in[0,g^2]} \frac{1+(2+\alpha)\, T_\alpha^n(\alpha-1)}{T_\alpha^n(\alpha-1)-\alpha} \frac{1+\alpha y}{2+\alpha-y} = \frac{1+(2+\alpha){T_\alpha^n(\alpha-1)}}{T_\alpha^n(\alpha-1)-\alpha} \frac{1+\alpha g^2}{1+g+\alpha} \\
\le \frac{(1-\alpha-\alpha^2) (1+\alpha g^2)}{1+g+\alpha} \le 3g^5.
\end{multline*}
If $\underline{b}{}^\alpha_{n+1} = (-1:3)$, i.e., $T_\alpha^n(\alpha-1) \in \big[\frac{-1}{2+\alpha}, \frac{-1}{3+\alpha}\big)$, then 
\[
\frac{\mu\big(\big(J^\alpha_{\underline{b}{}^\alpha_{[1,n]}} \cap \Delta_\alpha\big(\underline{b}{}^\alpha_{n+1}\big)\big) \times N_{\underline{b}{}^\alpha_{[1,n]}} \cdot \Psi_\alpha\big)}{\mu\big(J^\alpha_{\underline{b}{}^\alpha_{[1,n]}} \times N_{\underline{b}{}^\alpha_{[1,n]}} \cdot \Psi_\alpha\big)} \le \min_{y\in[0,g]} \frac{1+\alpha y}{(\alpha+1)^2\, (3+\alpha-y)} \le \frac{g}{(1+g^2)^3}\,.
\]
If $\underline{b}{}^\alpha_{n+1} = (-1:4)$, i.e., $T_\alpha^n(\alpha-1) \in \big[\frac{-1}{3+\alpha}, \frac{-1}{4+\alpha}\big)$, then 
\[
\frac{\mu\big(\big(J^\alpha_{\underline{b}{}^\alpha_{[1,n]}} \cap \Delta_\alpha\big(\underline{b}{}^\alpha_{n+1}\big)\big) \times N_{\underline{b}{}^\alpha_{[1,n]}} \cdot \Psi_\alpha\big)}{\mu\big(J^\alpha_{\underline{b}{}^\alpha_{[1,n]}} \times N_{\underline{b}{}^\alpha_{[1,n]}} \cdot \Psi_\alpha\big)} \le \min_{y\in[0,g]} \frac{1+\alpha y}{(\alpha^2+3\alpha+1)\, (4+\alpha-y)} \le \frac{1}{3(7g-2)}\,.
\]
We obtain that
\[
\frac{\mu\big(J^\alpha_{\underline{b}{}^\alpha_{[1,n+1]}} \times N_{\underline{b}{}^\alpha_{[1,n+1]}} \cdot \Psi_\alpha\big)}{\mu\big(J^\alpha_{\underline{b}{}^\alpha_{[1,n]}} \times N_{\underline{b}{}^\alpha_{[1,n]}} \cdot \Psi_\alpha\big)} \le \max\Big\{3g^5, \frac{g}{(1+g^2)^3}, \frac{1}{3(7g-2)}\Big\} = 3g^5 \approx 0.2705\,.
\]
In the same way, we get that
\[
\mu\big(J^\alpha_{\overline{b}{}^\alpha_{[1,n+1]}} \times N_{\overline{b}{}^\alpha_{[1,n+1]}} \cdot \Psi'_\alpha\big) \le 3g^5\, \mu\big(J^\alpha_{\overline{b}{}^\alpha_{[1,n]}} \times N_{\overline{b}{}^\alpha_{[1,n]}} \cdot \Psi'_\alpha\big) 
\]
for all $n \ge 1$. 
Since the elements of $\mathscr{C}_\alpha$ are disjoint (Theorem~\ref{t:shapeOmega}) and the closure in the definition of $\Upsilon_{\alpha,n}$ does not increase the measure by Lemma~\ref{l:exponentialbound}, this implies that
\begin{align*}
\mu(\Upsilon_{\alpha,n}) & = \sum_{j=n}^\infty \Big(\mu\big(J^\alpha_{\underline{b}{}^\alpha_{[1,j]}} \times N_{\underline{b}{}^\alpha_{[1,j]}} \cdot \Psi_\alpha\big) + \mu\big(J^\alpha_{\overline{b}{}^\alpha_{[1,j]}} \times N_{\overline{b}{}^\alpha_{[1,j]}} \cdot \Psi'_\alpha\big)\Big) \\
& \le \sum_{j=n}^\infty (3g^5)^j \Big(\mu\big(\mathbb{I}_\alpha \times \Psi_\alpha\big) + \tfrac{1}{3g^5}\, \mu\big(J^\alpha_{\overline{b}{}^\alpha_1} \times N_{\overline{b}{}^\alpha_1} \cdot \Psi'_\alpha\big)\Big) \le C_3\, (3g^5)^n
\end{align*}
for some constant $C_3 > 0$.
\end{proof}

\begin{Lem} \label{l:mudifference}
There exists a constant $C_4 > 0$ such that
\[
\big|\mu(\Omega_{\alpha'}) - \mu(\Omega_\alpha)\big| \le C_4\, n\, \big(\tfrac{1}{\sqrt{5}}\big)^n\,.
\]
for all $\alpha, \alpha' \in [g^2, \sqrt{2}-1] \setminus \Gamma$, $n \ge 1$, such that $\underline{b}{}^{\alpha'}_{[1,n)} = \underline{b}{}^\alpha_{[1,n)}$. 
\end{Lem}

\begin{proof}
For any $\alpha, \alpha' \in [g^2, \sqrt{2}-1] \setminus \Gamma$, $n \ge 1$, such that $\underline{b}{}^{\alpha'}_{[1,n)} = \underline{b}{}^\alpha_{[1,n)}$, let 
\begin{equation}
\Omega_{\alpha,\alpha',n} := \bigcup_{0 \le j < n} J^{\alpha'}_{\underline{b}{}^\alpha_{[1,j]}} \times N_{\underline{b}{}^\alpha_{[1,j]}} \cdot \Psi_\alpha\ \cup\ \bigcup_{1 \le j < n} J^{\alpha'}_{\overline{b}{}^\alpha_{[1,j]}} \times N_{\overline{b}{}^\alpha_{[1,j]}} \cdot \Psi'_\alpha\,.
\end{equation}
Since $\underline{b}{}^{\alpha'}_{[1,n)} = \underline{b}{}^\alpha_{[1,n)}$ implies $\overline{b}{}^{\alpha'}_{[1,n)} = \overline{b}{}^\alpha_{[1,n)}$, Lemma~\ref{l:Upsilon} yields that
\[
0 \le \mu(\Omega_\alpha) - \mu(\Omega_{\alpha,\alpha,n}) = \mu(\Upsilon_{\alpha,n}) \le C_3\, (3g^5)^n\,.
\]
For $\alpha < \alpha'$, we obtain similarly to the proof of Theorem~\ref{t:muOmega} that
\begin{multline*}
\mu(\Omega_{\alpha,\alpha',n}) - \mu(\Omega_{\alpha,\alpha,n}) \\
= \sum_{j=0}^{n-1} \Big(\mu\big(M_{\underline{b}{}^\alpha_{[1,j]}} \cdot [\alpha-1, \alpha'-1] \times N_{\underline{b}{}^\alpha_{[1,j]}} \cdot \Psi_\alpha\big) - \mu\big([\alpha, \alpha'] \times N_{\underline{b}{}^\alpha_{[1,j]}} \cdot \Psi_\alpha\big)\Big) \\
- \sum_{j=1}^{n-1} \Big(\mu\big(M_{\overline{b}{}^\alpha_{[1,j]}} \cdot [\alpha, \alpha'] \times N_{\overline{b}{}^\alpha_{[1,j]}} \cdot \Psi'_\alpha\big) + \mu\big([\alpha, \alpha'] \times N_{\overline{b}{}^\alpha_{[1,j]}} \cdot \Psi'_\alpha\big)\Big) \\
= \mu\big([\alpha, \alpha'] \times \Psi'_\alpha\big) - \sum_{j=0}^{n-1} \mu\big([\alpha, \alpha'] \times N_{\underline{b}{}^\alpha_{[1,j]}} \cdot \Psi_\alpha\big) - \sum_{j=1}^{n-1} \mu\big([\alpha, \alpha'] \times N_{\overline{b}{}^\alpha_{[1,j]}} \cdot \Psi'_\alpha\big)\,.
\end{multline*}
Since this quantity is equal to $\mu\big([\alpha, \alpha'] \times Y_{\alpha,n}\big)$, where $Y_{\alpha,n}$ is the projection of $\Upsilon_{\alpha,n}$ to the $y$-axis, there exists a constant $C_5 > 0$ such that
\[
0 \le \mu(\Omega_{\alpha,\alpha',n}) - \mu(\Omega_{\alpha,\alpha,n}) \le C_5\, (3g^5)^n\,.
\]
It remains to compare $\mu(\Omega_{\alpha,\alpha',n})$ with $\mu(\Omega_{\alpha',\alpha',n})$.
We have
\begin{align*}
\mu(\Omega_{\alpha',\alpha',n} \setminus \Omega_{\alpha,\alpha',n})  & \le  \sum_{j=0}^{n-1} \mu\big(J^{\alpha'}_{\underline{b}{}^\alpha_{[1,j]}} \times N_{\underline{b}{}^\alpha_{[1,j]}} \cdot \big(\Psi_{\alpha'} \setminus \Psi_\alpha\big)\big) + \!\sum_{j=1}^{n-1} \mu\big(J^{\alpha'}_{\overline{b}{}^\alpha_{[1,j]}} \times N_{\overline{b}{}^\alpha_{[1,j]}} \cdot \big(\Psi'_{\alpha'} \setminus \Psi'_\alpha\big)\big) \\
& \le \big(C_1\, n + C_2\, (n-1)\big)\, \big(\tfrac{1}{\sqrt{5}}\big)^n
\end{align*}
by Lemma~\ref{l:Psidifference}, thus $\big|\mu(\Omega_{\alpha',\alpha',n}) - \mu(\Omega_{\alpha,\alpha',n})\big| \le \big(C_1\, n + C_2\, (n-1)\big)\, \big(\tfrac{1}{\sqrt{5}}\big)^n$.
Putting all estimates together yields the lemma. 
\end{proof}

\begin{Lem} \label{l:2n}
For any $n \ge 1$, we have
\[
\#\big\{\underline{b}{}^\alpha_{[1,n]} \mid \alpha \in [g^2, g) \setminus \Gamma\big\} \le 2^n\,. 
\]
\end{Lem}

\begin{proof}
It follows from the proof of Lemma~\ref{l:longConstIntvl} that
\[
\bigcup_{n\ge1} \big\{\underline{b}{}^\alpha_{[1,n]} \mid \alpha \in [g^2, g) \setminus \Gamma\big\} \subset \big((-1:2) (-1:3)^* (-1:4) (-1:3)^*\big)^*\,. 
\]
For every word $w$ in this set, it is not possible that both $w\, (-1:2)$ and $w\, (-1:4)$ are in the set, thus $\#\big\{\underline{b}{}^\alpha_{[1,n]} \mid \alpha \in [g^2, g) \setminus \Gamma\big\} \le 2\, \#\big\{\underline{b}{}^\alpha_{[1,n)} \mid \alpha \in [g^2, g) \setminus \Gamma\big\}$.
\end{proof}

Finally, combining Lemmas~\ref{l:mudifference} and~\ref{l:2n} gives the main result of this section.

\begin{proof}[\textbf{Proof of Theorem~\ref{t:g2gConst}}]
By Theorem~\ref{t:muOmega} and Lemma~\ref{l:longConstIntvl}, $\mu(\Omega_\alpha)$ is constant on every interval $\Gamma_v \subset [g^2, g]$, $v \in \mathscr{F}$. 
Therefore, we only have to consider the difference between $\mu(\Omega_\alpha)$ and $\mu(\Omega_\alpha')$ for $\alpha, \alpha' \in [g^2, \sqrt{2}-1] \setminus \Gamma$.

Let $\alpha, \alpha' \in [g^2, \sqrt{2}-1] \setminus \Gamma$ with $\alpha < \alpha'$, and fix some $n \ge 1$. For $J \ge 0$, define two sequences $(\alpha_j)_{0 \le j \le J}$, $(\alpha'_j)_{0 \le j \le J}$  in the following manner.
Set $\alpha_0 := \alpha$ and, recursively, $\alpha'_j := \max\big\{\alpha'' \in [\alpha_j, \alpha'] \setminus \Gamma \mid \underline{b}{}^{\alpha''}_{[1,n)} = \underline{b}{}^{\alpha_j}_{[1,n)}\big\}$, $\alpha_{j+1} := \min\big((\alpha'_j, \alpha'] \setminus \Gamma\big)$ if $\alpha'_j \ne \alpha'$.
The maximum exists since all sufficiently large $\alpha''$ with $\underline{b}{}^{\alpha''}_{[1,n)} = \underline{b}{}^{\alpha_j}_{[1,n)}$ lie in~$\Gamma$, thus $\alpha'_j = \zeta_v$ for some $v \in \mathscr{F}$ or $\alpha'_j = \alpha'$, and $\alpha_{j+1} = \eta_v$ if $\alpha'_j \ne \alpha'$.
Since the $\alpha_j$ are increasing, the $\underline{b}{}^{\alpha_j}_{[1,n)}$ are different for distinct~$j$, hence there exists, by Lemma~\ref{l:2n}, some $J < 2^n$ such that $\alpha'_J = \alpha'$.
By Theorem~\ref{t:muOmega} and Lemma~\ref{l:longConstIntvl}, $\mu(\Omega_{\alpha'_j})$ is equal to $\mu(\Omega_{\alpha_{j+1}})$ for $0 \le j < J$, thus 
\[
\big|\mu(\Omega_{\alpha'}) - \mu(\Omega_\alpha)\big| \le \sum_{j=0}^J \big|\mu(\Omega_{\alpha'_j}) - \mu(\Omega_{\alpha_j})\big| \le 2^n\, C_4\, n\, \big(\tfrac{1}{\sqrt{5}}\big)^n
\]
by Lemma~\ref{l:mudifference}.
Since this inequality holds for every $n \ge 1$, and $\sqrt{5} > 2$, we obtain that $\mu(\Omega_{\alpha'}) = \mu(\Omega_\alpha)$.
\end{proof}

By the discussion at the end of Section~\ref{s:relation}, the entropy decreases to the right of $[g^2, g]$ and behaves chaotically immediately to the left of $[g^2, g]$.
However, the intervals to the left of $[g^2, g]$ where the entropy decreases seem to be much smaller than those where the entropy increases.
Therefore, we conjecture that $h(T_\alpha) < h(T_{g^2})$ for all $\alpha \in (0, g^2)$.
See also the plots of the function $\alpha \mapsto h(T_\alpha)$ in~\cite{LuzziMarmi08}.

\section{Limit points} \label{sec:limit-points}

Recall that $\tau_v$ denotes the limit point of the monotonically decreasing sequence $(\zeta_{\Theta^j(v)})_{j\ge0}$.

\begin{proof}[\textbf{Proof of Theorem~\ref{t:tauTransc}}]
We argue using the transcendence results of Adamczewski and Bugeaud \cite{AdamczewskiBugeaud05}.
Let $v \in \mathscr{F}$, and $a^{(j)}_{[1,2\ell_j+1]}$ be the characteristic sequence of~$\Theta^j(v)$, $j \ge 0$.
By the proof of Lemma~\ref{l:folding}, $a^{(j+1)}_{[1,2\ell_{j+1}+1]}$ starts with $a^{(j)}_{[1,2\ell_j+1]}\, a^{(j)}_{[1,2\ell_j)}$ (if $\ell_j \ge 1$). 

This implies that $\lim_{j\to\infty} \ell_j = \infty$.
Let $a'_{[1,\infty)}$ be the infinite sequence having all sequences $a^{(j)}_{[1,2\ell_j+1]}$ as prefix (with the exception of $1$ if $v$ is the empty word).
Then $a'_{[1,\infty)}$ is the characteristic sequence of $\underline{b}{}^{\tau_v}_{[1,\infty)}$, thus $\tau_v = [0; a'_1, a'_2, \ldots\,]$ by Proposition~\ref{p:0toRCF}.

The sequence $a'_{[1,\infty)}$ is not eventually periodic because it contains, for every $j \ge 0$, $a^{(j)}_{[1,2\ell_j+1]}$ and $a^{(j)}_{[1,2\ell_j]}\, (a^{(j)}_{2\ell_j+1}-1)$ or $a^{(j)}_{[1,2\ell_j)}\, (a^{(j)}_{2\ell_j}+1)$ as factors.
If $a'_{[1,\infty)}$ were eventually periodic, then every sufficiently long factor would determine uniquely the following element of the sequence.
Therefore, $\tau_v$ is not quadratic. (This is also mentioned in \cite{Carminati-Marmi-Profeti-Tiozzo:10}.)
Since $a'_{[1,\infty)}$ starts with arbitary long ``almost squares'' and $a'_j \le a'_1$ for all $j \ge 2$, Theorem~1 of \cite{AdamczewskiBugeaud05} applies, hence $\tau_v$ is transcendental.
 \end{proof}

\begin{Rmk}\label{r:largestTau} 
From the above, the largest element of $(0,1] \setminus (\Gamma \cup Z)$ is 
\begin{align*}
\tau_v  & = [0;2, 1, 1, 2, 2, 2, 1, 1, 2, 1, 1, 2, 1, 1, 2, 2, 2, 1, 1, 2, 2, 2, 1, 1, 2, 2, 2,  \ldots\, ] \\
& = 0.3867499707143007\cdots\,,
\end{align*}
with $v$ the empty word.  
This partial quotients sequence is the fixed point of the morphism defined by $2 \mapsto 211$, $1 \mapsto 2$.
It is known to be the smallest aperiodic sequence in $\{1,2\}^\omega$ with the property that all its proper suffixes are smaller than itself with respect to the alternate order, and appears therefore in several other contexts;  see \cite{Dubickas07,LiaoSteiner}.
Note that all elements of $Z = \{\zeta_v \mid v \in \mathscr{F}\}$ have (purely) periodic RCF expansion.
\end{Rmk}

\begin{Prop}\label{p:gSqrdLimPt}    
The point $g^2$ is a two sided limit of the set $\{\tau_v\,|\, v \in \mathscr{F}\}$.
\end{Prop}

\begin{proof}  
Let $v = (-1:2)\, (-1:3)^{\ell}$, then its characteristic sequence is $2\,1^{2\ell}$, thus $v$ belongs to~$\mathscr{F} \setminus \Theta(\mathscr{F})$. 
For increasing~$\ell$, $\zeta_v$~tends to $g^2$ from above, and the same clearly holds for~$\tau_v$.
Similarly, the $\tau_v$ corresponding to $v = (-1:2)\, (-1:3)^{\ell}\, (-1:2)$ (see Example~\ref{ex:infFamInF}) tend to $g^2$ from below.
\end{proof}

\section{Open questions}  
As usual,  we find that we now have more questions than when we began our project.   
We list a few, in the form of problems.  

\begin{itemize} 
\item  Prove that $h(T_\alpha)$ is maximal on $[g^2,g]$.

\item Determine explicit values for $h(T_\alpha)$ when $\alpha < g^2$ and for the invariant density $\nu_{\alpha}$ when $\alpha < \sqrt{2} - 1$.
 
\item Prove that  $\nu_{\alpha}$ is always of the form $A/(x+B)$ as \cite{Carminati-Marmi-Profeti-Tiozzo:10} conjecture.
 
\item (From H.~ Nakada) Determine all $\alpha$ such that $h(T_\alpha) = h(T_1)$.

\item In general,  determine  the sets of $\alpha$  with equal entropy.

\item Determine the sets of all $\alpha$  giving isomorphic  dynamical systems.     
 
\item Generalize our approach to use with other continued fractions, such as the $\alpha$-Rosen fractions considered in \cite{Dajani-Kraaikamp-Steiner:09}.  
\end{itemize}

We note that Arnoux and the second named author have work in progress that  responds to  a question
from \cite{LuzziMarmi08} that we had included in an earlier version of our open problems list: 
 Each $T_{\alpha}$ arises as a cross-section of the geodesic flow on the unit tangent bundle of the modular surface.    This result is shown to be equivalent to Theorem~\ref{t:natext}. 
  
\section*{Acknowledgments}
\noindent
We thank Pierre Arnoux and Hitoshi Nakada for comments on an initial version of this~work.

\bibliographystyle{amsalpha}
\bibliography{alphafrac}

\providecommand{\bysame}{\leavevmode\hbox to3em{\hrulefill}\thinspace}
\providecommand{\MR}{\relax\ifhmode\unskip\space\fi MR }
\providecommand{\MRhref}[2]{%
  \href{http://www.ams.org/mathscinet-getitem?mr=#1}{#2}
}
\providecommand{\href}[2]{#2}
\begin{thebibliography}{CMPT10}

\bibitem[AB05]{AdamczewskiBugeaud05}
B.~Adamczewski and Y.~Bugeaud, \emph{On the complexity of algebraic numbers.
  {II}. {C}ontinued fractions}, Acta Math. \textbf{195} (2005), 1--20.

\bibitem[Abr59]{Abramov59}
L.~M. Abramov, \emph{The entropy of a derived automorphism}, Dokl. Akad. Nauk
  SSSR \textbf{128} (1959), 647--650, English translation: Amer. Math. Soc.
  Transl. Ser. 2 {\bf 49} (1966), 162--166.

\bibitem[BCIT]{Bonanno-Carminati-Isola-Tiozzo}
C.~Bonanno, C.~Carminati, S.~Isola, and G.~Tiozzo, \emph{Dynamics of continued
  fractions and kneading sequences of unimodal maps}, Discrete Contin. Dyn.
  Syst., to appear.

\bibitem[CMPT10]{Carminati-Marmi-Profeti-Tiozzo:10}
C.~Carminati, S.~Marmi, A.~Profeti, and G.~Tiozzo, \emph{The entropy of
  $\alpha$-continued fractions: numerical results}, Nonlinearity \textbf{23}
  (2010), no.~10, 2429--2456.

\bibitem[CT]{Carminati-Tiozzo:10}
C.~Carminati and G.~Tiozzo, \emph{A canonical thickening of $\mathbb{Q}$ and
  the dynamics of continued fraction transformations}, Ergodic Theory Dynam.
  Systems, to appear.

\bibitem[DKS09]{Dajani-Kraaikamp-Steiner:09}
K.~Dajani, C.~Kraaikamp, and W.~Steiner, \emph{Metrical theory for
  $\alpha$-{R}osen fractions}, J. Eur. Math. Soc. (JEMS) \textbf{11} (2009),
  no.~6, 1259--1283.

\bibitem[Dub07]{Dubickas07}
A.~Dubickas, \emph{On a sequence related to that of {T}hue-{M}orse and its
  applications}, Discrete Math. \textbf{307} (2007), no.~9-10, 1082--1093.

\bibitem[Kea95]{Keane:95}
M.~Keane, \emph{A continued fraction titbit}, Fractals \textbf{3} (1995),
  641--650.

\bibitem[Kra91]{Kraaikamp:91}
C.~Kraaikamp, \emph{A new class of continued fraction expansions}, Acta Arith.
  \textbf{57} (1991), 1--39.

\bibitem[KS12]{Kalle-Steiner:12}
C.~Kalle and W.~Steiner, \emph{Beta-expansions, natural extensions and multiple
  tilings associated with {P}isot units}, Trans. Amer. Math. Soc. \textbf{364}
  (2012), no.~5, 2281--2318.

\bibitem[KSS10]{Kraaikamp-Schmidt-Smeets:10}
C.~Kraaikamp, T.A. Schmidt, and I.~Smeets, \emph{Natural extensions for
  $\alpha$-{R}osen continued fractions}, J. Math. Soc. Japan \textbf{62}
  (2010), 649--671.

\bibitem[LM08]{LuzziMarmi08}
L.~Luzzi and S.~Marmi, \emph{On the entropy of {J}apanese continued fractions},
  Discrete Contin. Dyn. Syst. \textbf{20} (2008), no.~3, 673--711.

\bibitem[LS]{LiaoSteiner}
L.~Liao and W.~Steiner, \emph{Dynamical properties of the negative beta
  transformation}, Ergodic Theory Dynam. Systems, to appear.

\bibitem[MCM99]{MoussaCassaMarmi:99}
P.~Moussa, A.~Cassa, and S.~Marmi, \emph{Continued fractions and {B}rjuno
  functions}, J. Comput. Appl. Math. \textbf{105} (1999), no.~1-2, 403--415,
  Continued fractions and geometric function theory (CONFUN) (Trondheim, 1997).

\bibitem[Nak81]{Nakada81}
H.~Nakada, \emph{Metrical theory for a class of continued fraction
  transformations and their natural extensions}, Tokyo J. Math. \textbf{4}
  (1981), no.~2, 399--426.

\bibitem[NIT77]{NakadaItoTanaka77}
H.~Nakada, S.~Ito, and S.~Tanaka, \emph{On the invariant measure for the
  transformations associated with some real continued-fractions}, Keio Engrg.
  Rep. \textbf{30} (1977), no.~13, 159--175.

\bibitem[NN02]{NakadaNatsui02}
H.~Nakada and R.~Natsui, \emph{Some metric properties of $\alpha$-continued
  fractions}, J. Number Theory \textbf{97} (2002), no.~2, 287--300.

\bibitem[NN08]{NakadaNatsui08}
\bysame, \emph{The non-monotonicity of the entropy of {$\alpha$}-continued
  fraction transformations}, Nonlinearity \textbf{21} (2008), no.~6,
  1207--1225.

\bibitem[Roh61]{Rohlin61}
V.~A. Rohlin, \emph{Exact endomorphisms of a {L}ebesgue space}, Izv. Akad. Nauk
  SSSR Ser. Mat. \textbf{25} (1961), 499--530, English translation: Amer. Math.
  Soc. Transl. Ser. 2 {\bf 39} (1964), 1--36.

\bibitem[Tio]{Tiozzo:09}
G.~Tiozzo, \emph{The entropy of $\alpha$-continued fractions: analytical
  results}, arXiv:0912.2379v1.

\bibitem[Zag81]{Zagier:81}
D.~Zagier, \emph{{Z}etafunktionen und quadratische {K}\"orper}, Springer,
  Berlin, 1981.

\end{thebibliography}
\end{document}